\documentclass{amsart}   
\usepackage[dvipsnames, table, xcdraw]{xcolor}
\usepackage{array, graphicx, amsfonts, amsmath, amssymb, amsthm, mathrsfs, enumitem, parskip, tikz, quiver, xspace, ulem, float, comment, thmtools, blkarray} 
\usepackage{hyperref}

\usepackage[margin=1in]{geometry}

\let\emph\relax 
\DeclareTextFontCommand{\emph}{\bfseries\itshape\boldmath}

\declaretheorem[style=plain, numberwithin=section, name=Theorem]{thm}
\declaretheorem[style=plain, sibling=thm, name=Corollary]{cor}
\declaretheorem[style=plain, sibling=thm, name=Proposition]{prop}
\declaretheorem[style=plain, sibling=thm, name=Facts]{facts}
\declaretheorem[style=plain, sibling=thm, name=Fact]{fact}
\declaretheorem[style=plain, sibling=thm, name=Lemma]{lemma}
\declaretheorem[style=remark, sibling=thm, name=Discussion]{discussion}
\declaretheorem[style=remark, sibling=thm, name=Remark]{rmk}
\declaretheorem[style=remark, sibling=thm, name=Remark]{remark}
\declaretheorem[style=definition, sibling=thm, name=Definition]{defn}
\declaretheorem[style=definition, sibling=thm, name=Example]{examp}
\declaretheorem[style=definition, sibling=thm, name=Notation]{notation}
\declaretheorem[style=definition, sibling=thm, name=Notation and Discussion]{notationdiscussion}
\newtheorem*{thm4.21*}{Theorem~\ref{Main4}}
\newtheorem*{cor3.9*}{Corollary~\ref{cor-hochsform-tSR-koszulver}}
\newtheorem*{thm5.18*}{Theorem~\ref{thm-tSR-terai}}

\newcommand{\new}{\newcommand}
\new{\und}{\underline}
\new{\comm}{\iftrue}  
\new{\bem}[1]{\bibitem[#1]{#1}} 
\new{\rbx}[2]{ \raisebox{#1}{\mbox{#2}}}
\new{\Strut}[1]{\rule{0.0pt}{#1}}

\new{\spx}[3]{\{v_{#1}, \, v_{#2}, \, v_{#3}\}}
\new{\emp}[1]{{\it #1}}
\new{\ov}{\overline}
\new{\vect}[2]{{#1}_1,\, \ldots, \, {#1}_{#2}}
\new{\ben}{\begin{enumerate}}      
\new{\bena}{\ben[label=(\alph*)]}     
\new{\beni}{\ben[label=(\roman*)]}    
\new{\bend}{\ben[label={\arabic*}.]}  
\new{\benn}{\ben[label=(\arabic*)]}   
\new{\benr}{\ben[resume*]}            
\new{\bensp}[1]{\ben[label=#1]}       
\new{\een}{\end{enumerate}}
\new{\bs}{$\backslash$}
\new{\bu}{{\bullet}}
\new{\Ga}{G_{\ua}}
\newcommand{\Znplus}[1]{({#1}_0, \underline{#1})}
\newcommand{\Zn}[1]{\underline{#1}}
\newcommand{\varsnplus}{x_0, \underline{x}}
\newcommand{\varsn}{\underline{x}}
\newcommand{\seq}[1]{\underline{#1}}
\new{\busm}{\,\raisebox{.65pt}{\mbox{\scalebox{.7}{$\bullet$}}}\,} 
\new{\nl}{\scalebox{.9}[1.2]{$\varnothing$}}
\new{\Nl}{\nl \kern -6.5 pt \nl}
\new{\wt}{\widetilde}
\new{\mI}{(-1,\ldots,-1)}
\newcommand{\monom}[1]{x_0^{{#1}_0}\underline{x}^{\underline{#1}}}

\newcommand{\tmonom}[1]{t^{{#1}_0}\underline{x}^{\underline{#1}}}

\newcommand{\tmonomlong}[1]{t^{{#1}_0}x_1^{{#1}_1}\hdots x_n^{{#1}_n}}
\newcommand{\monompos}[1]{\underline{x}^{\underline{#1}}}
\newcommand{\monomposlong}[1]{x_1^{{#1}_1}\hdots x_n^{{#1}_n}}

\newcommand{\SRrcomp}[1]{V[#1]}
\newcommand{\tSRrcomp}[1]{V\overline{[#1]}}

\newcommand{\SRi}[1]{I_{#1}}
\newcommand{\tSRi}[1]{{\overline{\SRi{#1}}}}
\new{\RV}{V\ov{[\Delta]}}

\newcommand{\monomify}[1]{#1'}

\newcommand{\monomF}[1]{\varsn^{#1}}
\new{\Fx}{\monomF{F}}
\newcommand{\zmonomF}[1]{(\varsnplus)^{#1}}
\newcommand{\tmonomF}[1]{(t, \varsn)^{#1}}

\newcommand{\vertset}{\{v_1, \hdots, v_n\}}
\newcommand{\zvertset}{\{v_0, v_1, \hdots, v_n\}}
\newcommand{\vertex}[1]{v_{#1}}

\new{\ang}[1]{\langle #1 \rangle}
\new{\oriented}[1]{\langle #1 \rangle}
\new{\vertlist}[2]{\vertex{#1}, \hdots, \vertex{#2}}

\newcommand{\del}{\partial}

\new{\bnd}{\del}
\new{\cobnd}{\delta}
\new{\bndsgn}[2]{[#1:#2]}
\newcommand{\coord}[3]{#1\bndsgn{#2}{#3}}

\new{\La}{{\Lambda}}
\new{\la}{{\lambda}}

\new{\gr}{\op{gr}}
\new{\link}{\op{link}}
\new{\st}{\overline{\op{St}}}

\newcommand{\Z}{\mathbb{Z}}
\newcommand{\Q}{\mathbb{Q}}
\newcommand{\R}{\mathbb{R}}

\newcommand{\N}{\mathbb{N}}

\newcommand{\C}{\mathbb{C}}
\newcommand{\Prj}{\mathbb{P}}
\new{\sm}{\,\mbox{\raisebox{1.2pt}{$\smallsetminus$}}\,}
\new{\Zo}{\und{\mathbf{0}}}

\new{\fA}{\mathfrak{A}}
\new{\fB}{\mathfrak{B}}
\new{\fC}{\mathfrak{C}}
\new{\fF}{\mathfrak{F}}
\new{\fk}{\mathfrak{k}}
\newcommand{\m}{\mathfrak{m}}
\newcommand{\n}{\mathfrak{n}}
\newcommand{\p}{\mathfrak{p}}
\newcommand{\q}{\mathfrak{q}}

\new{\cK}{\mathcal{K}}
\new{\cP}{\mathcal{P}}
\newcommand{\cC}{\mathcal{C}}
\new{\cS}{\mathcal{S}}

\new{\sF}{\scr{F}}
\new{\sG}{\scr{G}}

\new{\ua}{{\und{a}}}
\new{\uf}{{\und{f}}}
\new{\ug}{{\und{g}}}
\new{\umm}{{\und{m}}}
\new{\ux}{{\und{x}}}
\new{\uy}{{\und{y}}}
\new{\uz}{{\und{z}}}

\new{\tC}{\widetilde{C}}
\new{\tH}{\widetilde{\op{H}}}

\newcommand{\scr}[1]{\mathscr{#1}}
\newcommand{\calig}[1]{\mathcal{#1}}

\new{\op}{\operatorname}
\new{\buc}{^{\bullet}}
\new{\ch}[1]{\op{char}(#1)}
\new{\Homol}{\op{H}}
\new{\HH}{\op{H}}
\newcommand{\simpHomol}{\widetilde{\Homol}}
\newcommand{\simpchain}{\widetilde{\mathrm{C}}}

\new{\ba}[1]{[#1]_{\ua}} 
\new{\bo}[1]{[#1]_{\Zo}} 
\new{\mapcone}[1]{\op{Cone}(#1)}
\new{\Hm}{H_{\m}}
\new{\Hom}{\op{Hom}}
\new{\Ann}{\op{Ann}}
\new{\Ass}{\op{Ass}}
\new{\Coker}{\op{Coker}}
\new{\Ker}{\op{Ker}}
\new{\Span}{\op{Span}}
\newcommand{\Kos}{\mathcal{K}}
\new{\Tor}{\op{Tor}}
\new{\Ext}{\op{Ext}}

\new{\mplus}{\scalebox{1.3}{$\oplus$}}
\new{\Spec}{\op{Spec\,}}
\new{\depth}{\op{depth}}
\new{\height}{\op{ht}}
\new{\Hx}[2]{H^{#1}_{(\ux)}(#2)}
\new{\charac}{\op{char}}
\new{\rank}{\op{rank}}
\newcommand{\Serre}[1]{$(\mathrm{S}_{#1})$}
\new{\supp}{\op{supp}}
\new{\sign}{\op{sgn}}
\new{\SR}{Stanley-Reisner\xsp}
\new{\tSR}{$t$-\SR}
\new{\VD}{V\ov{[\Delta]}}
\new\arwf[1]{\xrightarrow{#1}}
\new{\imp}{\Rightarrow}
\new{\inc}{\subseteq}
\newcommand{\inj}{\hookrightarrow}
\newcommand{\surj}{\twoheadrightarrow}

\new{\then}{\Rightarrow}
\new{\Cv}{\check C}

\new{\xsp}{\xspace}
\new{\fg}{finitely generated\xsp}
\new{\LetV}{Let $(V,t)$  be a discrete valuation ring\xsp}
\new{\LetVDelta}{\LetV, and let $\Delta$ be a simplicial complex on $n+1$ vertices $\zvertset$\xsp}
\new{\nzd}{nonzerodivisor\xsp}
\new{\CM}{Cohen-Macaulay\xsp}
\new{\mlt}{multihomogeneous\xsp}
\new{\mltd}{multidegree\xsp}
\new{\mltg}{multigraded\xsp}
\new{\resp}{respectively\xsp}
\new{\sqf}{square-free\xsp}
\new{\sop}{system of parameters\xsp}
\new{\Cech}{\v Cech\xsp}

\definecolor{gre}{rgb}{0.0, .42, 0.0} 
\definecolor{pnk}{rgb}{1 , .35 , .35}  
\definecolor{gld}{rgb}{1 , .65 , 0}    
\definecolor{ltblue}{rgb}{.565,.835, 1} 
\definecolor{ltvi}{rgb}{.96, .93, .97}  
\definecolor{purple}{rgb}{.42, .1, .55} 
\definecolor{vigr}{rgb}{.96, .96, 1}    
\definecolor{vlb}{rgb}{.8,.8,.999}    

\new\blue[1]{{\color{blue} #1}}  
\new\cred[1]{{\color{red} #1}}  
\new\gre[1]{{\color{gre} #1}}
\new\mgn[1]{{\color{magenta} #1}}      
\new\orange[1]{{\color{orange} #1}}   
\new\pur[1]{{\color{purple} #1}}
\new\violet[1]{{\color{violet} #1}}

\begin{document}
\title{Stanley-Reisner Theory in Mixed Characteristic}
\author{Mel Hochster and Olivia Strahan}

\thanks{Hochster and Strahan were partially supported by National Science Foundation grants 
DMS--2200501 and DMS--2101075.}
\subjclass[2020]{Primary. 13F55, 13D45, 13A02, 13P20, 13H10.}

\keywords{Cellular sheaf cohomology, Cohen-Macaulay rings, Koszul cohomology, local cohomology, 
Serre conditions, simplicial cohomology, simplicial complex, Stanley-Reisner ring, $t$-Stanley-Reisner ring.}

\begin{abstract}
    We discuss a class of mixed characteristic rings defined analogously to Stanley-Reisner rings by replacing one variable with a uniformizing parameter for a discrete valuation ring. We adapt Hochster's formula for Tor and Ext modules, Hochster's formula for local cohomology, and Terai's criterion for Serre conditions to this new setting; one of the tools needed is cellular sheaf cohomology, for which we give a brief treatment.
\end{abstract}

\maketitle

\bigskip

\section{Introduction} 

A classic paper of Reisner  \cite{Rei76} studying what are now called {\it Stanley-Reisner rings} established the value of using simplicial cohomology and other tools from algebraic topology to study rings that are quotients of polynomial rings by ideals generated by square-free monomials, a theory with many  applications that has generated an immense literature (e.g., \cite{Sta96}, \cite{cca}, \cite{AdSa16}, \cite{hochform}, and their references).
In \cite{qseq} and \cite{thesis},
the second author initiated the study of various rings defined combinatorially, including 
Stanley-Reisner rings, in which the uniformizing parameter $t$ of a discrete valuation ring $(V,\,tV)$ plays the role of one of the variables. Such constructions are especially interesting in the cases where $V$, like the $p$-adic integers, has mixed characteristic. This paper continues with a deeper study of the properties of these $t$-Stanley-Reisner rings.   

Throughout the paper, $(V,\,tV)$ is a Noetherian discrete valuation domain, although we follow the frequent practice of referring to such a ring simply as a discrete valuation ring or DVR.  The residue field $V/tV$ and the fraction field of $V$ will be denoted $K$ and $L$, respectively.  If $V$ has mixed characteristic, we use $p$ to denote the characteristic of $K$. 

The simplicial complexes we consider are finite. Given a simplicial complex $\Delta$ with a 
distinguished vertex $\vertex{0}$, we define the associated 
$t$-Stanley-Reisner ring $\tSRrcomp{\Delta}$ to be the image of the Stanley-Reisner ring $V[\Delta]$ (see Definition~\ref{defn-SRnotation}) under the evaluation map sending $x_0$ to $t$. In the mixed characteristic case, we will see that $\tSRrcomp{\Delta}$ displays uniquely mixed characteristic behavior whenever $\Delta$ or certain subcomplexes of $\Delta$ have $p$-torsion in their cohomology with coefficients in $\Z$. In the computation of a given algebraic invariant of $\tSRrcomp{\Delta}$, there will be some pieces of the problem that depend on $L$ and thus behave generically. Other pieces of the problem will depend on $K$, and these will also behave generically \textit{except} when the characteristic of $K$ is one of finitely many special primes. Mixed characteristic behavior occurs when some parts behave generically and others behave exceptionally.

In \S\ref{sec-bg}, we cover the essentials of classical Stanley-Reisner theory before defining  $t$-Stanley-Reisner rings and outlining their basic properties. We adapt Hochster's formula for Koszul homology \cite{hochform} in \S\ref{sec-koszul}, followed by Hochster's formula for local cohomology \cite[Thm. 5.3.8]{brunsherz} in \S\ref{sec-local-cohomology}. Finally, in \S\ref{sec-serre} we adapt Terai's theorem for Serre conditions \cite{terai}. The main results of the paper are in \S\S3--5 and include Corollary~\ref{cor-hochsform-tSR-koszulver}, Theorem~\ref{Main4}, and Theorem~\ref{thm-tSR-terai}.

\begin{cor3.9*}
     \LetVDelta. Considering the residue field $K$ as a $V[\varsn]$-module via $K \cong V[\varsn]/(t, \varsn)$, we have isomorphisms of $K$-vector spaces
    \begin{align*}
        \Tor_i^{V[\varsn]}(K, \tSRrcomp{\Delta}) &\cong \frac{\Tor_{i}^{V[\varsnplus]}(V, \SRrcomp{\Delta})}{t \cdot \Tor_{i}^{V[\varsnplus]}(V, \SRrcomp{\Delta})}  \; \oplus\;  \Ann_{\Tor_{i-1}^{V[\varsnplus]}(V, \SRrcomp{\Delta})}t\\
        \\
        \Ext_{V[\varsn]}^j(K, \tSRrcomp{\Delta}) &\cong \frac{\Ext^{j}_{V[\varsnplus]}(V, \SRrcomp{\Delta})}{t \cdot \Ext^{j}_{V[\varsnplus]}(V, \SRrcomp{\Delta})}  \; \oplus\;\Ann_{\Ext^{j+1}_{V[\varsnplus]}(V, \SRrcomp{\Delta})}t .
    \end{align*}  
{\emergencystretch=2em
    Combining this with Hochster's formula (Thm.~\ref{thm-og-hoch-koszulver}), we see that the $K$-vector space dimensions of $\Tor_i^{V[\varsn]}(K, \tSRrcomp{\Delta})$ and $\Ext_{V[\varsn]}^j(K, \tSRrcomp{\Delta})$ can be computed from the simplicial cohomology modules of deletions $\simpHomol^q(\Delta - T; V)$.\par}

\end{cor3.9*} 

\begin{thm4.21*}
Let $(V,tV)$ be a DVR with fraction field $L$ and residue field $K$, and let $\Delta$ be a finite simplicial complex on $\{v_0,v_1,\ldots,v_n\}$. Put
\[
 R=\RV,\qquad\m=(t,x_1,\ldots,x_n)R.
\]
For $\ua=(a_1,\ldots,a_n)\in\Z^n$, set
\[
 G_\ua=\{v_i:a_i<0\},\qquad \La_\ua=\link_\Delta G_\ua.
\]
On $\La_\ua$, let $\sF_\ua$ be the cellular sheaf with values
\[
 \sF_\ua(F)=
 \begin{cases}
 L,&v_0\in F,\\
 V,&F\in\link_{\La_\ua}\{v_0\},\\
 K,&F\notin\st_{\La_\ua}\{v_0\},
 \end{cases}
\]
and with restriction maps the identities within each type, $V\hookrightarrow L$, and $V\twoheadrightarrow K$. Then
\[
 [H^i_\m(R)]_\ua\cong
 \begin{cases}
 \tH^{i-|G_\ua|-1}(\La_\ua;\sF_\ua),&\ua\leq\Zo,\\
 0,&\text{some }a_i>0.
 \end{cases}
\]
If $G_\ua\notin\Delta$, the right-hand side is understood to be zero. In particular, writing $\sF_G$ for the corresponding sheaf on $\link_\Delta G$,
\[
 H^i_\m(R)=0
 \quad\Longleftrightarrow\quad
 \tH^{i-|G|-1}(\link_\Delta G;\sF_G)=0
 \quad\text{for every }G\in\Delta-v_0.
\]
\end{thm4.21*}

\begin{thm5.18*}
Assume that $\Delta$ is a finite abstract simplicial complex of dimension $d$,
and that $k \geq 2$. Then
$R:=V\ov{[\Delta]}$ satisfies \Serre{k} if and only if $\Delta$ has pure dimension and
all of the following 
conditions hold for all faces $F$ of $\Delta$, including the empty face $\nl$.
\benn
\item For every face $F\in\Delta$ with $v_0\notin F$,   
$\tH^j(\link_{\Delta}F;\,K)=0$
for all $j < \min\{k-1,d-\dim F-1\}$.  

\item For every face $F \in \Delta$ with $v_0\in F$,
   \beni
    \item $\,\tH^j(\link_{\Delta}F;\,L)=0 \text{\ for all\ } j<\min\{k-1,d-\dim F - 1\} \text{\ and}$
     \item $\tH^j(\link_{\Delta}F;\, K)=0\text{\ for all\ } j< \min\{k-2,\,d-\dim F - 1\}$.
  \een
\een
Equivalently, since $\Delta$ is pure and thus $d - \dim F - 1 = \dim\link_{\Delta} F$,
the second entries in the $\min$ in both conditions may be 
replaced by $\dim\link_{\Delta} F$.
\end{thm5.18*}

As a case study, we apply each of the main theorems to the $t$-Stanley-Reisner ring defined by a triangulation of the real projective plane,  which has $2$-torsion in its simplicial homology with coefficients in $\Z$.

In our adaptation of Hochster's formula for local cohomology, we need to consider homology and cohomology with coefficients that vary with the simplex.  The terms of the graded pieces of the \v Cech complex of a $t$-Stanley-Reisner ring are, in general, not free $V$-modules, but rather direct sums of copies of $K$, $L$, and $V$.  In the more usual version of simplicial cohomology, one may think of cochains as functions on oriented simplices of a  given dimension with values in a fixed ring or module.  Here, we need a version in which the values on a given oriented simplex are determined functorially, in a certain precise sense,  from the choice of that oriented simplex.   In the simplicial case, this 
turns out to be what is sometimes called {\it cellular sheaf cohomology}, particularly in a number of papers written with an eye towards various applications \cite{shepard}, \cite{curry}. This terminology is used because the cohomology modules obtained are the same as the sheaf cohomology of a sheaf, in the usual sense, on a finite topological space constructed from the simplicial complex.
Note that we use an augmented version, analogous to reduced singular cohomology, with a cochain complex that has a term in degree $-1$. 

Throughout this paper, $\N \inc \Z \inc \Q \inc \R \inc \C$ denote the nonnegative integers, the integers, the rational numbers, the real numbers, and the complex numbers, \resp. All given rings are assumed to be commutative with identity, and local rings are assumed to be Noetherian.

\section{Background} \label{sec-bg}
\subsection{Simplicial complexes}  

  Simplicial complexes are central objects of study in both combinatorics and topology. 
  In this section we give a brief overview of the most relevant terminology and facts.

\begin{defn}[{\cite{Mun84a, cca}}]
    \label{def-simplicialcomplex} Let $V = \vertset$ be a collection of vertices. An \emph{$i$-dimensional simplex} is a subset $F \subseteq V$ of cardinality $i + 1$. A \emph{simplicial complex} is a collection of simplices $\Delta \subseteq \mathcal{P}(V)$ with the property that if $F \in \Delta$, then $G \in \Delta$ for every subset $G \subseteq F$. The simplices in $\Delta$ are called \emph{faces}, and a face which is maximal with respect to inclusion is called a \emph{facet}. An  \emph{\boldmath $i$-face} of $\Delta$ is an $i$-dimensional simplex $F \in \Delta$ . The \emph{dimension} of $\Delta$ is the maximum dimension of its faces. 
    
\end{defn}

\begin{rmk}\label{emptyset}
    By this definition, the empty set $\nl$ is a $(-1)$-face of any nonempty simplicial complex.
\end{rmk}

\begin{defn}\label{def-link-star}
    Let $\Delta$ be a simplicial complex on vertex set $V$, and let $F\subseteq V$. The  \emph{closed star},
 also called simply the \emph{star} in the combinatorial commutative algebra literature,  
    and the \emph{link} of $F$ in $\Delta$ are, respectively, the simplicial complexes
    \begin{align*}
        \st_\Delta F &:= \{G \in \Delta : F \cup G \in \Delta\},\\
        \link_\Delta F &:= \{G \in \Delta : F \cap G = \nl,\;  F \cup G \in \Delta  \}.  
    \end{align*}
    Both $\st_\Delta F$ and $\link_\Delta F$ are the empty simplicial complex $\Nl$ 
    unless $F$ is a face of $\Delta$. \smallskip
    
If $T \inc V$, the \emph{deletion} of $T$ is the subcomplex
$\Delta-T:=\{G\in\Delta:G\cap T=\nl\}$.  We write $\Delta-v$ for $\Delta-\{v\}$.
It is important to note that $\Delta-T$ is very different from $\Delta \sm T$. 
    \end{defn}

\begin{rmk}\label{0} It is very 
    important in combinatorial commutative algebra to distinguish between the empty simplicial complex,
    which we denote  $\Nl$, which has no faces, and the simplicial complex $\{\nl\}$ whose unique face is the  empty set, $\nl$. Note that the reduced simplicial cohomology or 
    homology of the former with coefficients in
    $A$ is 0 in every degree, while the reduced simplicial cohomology or homology of the latter is nonzero
    precisely in degree $-1$, where it is isomorphic to $A$.
\end{rmk}

\begin{defn} \label{def-simpchain}
    Let $\Delta$ be a simplicial complex on vertices $\vertset$. An \emph{oriented $q$-face} of $\Delta$ is an equivalence class $\oriented{\vertlist{i_0}{i_q}}$ of signed ordered lists of distinct vertices $\vertex{i_j}$ forming a $q$-face of $\Delta$, under the equivalence relation generated by requiring the equivalence of $\oriented{\vertlist{i_0}{i_q}}$ and 
    $\sign(\sigma) \oriented{\vertlist{i_{\sigma(0)}}{i_{\sigma(q)}}}$ for any permutation $\sigma$ of $\{0, \hdots, q\}$. For a $q$-face $F \in \Delta$, let $\oriented{F}$ be the equivalence class of the vertices of $F$ listed in increasing order.
    
    For an oriented $(q-1)$-face $\sigma$ and an oriented $q$-face $\tau$, define
    \[
      \bndsgn{\sigma}{\tau} := \begin{cases}
          0 &\text{if $\sigma$ is not a face of $\tau$,}\\
          1 &\text{if $\sigma \subset \tau$ and the orientation of $\tau$ induces that of $\sigma$,}\\
          -1 &\text{if $\sigma \subset \tau$ and the orientation of $\tau$ induces the opposite orientation of $\sigma$}.
      \end{cases}
    \]
    More explicitly, for $\tau = \oriented{\vertlist{i_0}{i_{q}}}$ and $\sigma = \oriented{\vertex{i_0}, \hdots, \widehat{\vertex{i_j}}, \hdots, \vertex{i_q}}$, we have $\bndsgn{\sigma}{\tau} = (-1)^j$.

    For a commutative ring $A$, define the \emph{(reduced) simplicial chain complex} of $\Delta$ over $A$ to be the complex of free $A$-modules $\simpchain_q(\Delta; A)$ with basis $\{\oriented{F} : F \in \Delta, \; \dim(F) = q\}$, where the boundary map acts on oriented $q$-faces $\oriented{F}$ by 
    \[
        \bnd (\oriented{F}) = \sum_{\dim G = q-1} \bndsgn{\oriented{G}}{\oriented{F}} \oriented{G}
    \]
    The homology of this complex is denoted by $\simpHomol_\bullet(\Delta; A)$. 

    The \emph{(reduced) simplicial cochain complex} is $\simpchain^\bullet(\Delta; A) := \Hom_A(\simpchain_\bullet(\Delta; A), A)$, with cohomology denoted by $\simpHomol^\bullet(\Delta; A)$.

\end{defn}

\begin{rmk}
    One need not fix an ordering of the vertices to define simplicial homology. We do so here because the vertex ordering provides the small convenience of a default orientation for each face of $\Delta$.
\end{rmk}

\begin{rmk}
    The \textit{non-reduced} simplicial (co)homology is obtained by omitting the degree $-1$ term from the simplicial (co)chain complex. The reduced and non-reduced (co)homology modules are identical except possibly in 
    degrees $0$ and $-1$.
\end{rmk}

\begin{defn} \label{def-geometric-realization}
    Let $e_i \in \R^n$ be the point whose $i^{th}$ coordinate is $1$, with all other coordinates vanishing. For a $q$-simplex $F = \{\vertlist{i_0}{i_q}\} \inc\vertset$, define $|F| \subset \R^n$ to be the closed subset 
    $\{\lambda_0 e_{i_0} + \hdots + \lambda_q e_{i_q} \; : \; 0 \leq \lambda_i \leq 1, \; \lambda_0 + \hdots + \lambda_q = 1 \}$. 
    
    Let $\Delta$ be a simplicial complex on vertices $\vertset$. The \emph{geometric realization} of $\Delta$ is the topological space $|\Delta| = \bigcup|F| \subset \R^n$, where the union is taken over all faces $F \in \Delta$.
\end{defn}

\begin{discussion}\label{simp-vs-sing} Homotopic continuous maps of the underlying spaces induce the same map of simplicial (co)homology with specified coefficients.  This can be proved by the simplicial approximation theorem \cite[Ch.~3,$\,$\S\S3--4]{Spa95}. 
It also follows from the corresponding result for singular (co)homology, and the fact that the singular and simplicial theories agree on non-trivial finite simplicial complexes. See \cite[4.6 Thm.~8]{Spa95}. The trivial simplicial complex $\{\nl\}$ and the empty simplicial complex $\Nl$ are indistinguishable as topological spaces, so one must be slightly careful about the statement in homological degree $-1$.

Once we know that the invariants we are interested in are given by simplicial (co)homology, we may use a vast array of topological techniques to compute them, e.g., CW-complex methods \cite[Ch.\,7,\,\S6]{Spa95}, \cite[\S1]{For02} and Morse theory \cite{milnor}, which exists in a discrete combinatorial version \cite{For98, For02}. \end{discussion}

\subsection{Stanley-Reisner theory} 

Stanley-Reisner rings are an important class of rings in bijective correspondence with simplicial complexes; the relationships between algebraic, topological, and combinatorial properties in this correspondence have been studied extensively in the equal characteristic setting. Classically, Stanley-Reisner rings are defined over fields, but in this section we will define and state properties of Stanley-Reisner rings over an arbitrary Noetherian commutative ring.

\begin{defn} \label{def-SRring}
    Let $\varsn = x_1, \hdots, x_n$ be variables, and let $A$ be a commutative ring. A monomial $\monompos{m} = \monomposlong{m}$ is called \emph{\sqf} if $m_i = 0$ or $1$ for all $i$. An ideal $I \subseteq A[\varsn]$ is called a \emph{\sqf monomial ideal} if it can be generated by \sqf monomials. A \emph{Stanley-Reisner ring} is the quotient $A[\varsn]/I$ by a \sqf monomial ideal $I$. 
\end{defn}

\begin{rmk}
    A monomial ideal $I \subseteq A[\varsn]$ has a unique minimal set of monomial generators (\cite[Lemma 1.2] {cca}), so $I$ is \sqf if and only if all of its minimal monomial generators are \sqf.
\end{rmk}

\begin{defn}\label{defn-SRnotation}
    We define the following notations:
    \benn
        \item The \emph{support} of a monomial $\monompos{m} = \monomposlong{m}$ is the simplex
        $\supp \monompos{m} := \{\vertex{i} \; : \; m_i \neq 0\}$.  
         \item A simplex $F \subseteq \vertset$ defines a \sqf monomial 
         $\monomF{F} := \prod_{\vertex{i}\in F}x_i$.
        The \sqf monomials in $A[x_1, \hdots, x_n]$ are in bijection with simplices on $\vertset$. Note that $\monomF{\nl} = 1$.

        \item If $\Delta$ is a simplicial complex on $\vertset$, define the \emph{Stanley-Reisner ideal} of $\Delta$ to be the \sqf monomial ideal $\SRi{\Delta} := (\monomF{F}: F \notin \Delta)$ of $A[\varsn]$. Define the \emph{Stanley-Reisner ring} of $\Delta$ over $A$ to be the quotient $A[\Delta] := A[\varsn]/\SRi{\Delta}$.

        \item A simplex $F \subseteq \vertset$ defines a monomial ideal $P_F := (x_i : v_i \notin F)$. This ideal is prime when $A$ is a domain; for arbitrary $A$, the ideals $P_{\p,F} := \p+P_F$ are prime whenever $\p\in\Spec A$.
    \een
\end{defn}

We next note the following useful fact from 
\cite[Ch.\,IV,\,\S2,\,no.\,6,\,Th.\,2,\,Cor.\,2,\,pp.\,279--280]{Bour89} 
and the Remark following Cor.\,2. 

\begin{prop}\label{ass} Let $A \to R$ be Noetherian with $R$ flat over $A$.
Let $\q \in \Spec(R)$ lie over $\p \in \Spec(A)$, and let $\ov{\q}= (A\sm\p)^{-1}\q\big/\p(A\sm\p)^{-1}R \cong (\q/\p R)_{\p}$
be the corresponding prime of $\kappa_{\p}\otimes_A R$. Then: \linebreak
\centerline{$\q\in \Ass_{R}(R) \iff \Big(\p\in\Ass_A(A) \text{\ and\ }  \ov{\q}\in \Ass_{\kappa_{\p}\otimes_A R}
\big(\kappa_{\p}\otimes_A R)\Big)$.}
\end{prop}

Note the following local depth formula \cite[(21.B)\,Thm.\,50\,and\,(21.C)\,Cor.\,21.1]{matsumura}

{\emergencystretch=2em
\begin{prop}\label{flatdp} For a flat local map of local rings $(A, \n) \to (B, \m)$, 
$\depth_{\m}B = \depth_{\n}A + \depth_{\m/\n}B/\n B$.
\end{prop}\par}

\begin{prop}\label{prop-ADeltabasics}
    Let $A$ be a nonzero Noetherian ring, and
    let $\Delta$ be a nonempty simplicial complex on $n$ vertices $\vertset$.
    \bena
        \item $A[\Delta]$ is a $\Z^n$-graded $A$-module whose nonzero graded components are $[A[\Delta]]_{\Zn{m}} = A\monompos{m} \cong A$ whenever $\supp \monompos{m} \in \Delta$. In particular, $A[\Delta]$ is a free $A$-module with basis $\{\monompos{m} : \supp \monompos{m} \in \Delta\}$.
        \item The set of $\Z^n$-graded prime ideals in $A[\Delta]$ is precisely $\{\p + P_F \; | \; \p \in \Spec A, \; F \in \Delta\}$.
        
        \item The Krull dimension of $A[\Delta]$ is
        $\dim A[\Delta] = \dim A + \dim \Delta + 1$.
         In fact, for any multihomogeneous prime $P = \p + P_F$, we have
        $\height P = \height \p + \dim \; \link_\Delta F + 1.$
        \item {\emergencystretch=2em If $(A, \n, K)$ is local and $\m = \n + (\varsn)$ is the multihomogeneous maximal ideal of $A[\Delta]$, then $\depth_\m A[\Delta] = \depth_\n A + \depth_{(\varsn)}K[\Delta]$.\par}
        \item  If $A$ is Noetherian, the associated primes of $A[\Delta]$ are $\p + P_F$ for $\p \in \Ass(A)$ and $F$ a facet of $\Delta$. In particular, $A[\Delta]$ is reduced if and only if $A$ is reduced.
    \een
\end{prop}
\begin{proof} The first two parts follow from the
monomial basis. For (c), let $P=\p+P_F$ be a multihomogeneous prime. The map
$A_\p\to A[\Delta]_P$ is flat and local, with closed fiber
\[
 \left(\kappa_\p(x_i:v_i\in F)[\link_\Delta F]\right)_{(x_j:v_j\in\link_\Delta F)}.
\]

This fiber has dimension $\dim\link_\Delta F+1$, so the local dimension formula for flat local maps
\cite[Lemma 10.112.7, Tag 00ON]{Stacks} gives the asserted height formula.
For the total dimension let $\cS$ be the set of facets of $\Delta$. The facet decomposition
$I_\Delta=\bigcap_{F \in \cS}P_F$ implies
$ \dim A[\Delta]
 =\max_{F \in \cS} \dim A[x_i:v_i\in F]
 =\dim A+\dim\Delta+1.$
 
Part (d) follows from Proposition~\ref{flatdp}, while part (e) follows from Proposition~\ref{ass}, since $A[\Delta]$ is $A$-free and the fibers
are  Stanley-Reisner rings $\kappa_{\p}[\Delta]$ over a field, in which the associated primes, which are minimal, are generated by images of sets of variables corresponding to sets of vertices of the form $\{\vect v n\} \sm F$, where $F$ is a facet of $\Delta.$
\end{proof}

\subsection{\texorpdfstring{$t$}{t}-monomials and \texorpdfstring{$t$}{t}-Stanley-Reisner rings} 

   Fix a discrete valuation ring $(V,t)$. In this section we will introduce \emph{$t$-monomials} and the correspondence with classical monomials induced by evaluation $x_0 \mapsto t$. We then define \emph{$t$-Stanley-Reisner rings} analogously to classical Stanley-Reisner rings. We shall see that the fundamentals of Stanley-Reisner correspondence still hold in this new setting. 

\begin{defn}
    \label{def-tmonomial}
    Let $\varsn = x_1, \hdots, x_n$ be variables. A \emph{$t$-monomial} is an element 
    \[\tmonom{m} = \tmonomlong{m} \in V[\varsn], \text{\ where\ } \Znplus{m} = (m_0, m_1, \hdots, m_n) \in \N^{n+1}.\]
\end{defn}

\begin{rmk}
    The polynomial ring $V[\varsn]$ has a natural $\Z^n$-grading, given by multidegree in each variable. Up to multiplication by a unit, the homogeneous elements are precisely the $t$-monomials. Homogeneous ideals, therefore, are precisely the ``$t$-monomial ideals"; that is to say, ideals generated by $t$-monomials.
\end{rmk}

\begin{thm}[\mbox{{\cite[Thm. III.2.3]{thesis}}}]
    \label{thm-phi}
    Let $(V,t)$ be a discrete valuation ring , and let $I \subset V[\varsn]$ be a $t$-monomial ideal. There is a unique monomial ideal $\monomify{I} \subset V[\varsnplus]$ such that
        \[
            \tmonom{m} \in I \quad \text{if and only if} \quad \monom{m} \in \monomify{I}.
        \]
    Equivalently, $\monomify{I}$ is the unique monomial ideal such that $\phi(\monomify{I})V[\varsn] = I$, where $\phi$ is the evaluation map $x_0 \mapsto t$. This correspondence preserves associated primes, in the sense that 
        \[
  \Ass_{V[\ux]}(V[\ux]/I) = \{ \phi(\monomify{Q})V[\varsn] \; : \; \monomify{Q} \in 
  \Ass_{V[x_0,\ux]}(V[x_0,\ux]/\monomify{I}) \}.
        \]
\end{thm}

\begin{rmk} Since \cite{thesis} is unpublished, we observe that one may use the methods of
\cite[\S2]{EaHo74} to prove that if $\fA, \, \fB$ are ideals generated by monomials in a permutable regular sequence  $\vect f k$ in an arbitrary commutative ring $T$, then $\fA \cap \fB$, and 
$\fA :_T \fB$ may be computed exactly as in the case of formal monomials in indeterminates: 
one may use the map
$\Z[\vect Xk] \to T$ such that $\vect Xk \mapsto \vect f k$ to get a correspondence. If, in addition, every subset of $\{\vect f k\}$ generates a prime ideal (i.e., the sequence is a permutable {\it prime} sequence), the correspondence preserves primary decomposition.  This is the case for the sequence  $t, \vect x n$ that one uses in the definition
of a $t$-Stanley-Reisner ring. 
\end{rmk}

\begin{rmk}
    Note that the associated primes of a graded ideal are themselves graded. So, the associated primes of a $t$-monomial ideal are $t$-monomial primes, i.e., ideals generated by some subset of $\{t, x_1, \hdots, x_n\}$.
\end{rmk}

\begin{defn} \label{def-tSRnotation}
     We define the following notations, analogous to those in Definition \ref{defn-SRnotation}:
    \benn
        \item The \emph{support} of a $t$-monomial $\tmonom{m} = \tmonomlong{m}$ is the simplex $\supp \tmonom{m} := \{\vertex{i} \; : \; m_i \neq 0\}$.

        \item A $t$-monomial $\tmonom{m}$ is called \emph{\sqf} if $m_i \leq 1$ for every $i$.

        \item A simplex $F \subseteq \zvertset$ defines a \sqf $t$-monomial
        \[
            \tmonomF{F} := \begin{cases}
                \displaystyle t\prod_{v_0 \not=\vertex{i}\in F}x_i &\text{ if } v_0 \in F\\
                \displaystyle \prod_{\vertex{i}\in F}x_i &\text{if } v_0 \notin F
            \end{cases}
        \]
        Note that $\tmonomF{F} = \phi(\zmonomF{F})$, where $\phi$ is the evaluation map $x_0 \mapsto t$.
        \item If $\Delta$ is a simplicial complex on $\zvertset$, define the \emph{$t$-Stanley-Reisner ideal} of $\Delta$ to be the \sqf $t$-monomial ideal
        $\tSRi{\Delta} := (\tmonomF{F}: F \notin \Delta)$, and define the \emph{$t$-Stanley-Reisner ring} of $\Delta$ to be the quotient
        $\tSRrcomp{\Delta} := V[\varsn]/\,\tSRi{\Delta}$.

        \item A simplex $F \subseteq \zvertset$ defines a $t$-monomial prime ideal $\ov{P_F} = \phi(P_F) V[\varsn]$, where $\phi$ is evaluation $x_0 \mapsto t$ and $P_F = (x_i : v_i \notin F) \subset V[\varsnplus]$. Note that $\ov{P_F} = (\tmonomF{\{\vertex{i}\}} : \vertex{i} \notin F)$.
        
    \een
\end{defn}

\begin{rmk} \label{rmk-tSR}
    Using Theorem \ref{thm-phi}, one can check that we have described a bijection between simplicial complexes on $n+1$ vertices and $t$-Stanley-Reisner rings; we call this bijection \emph{the $t$-Stanley-Reisner correspondence.} In fact, the correspondence in Theorem \ref{thm-phi} commutes with Stanley-Reisner correspondence, in the sense that $\tSRi{\Delta} = \phi(\SRi{\Delta})V[\varsn]$ and $\monomify{(\tSRi{\Delta})} = \SRi{\Delta}$. It follows that evaluation $x_0 \mapsto t$ induces a 
    $V$-algebra surjection $ \displaystyle \phi: \SRrcomp{\Delta} \to \tSRrcomp{\Delta} \cong \frac{\SRrcomp{\Delta}}{(x_0 - t)\SRrcomp{\Delta}}.$
    It should be noted that, while $x_0 - t$ is not a homogeneous element, it \textit{is} a nonzerodivisor in the maximal homogeneous ideal of $\SRrcomp{\Delta}$. This is enough to ensure that the quotient map $\phi$ is quite well-behaved.
\end{rmk}

\begin{prop}\label{prop-tSRbasics}
    Let $(V,t, K)$ be a discrete valuation ring, and let $\Delta$ be a nonempty simplicial complex on $n+1$ vertices $\zvertset$. 
    \bena
        \item $\tSRrcomp{\Delta}$ is a $\Z^n$-graded $V$-module with components 
        that vanish if $\Zn{m} \notin \N^n$, while for $\Zn{m} \in \N^n$:
        
        \[
          \left[\tSRrcomp{\Delta}\right]_{\umm} = V \monompos{m} \cong \begin{cases}
              V &\text{ if } \supp{\monompos{m}} \in \link_\Delta \vertex{0}\\
              K &\text{ if } \supp{\monompos{m}} \in (\Delta-v_0)\sm\link_\Delta\{v_0\},\\
              0 &\text{if } \supp{\monompos{m}} \notin \Delta.
          \end{cases}
        \]
     
        In general $\tSRrcomp{\Delta}$ is \emph{not} free over $V$ unless $\Delta$ is a cone over $\vertex{0}$.
        \item The $\Z^n$-graded prime ideals in $\tSRrcomp{\Delta}$ are precisely the images of the primes $\ov{P_F}$ for $F \in \Delta$. We will write $\ov{P_F}$ to mean $\ov{P_F}\tSRrcomp{\Delta}$ from now on.
      
        \item The Krull dimension of $\tSRrcomp{\Delta}$ is $\dim \tSRrcomp{\Delta} =\dim \Delta + 1.$
        In fact, for any graded prime $\ov{P_F}$, 
        we have $\height \ov{P_F} = \dim \; \link_\Delta F + 1.$
        \item Let $\m = \ov{P_{\nl}}= (t, \varsn)$ be the graded maximal ideal of $\tSRrcomp{\Delta}$. Then $\depth_\m \tSRrcomp{\Delta} = \depth_{(\varsnplus)}K[\Delta]$.

        \item The associated primes of $\tSRrcomp{\Delta}$, which are the same as the minimal primes, are $\ov{P_F}$ for facets $F$ of $\Delta$. In particular, $\tSRrcomp{\Delta}$ is always reduced.
    \een
\end{prop}
\begin{proof}
    Statement (a) follows from the observation that $t \cdot \monompos{m} = 0$ in $\tSRrcomp{\Delta}$ if and only if $\supp (t\monompos{m}) = \{\vertex{0}\} \cup \supp(\monompos{m}) \notin \Delta$. Statement (b) follows in a straight-forward fashion from Thm.~\ref{thm-phi} and Prop.~\ref{prop-ADeltabasics}.

    For the proof of (c), note that it follows from (b) that we have an inclusion-reversing bijection between graded prime ideals in $\tSRrcomp{\Delta}$ and faces of $\Delta$. The homogeneous maximal ideal of $\tSRrcomp{\Delta}$ corresponds to the empty face $\nl \in \Delta$.

    Note that the ring is catenary, the chains  of primes corresponding to the faces in the chains indicated below are saturated, and that all minimal primes are graded. Consequently,
    for $P = \ov{P_F} =  (\tmonomF{\vertex{i}} : \vertex{i} \notin F)$ where $F \in \Delta$, we have
    \begin{align*}
        \height P &= \max \{h : \exists \; G_i \in \Delta, \;  F \subset G_1 \subset \hdots \subset G_h\}\\
        &= \max\{|G - F| : F \subseteq G \in \Delta\}\\
        &= \max \{|H| : H \in \link_\Delta F\}\\
        &= \dim \link_\Delta F + 1.
    \end{align*}

In an $\N$-graded algebra over a local ring, the height of the unique  homogeneous maximal ideal is the Krull dimension. Therefore, in particular, we have
    $\dim \tSRrcomp{\Delta} = \height (t,\varsn) = \dim \link_\Delta \nl + 1 = \dim \Delta + 1$.

    For (d), we have observed that $\tSRrcomp{\Delta}$ is the quotient of $\SRrcomp{\Delta}$ by a nonzerodivisor in the homogeneous maximal ideal. It follows that
    \begin{align*}
        \depth_{\m}\tSRrcomp{\Delta} &= \depth_{(t,\varsnplus)}\SRrcomp{\Delta} - 1\\
        &= \depth_t V + \depth_{(\varsnplus)}K[\Delta] - 1 & &\text{by Prop.~\ref{prop-ADeltabasics} (d)}\\
        &= \depth_{(\varsnplus)}K[\Delta] & &\text{since } \depth_t V = 1.
    \end{align*}

The assertion about associated primes in (e) follows from
Proposition \ref{prop-ADeltabasics} (e) and Theorem \ref{thm-phi}. 
Moreover, the monomial intersection rule in Remark~2.19 gives that 
$\displaystyle
\tSRi{\Delta}=\bigcap_{F\text{ a facet of }\Delta}\ov{P_F}$ in $V[\ux]$.
Since every ideal in this intersection is prime, $\tSRrcomp{\Delta}$ is reduced.
\end{proof}

\section{Koszul homology and cohomology} \label{sec-koszul}
There are two famous theorems in Stanley-Reisner theory which are commonly referred to as \textit{Hochster's formula:} one for Koszul homology/cohomology, and one for local cohomology. These formulas employ a common strategy. We start with a multigraded complex of free $K[\Delta]$-modules, the homology or cohomology of which computes the algebraic invariant in which we are interested. The graded pieces are then described in terms of simplicial cochain complexes of certain sub-complexes of $\Delta$. We thereby convert an algebraic computation into a series of smaller computations that can often be simplified using insights and standard results from topology.

The goal of \S\ref{sec-koszul} is to adapt Hochster's formula for Koszul homology and cohomology to $t$-Stanley-Reisner rings. We accomplish this by providing short exact sequences that relate $\Homol(t, \varsn; \tSRrcomp{\Delta})$ to $\Homol(\varsnplus; V[\Delta])$, which can be computed with the original Hochster's formula. As an application, we compute the Betti numbers of the $t$-Stanley-Reisner ring defined by a triangulation of $\Prj^2_{\R}$.
Later, in \S\ref{sec-local-cohomology}, we will also see the adaptation of the sister formula for local cohomology.

\subsection{Background: Koszul Homology and Cohomology} 

\begin{defn} \label{def-mappingcone}
    Let $\psi_\bullet:(A_\bullet, d^A_\bullet) \to (B_\bullet, d^B_\bullet)$ be a chain map. The \emph{mapping cone} of $\psi$ is the complex $\mapcone{\psi}_\bullet$ with terms $\mapcone{\psi}_i := A_{i-1} \oplus B_i$ and $i^{th}$ differential
    \[
       \begin{pmatrix}
            -d^A_{i-1} & 0\\
            \psi_{i-1} & d^B_i
        \end{pmatrix}: A_{i-1}\oplus B_i \to A_{i-2}\oplus B_{i-1} \qquad x \oplus y \mapsto -d^A_{i-1} (x) \oplus (\psi_{i-1}(x) + d^B_i (y)).
    \]

See \cite[\S10 p.~650]{rotman} or 
\cite[\S1.5]{weibel}.
    
    For convenience, we state the cohomological version separately: if $\psi^\bullet: (A^\bullet, d_A^\bullet) \to (B^\bullet, d_B^\bullet)$ is a map of complexes then $\mapcone{\psi}^\bullet$ is the complex with terms $\mapcone{\psi}^i := A^{i+1} \oplus B^i$ and $i^{th}$ differential
    \[
        \begin{pmatrix}
            -d_A^{i+1} & 0\\
            \psi^{i+1} & d_B^i
        \end{pmatrix}: A^{i+1}\oplus B^i \to A^{i+2}\oplus B^{i+1} \qquad
    x \oplus y \mapsto -d_A^{i+1}(x) \oplus (\psi^{i+1}(x) + d_B^i(y)).
    \]
\end{defn}

By \cite[Lemma 10.38]{rotman} or
\cite[1.5.2, p.~19]{weibel}, we have the following:

\begin{fact} If $\psi_\bullet: A_\bullet \to B_\bullet$ is a chain map, there is a short exact sequence of complexes
    \[
       0 \to B \to \mapcone{\psi} \to A[-1] \to 0
    \]
    and thus an induced long exact sequence
    \[
       \hdots \to \Homol_i(A) \to \Homol_i(B) \to \Homol_i(\mapcone{\psi}) \to \Homol_{i-1}(A) \to \hdots
    \]
    Likewise, if $\psi^\bullet: A^\bullet \to B^\bullet$ is a \textit{cochain} map, we have exact sequences
    \begin{align*}
        & 0 \to B \to \mapcone{\psi} \to A[1] \to 0\\
        \hdots &\to\Homol^i(A) \to \Homol^i(B) \to \Homol^i(\mapcone{\psi}) \to \Homol^{i+1}(A) \to \hdots
    \end{align*}
\end{fact}

\begin{defn} \label{def-koszul}
    Let $R$ be a commutative ring. Given a sequence of elements in $R$, we define the \emph{Koszul complex} inductively as follows. For any element $f \in R$, define $\Kos_\bullet(f; R)$ to be
    \[
         0 \to \Kos_1 \xrightarrow{\cdot f} \Kos_0 \to 0,
    \]
    where $\Kos_0 = \Kos_1 = R$. If $f \in R$ and $\seq{g} = g_1, \hdots, g_n \in R$, then $\Kos_\bullet(f, \seq{g};R)$ is the mapping cone of
    \[
         \Kos_\bullet(\seq{g}; R) \xrightarrow{\cdot f} \Kos_\bullet(\seq{g}; R).
    \]
    If $M$ is an $R$-module, the Koszul complex of $\seq{f}$ on $M$ is defined as $\Kos_\bullet(\seq{f};M) := M \otimes_R \Kos_\bullet(\seq{f};R)$. The \emph{Koszul homology} of $\seq{f}$ on $M$ is the homology of $\Kos_\bullet(\seq{f}; M)$ and is denoted by $\Homol_\bullet(\seq{f};M)$. 

    The \emph{Koszul co-complex} or \emph{cohomological Koszul complex} of $\seq{f}$ on $M$ is the dual complex $\Kos^\bullet(\seq{f};M) := \Hom_R(\Kos_\bullet(\seq{f};R), M)$. The \emph{Koszul cohomology} of $\seq{f}$ on $M$ is the cohomology of $\Kos^\bullet(\seq{f}; M)$ and is denoted by $\Homol^\bullet(\seq{f};M)$. 
\end{defn}

\begin{facts} \label{facts-koszul}
    Let $R$ be a commutative ring, and let $\seq{f} = f_1, \hdots, f_n \in R$. Let $M$ be an $R$-module.
    \bena
        \item We have $\Homol^i(\seq{f}; M) \cong \Homol_{n-i}(\seq{f};M)$ for every $0 \leq i \leq n$. In fact, the dual complex $\Kos^\bullet(\seq{f};M)$ is isomorphic to $\Kos_\bullet(\seq{f};M)$ with the terms numbered in the reverse order. 
       
        \item $\Kos_i(f_1, \hdots, f_n;R)$ is a free $R$-module of rank $\binom{n}{i}$.
        \item We have $\Homol_0(\seq{f};M) \cong \Homol^{n}(\seq{f};M) \cong M/\seq{f}M$. Moreover, if $\seq{f}$ is a regular sequence on $M$, then $\Kos_\bullet(\seq{f};M)$ has 0 homology in positive degrees and is a free resolution of $M/(\seq{f})M$ if $M$ is $R$-free, e.g., if $M = R$.
       
        \item If $\psi:R \to S$ is a map of rings, then $\Kos_\bullet(\psi(\seq{f});S) \cong \Kos_\bullet (\seq{f};R) \otimes_R S$.
        \item \label{koszul-facts-matrix} If $B$ is an invertible $n \times n$ matrix over $R$, then $\Kos_\bullet(\seq{f};R) \cong \Kos_\bullet(B\seq{f}; R)$.
        \item \label{koszul-facts-ses} For any $f, \seq{g} \in R$ and any index $i > 0$, there is a short exact sequence
        \[
        0 \to \frac{\Homol_i(\seq{g};R)}{f \cdot \Homol_i(\seq{g};R)} 
           \to \Homol_i(f, \seq{g};R) 
           \to \Ann_{\Homol_{i-1}(\seq{g};R)}f   \to 0.
        \]
   \item The elements  $\seq{f}$ annihilate all of the modules 
       $\Homol_{\bullet}(\seq{f};M)$ and $\Homol^{\bullet}(\seq{f};M)$.
        
    \een
\end{facts}

There is a compendium of properties of Koszul homology in \cite[Thm.~3.4]{BHM24}. Part (f)~above follows from 
\cite[IV-2, Prop.\,1]{serre}.

\begin{rmk}\label{Koszul-simp} In a standard alternative treatment (see, for example, \cite[Ch.\,IV.,\,A)]{serre}) one constructs the Koszul complex from the free $R$-module $G = \bigoplus_{i=1}^n Ru_i$ as the exterior algebra $\bigwedge_R^{\bullet}(G) =
\bigoplus_{i=0}^n \bigwedge_R^i(G)$ in which the differential is the unique $R$-derivation of degree $-1$ extending the map $\bigwedge^1(G) = G \to R = \bigwedge^0(G)$ such that $u_i  \mapsto f_i$. The derivation property means when $u$ is homogeneous of degree $k$, we have $d(u \wedge u') = du \wedge u' + (-1)^k u\wedge du'$.  

The simplicial homology of  $\Delta$ with coefficients in a ring $R$,
where $\Delta$ is a simplicial complex contained in the 
$(n-1)$-simplex 
$\Sigma$ with vertices  $\{\vect v n\}$, may be recovered (or even defined) as follows.  In the formulation of the Koszul complex above, take the $u_i$ to correspond bijectively with
the $v_i$. Let $\seq{1}$ denote the sequence $f_1 := 1, \hdots, f_n := 1$.
If $F$ is the oriented class of 
$(v_{i_0}, \ldots, v_{i_k})$ in $\Sigma$, let 
$u_F := u_{i_0} \wedge \cdots \wedge u_{i_k} \in\ \Kos_{\bullet}(\und{1};\,R)$
correspond to $F$. For every $j \in \N$, place
$\cK_j(\und{1};\, R)$ in simplicial degree $j-1$, retaining its differential.
The subcomplex spanned by the basis elements $u_F$ is 
the reduced simplicial chain complex of $\Delta$ with coefficients in $R$.  
To treat cohomology, for every $j \in \N$ place $\Kos^j(\und{1};R)$ in degree $j-1$ and form the 
quotient by the subcomplex spanned by the dual basis vectors corresponding to non-faces of 
$\Delta$. This gives  the reduced simplicial cochain complex.
\end{rmk}

\subsection{Hochster's formula} 
We will now give the original statement of Hochster's formula and  provide the short exact sequences that relate Koszul (co)homology of $t$-Stanley-Reisner rings to Koszul (co)homology of Stanley-Reisner rings.

\begin{thm}[{\cite[Thm 5.2]{hochform}}] \label{thm-og-hoch-koszulver}
   Let $A$ be any commutative ring, and let $\Delta$ be a simplicial complex on vertices $V = \zvertset$. Consider $A$ as an $A[\varsnplus]$-module via $A \cong A[\varsnplus]/ (\varsnplus)$. Then we have isomorphisms of $A[\varsnplus]$-modules
   \begin{align*}
       \Tor_i^{A[\varsnplus]}(A, A[\Delta]) &\cong \bigoplus_{\begin{array}{c}
            -1 \leq q \leq n - i\\
            T \subseteq V, \; |T| = n - i-q
        \end{array}} \simpHomol^q(\Delta- T;A),\\
        \\
        \Ext_{A[\varsnplus]}^j(A, A[\Delta]) &\cong \bigoplus_{\begin{array}{c}
            -1 \leq q \leq j-1\\
            T \subseteq V, \; |T| = j - q - 1        \end{array}} \simpHomol^q(\Delta- T;A).
   \end{align*}
\end{thm}

\begin{rmk}
    By Facts \ref{facts-koszul} (c), the Koszul complex $\Kos_\bullet (\varsnplus;A[\varsnplus])$ is a free resolution of $A$ as an $A[\varsnplus]$-module. It follows that
    \begin{align*}
        \Tor_i^{A[\varsnplus]}(A, A[\Delta]) &= \Homol_i(\varsnplus; A[\Delta]),\\
        \Ext_{A[\varsnplus]}^j(A, A[\Delta]) &= \Homol^j(\varsnplus; A[\Delta]) \cong \Homol_{n+1-j}(\varsnplus; A[\Delta]).
    \end{align*}
\end{rmk}

\begin{thm}
    \label{thm-koszul-ses} \LetVDelta. Then for every $i \in \N$ there is a short exact sequence 
    of $V[\varsnplus]$-modules  
    \[
      0 \to  \frac{\Homol_i(\varsnplus;\SRrcomp{\Delta})}{t \cdot \Homol_i(\varsnplus;\SRrcomp{\Delta})}      \to 
      \Homol_i(t, \varsn; \tSRrcomp{\Delta}) \to \Ann_{\Homol_{i-1}(\varsnplus;\SRrcomp{\Delta})}t 
      \to 0.
    \]
    Likewise, for every $j\in \N$, we have a short exact sequence
    \[
      0 \to  \frac{\Homol^j(\varsnplus;\SRrcomp{\Delta})}{t \cdot \Homol^j(\varsnplus;\SRrcomp{\Delta})} \to \Homol^j(t, \varsn; \tSRrcomp{\Delta}) \to \Ann_{\Homol^{j+1}(\varsnplus;\SRrcomp{\Delta})}t  \to 0.
    \]
\end{thm}

\begin{proof}
    The short exact sequence for $\Homol_\bullet$ is a corrected version of a result proved in \cite[Thm IV.4.12]{thesis}. From the short exact sequence of complexes
    \[
       0 \to \Kos_\bullet(\varsnplus; \SRrcomp{\Delta}) \xrightarrow{\cdot (x_0 - t)} \Kos_\bullet(\varsnplus; \SRrcomp{\Delta}) \to \Kos_\bullet(t, \varsn; \tSRrcomp{\Delta}) \to 0,
    \] 
{\emergencystretch=2em
    take the long exact sequence in homology. Multiplication by $x_0$ induces zero on $\Homol_i(\varsnplus;V[\Delta])$, by Facts~\ref{facts-koszul}(g). Consequently, the map induced by $x_0-t$ is multiplication by $-t$. Its cokernel is the quotient by $t$, and its kernel is the annihilator of $t$, giving the asserted short exact sequence. Here $\Homol_{-1}=0$ if $i=0$. \par}

    The corresponding short exact sequence for $\Homol^\bullet$ follows immediately from the self-duality in  Facts~\ref{facts-koszul} (a), since both Koszul complexes have $n+1$ generators.
\end{proof} 

\begin{cor}\label{cor-hochsform-tSR-koszulver}
     \LetVDelta. Considering the residue field $K$ as a $V[\varsn]$-module via $K \cong V[\varsn]/(t, \varsn)$, we have isomorphisms of $K$-vector spaces
    \begin{align*}
        \Tor_i^{V[\varsn]}(K, \tSRrcomp{\Delta}) &\cong \frac{\Tor_{i}^{V[\varsnplus]}(V, \SRrcomp{\Delta})}{t \cdot \Tor_{i}^{V[\varsnplus]}(V, \SRrcomp{\Delta})}  \; \oplus\;  \Ann_{\Tor_{i-1}^{V[\varsnplus]}(V, \SRrcomp{\Delta})}t\\
        \\
        \Ext_{V[\varsn]}^j(K, \tSRrcomp{\Delta}) &\cong \frac{\Ext^{j}_{V[\varsnplus]}(V, \SRrcomp{\Delta})}{t \cdot \Ext^{j}_{V[\varsnplus]}(V, \SRrcomp{\Delta})}  \; \oplus\;\Ann_{\Ext^{j+1}_{V[\varsnplus]}(V, \SRrcomp{\Delta})}t .
    \end{align*}  
{\emergencystretch=2em
    Combining this with Hochster's formula (Thm.~\ref{thm-og-hoch-koszulver}), we see that the $K$-vector space dimensions of $\Tor_i^{V[\varsn]}(K, \tSRrcomp{\Delta})$ and $\Ext_{V[\varsn]}^j(K, \tSRrcomp{\Delta})$ can be computed from the simplicial cohomology modules of deletions $\simpHomol^q(\Delta - T; V)$.\par}

    \end{cor}

\subsection{Betti numbers} 
\label{sec-betti-numbers}

\begin{defn}
    Let $(R, \m)$ be a (graded) local ring with residue field $K = R/\m$, and let $M$ be a finitely generated $R$-module, assumed graded in the graded case. The \emph{$i^{th}$ Betti number} of $M$ is $b_i(M) := \dim_K \Tor^R_i(K, M)$.
\end{defn}

\begin{rmk}
     If $(R,\m, K)$ is regular, then $b_i(M)$ is the rank of the $i^{th}$ term in a minimal free resolution of $M$ (a minimal graded free resolution in the graded case). Without the graded hypothesis on $M$, these numbers instead give the ranks in a minimal free resolution of $M_\m$ over $R_\m$. In this setting, we can compute $b_i(M)$ with Koszul homology: for any system of parameters $\seq{x}$ that generates $\m$, i.e., any regular system of parameters, the Koszul complex $\Kos_{\bullet}(\seq{x}; R)$ is a free resolution of $K$, and thus $b_i(M) = \dim_K \Homol_i(\seq{x};M)$.
    Let $(V,t)$ be a DVR with residue field $K$. The polynomial ring $V[\varsn]$ is regular with homogeneous maximal ideal $(t, \varsn)$. It follows that the Betti numbers of $\tSRrcomp{\Delta}$ are 
    $\dim_K \Homol_i(t, \varsn; \tSRrcomp{\Delta})$. 
    By Corollary \ref{cor-hochsform-tSR-koszulver}, these can be computed from 
    the data $\simpHomol^q(\Delta - T; V)$ for 
    $T \subseteq \zvertset$ and 
    $-1 \leq q \leq n-i+1$, with terms  outside the range of Hochster's formula omitted.
    More explicitly, for $i\geq1$, the two summands in Corollary~\ref{cor-hochsform-tSR-koszulver} give
    \[
    \begin{aligned}
     b_i(\tSRrcomp{\Delta})
     &=\sum_{\substack{-1\leq q\leq n-i\\
                      T\subseteq\zvertset,\ |T|=n-i-q}}
     \dim_K\frac{\simpHomol^q(\Delta-T;V)}
                      {t\simpHomol^q(\Delta-T;V)}\\
     &\quad+\sum_{\substack{-1\leq q\leq n-i+1\\
                      T\subseteq\zvertset,\ |T|=n-i+1-q}}
     \dim_K\Ann_{\simpHomol^q(\Delta-T;V)}t.
    \end{aligned}
    \]
    For $i=0$, only the first sum occurs, since $\Homol_{-1}=0$.
    Thus $n-i+1$ is the upper bound for the union of the two ranges.
    The extra endpoint can contribute: in Example~\ref{ex-betti-P2},
    $n=5$, $i=4$, and in mixed characteristic $2$ the term with $q=2$
    and $T=\nl$ contributes to $b_4$.
\end{rmk}


\begin{examp}\label{ex-betti-P2}
The triangulation $\Delta$ of the real projective plane described next will be used as a running example throughout this paper.

\noindent 
\begin{minipage}[t]{0.64\textwidth}
This triangulation has 6 vertices $v_0, v_1, v_2, v_3, v_4, v_5$, 15 edges (all that are possible), and 10 two-simplices. Consider a regular hexagon with opposite sides identified, so that the boundary becomes a triangulated circle with vertices $v_0,\, v_1,\, v_2$. Edges are identified so that the corresponding vertices match up. Connect interior points  $v_3, \,v_4, \, v_5$ to each other and to $v_0,\, v_1, \, v_2$. The 10 two-simplices, which are the facets,  are these: $\{v_0,\,v_1,\,v_4\}$, $\{v_0,\,v_1,\,v_5\}$, $\{v_0,\,v_2,\,v_3\}$, $\{v_0,\,v_2,\,v_5\}$, $\{v_0,\,v_3,\,v_4\}$, $\{v_1,\,v_2,\,v_3\}$, $\{v_1,\,v_2,\,v_4\}$, $\{v_1,\,v_3,\,v_5\}$, $\{v_2,\,v_4,\,v_5\}$, and $\{v_3,\,v_4,\,v_5\}$.    

\end{minipage}%
\hfill
\begin{minipage}[t]{0.32\textwidth}
{  \tikzset{every picture/.style={line width=0.75pt}}   
\begin{tikzpicture}[x=0.75pt,y=0.75pt,yscale=-1,xscale=1,baseline={([yshift=-2.0ex]current bounding box.north)}]
\draw  [fill={rgb, 255:red, 155; green, 155; blue, 155 }  ,fill opacity=0.41 ] (74.78,133.42) -- (23.81,105.33) -- (23.81,49.17) -- (74.78,21.08) -- (125.75,49.17) -- (125.75,105.33) -- cycle ;
\draw [color={rgb, 255:red, 189; green, 16; blue, 224 }  ,draw opacity=1 ][line width=1.5]    (23.81,49.17) -- (74.78,21.08) ;
\draw [color={rgb, 255:red, 189; green, 16; blue, 224 }  ,draw opacity=1 ][line width=1.5]    (74.78,133.42) -- (125.75,105.33) ;
\draw [color={rgb, 255:red, 126; green, 211; blue, 33 }  ,draw opacity=1 ][line width=1.5]    (23.81,49.17) -- (23.81,105.33) ;
\draw [color={rgb, 255:red, 126; green, 211; blue, 33 }  ,draw opacity=1 ][line width=1.5]    (125.75,49.17) -- (125.75,105.33) ;
\draw [color={rgb, 255:red, 245; green, 166; blue, 35 }  ,draw opacity=1 ][line width=1.5]    (74.78,21.08) -- (125.75,49.17) ;
\draw [color={rgb, 255:red, 245; green, 166; blue, 35 }  ,draw opacity=1 ][line width=1.5]    (23.81,105.33) -- (74.78,133.42) ;
\draw  [fill={rgb, 255:red, 0; green, 0; blue, 0 }  ,fill opacity=1 ] (50.72,49.08) .. controls (50.72,47.11) and (52.4,45.5) .. (54.48,45.5) .. controls (56.55,45.5) and (58.23,47.11) .. (58.23,49.08) .. controls (58.23,51.06) and (56.55,52.67) .. (54.48,52.67) .. controls (52.4,52.67) and (50.72,51.06) .. (50.72,49.08) -- cycle ;
\draw    (23.81,49.17) -- (125.75,49.17) ;
\draw    (23.81,105.33) -- (125.75,105.33) ;
\draw  [fill={rgb, 255:red, 0; green, 0; blue, 0 }  ,fill opacity=1 ] (91.67,49.08) .. controls (91.67,47.11) and (93.35,45.5) .. (95.42,45.5) .. controls (97.49,45.5) and (99.17,47.11) .. (99.17,49.08) .. controls (99.17,51.06) and (97.49,52.67) .. (95.42,52.67) .. controls (93.35,52.67) and (91.67,51.06) .. (91.67,49.08) -- cycle ;
\draw  [fill={rgb, 255:red, 0; green, 0; blue, 0 }  ,fill opacity=1 ] (71.02,105.33) .. controls (71.02,103.36) and (72.7,101.75) .. (74.78,101.75) .. controls (76.85,101.75) and (78.53,103.36) .. (78.53,105.33) .. controls (78.53,107.31) and (76.85,108.92) .. (74.78,108.92) .. controls (72.7,108.92) and (71.02,107.31) .. (71.02,105.33) -- cycle ;
\draw    (74.78,108.92) -- (74.78,133.42) ;
\draw    (54.48,49.08) -- (23.81,105.33) ;
\draw    (54.48,49.08) -- (74.78,105.33) ;
\draw    (95.42,49.08) -- (74.78,105.33) ; 
\draw    (95.42,49.08) -- (125.75,105.33) ;
\draw    (54.48,49.08) -- (74.78,21.08) ;
\draw    (74.78,21.08) -- (95.42,49.08) ;
\draw  [fill={rgb, 255:red, 208; green, 2; blue, 27 }  ,fill opacity=1 ] (71.02,21.08) .. controls (71.02,19.1) and (72.7,17.5) .. (74.78,17.5) .. controls (76.85,17.5) and (78.53,19.1) .. (78.53,21.08) .. controls (78.53,23.06) and (76.85,24.66) .. (74.78,24.66) .. controls (72.7,24.66) and (71.02,23.06) .. (71.02,21.08) -- cycle ;
\draw  [fill={rgb, 255:red, 208; green, 2; blue, 27 }  ,fill opacity=1 ] (71.02,133.42) .. controls (71.02,131.44) and (72.7,129.84) .. (74.78,129.84) .. controls (76.85,129.84) and (78.53,131.44) .. (78.53,133.42) .. controls (78.53,135.4) and (76.85,137) .. (74.78,137) .. controls (72.7,137) and (71.02,135.4) .. (71.02,133.42) -- cycle ; 
\draw  [fill={rgb, 255:red, 248; green, 231; blue, 28 }  ,fill opacity=1 ] (121.99,49.17) .. controls (121.99,47.19) and (123.67,45.58) .. (125.75,45.58) .. controls (127.82,45.58) and (129.5,47.19) .. (129.5,49.17) .. controls (129.5,51.14) and (127.82,52.75) .. (125.75,52.75) .. controls (123.67,52.75) and (121.99,51.14) .. (121.99,49.17) -- cycle ;
\draw  [fill={rgb, 255:red, 248; green, 231; blue, 28 }  ,fill opacity=1 ] (20.06,105.33) .. controls (20.06,103.36) and (21.74,101.75) .. (23.81,101.75) .. controls (25.88,101.75) and (27.56,103.36) .. (27.56,105.33) .. controls (27.56,107.31) and (25.88,108.92) .. (23.81,108.92) .. controls (21.74,108.92) and (20.06,107.31) .. (20.06,105.33) -- cycle ;
\draw  [fill={rgb, 255:red, 74; green, 144; blue, 226 }  ,fill opacity=1 ] (121.99,105.33) .. controls (121.99,103.36) and (123.67,101.75) .. (125.75,101.75) .. controls (127.82,101.75) and (129.5,103.36) .. (129.5,105.33) .. controls (129.5,107.31) and (127.82,108.92) .. (125.75,108.92) .. controls (123.67,108.92) and (121.99,107.31) .. (121.99,105.33) -- cycle ;
\draw  [fill={rgb, 255:red, 74; green, 144; blue, 226 }  ,fill opacity=1 ] (20.06,49.17) .. controls (20.06,47.19) and (21.74,45.58) .. (23.81,45.58) .. controls (25.88,45.58) and (27.56,47.19) .. (27.56,49.17) .. controls (27.56,51.14) and (25.88,52.75) .. (23.81,52.75) .. controls (21.74,52.75) and (20.06,51.14) .. (20.06,49.17) -- cycle ;

\draw (66,6) node [anchor=north west][inner sep=0.75pt]    {$v_{0}$};
\draw (45,61) node [anchor=north west][inner sep=0.75pt]    {$v_{3}$};
\draw (88.67,61.08) node [anchor=north west][inner sep=0.75pt]    {$v_{4}$};
\draw (78,107) node [anchor=north west][inner sep=0.75pt]    {$v_{5}$};
\draw (130,45) node [anchor=north west][inner sep=0.75pt]    {$v_{1}$};
\draw (130,100) node [anchor=north west][inner sep=0.75pt]    {$v_{2}$};
\draw (68,138) node [anchor=north west][inner sep=0.75pt]    {$v_{0}$};
\draw (3,100) node [anchor=north west][inner sep=0.75pt]    {$v_{1}$};
\draw (2,45) node [anchor=north west][inner sep=0.75pt]    {$v_{2}$};
\end{tikzpicture}}
\end{minipage}

The faces and non-faces of the complex are highly symmetric. For $0 \leq i \leq 5$,  let $i' = 5-i$. Then we may describe the 10 two-simplices that are in $\Delta$ (\resp, not in $\Delta$) as $\spx {i'} {j'} {k'}$, $\spx ij {i'}$, and $\spx {i'} {j'} k$ for all choices of $i, j, k$ such that $\{i,\,j,\, k\} = \{0,\,1,\,2\}$ (\resp, $\{3,\,4,\,5\}$).  Note that $\Delta$ has simplicial automorphisms that arbitrarily permute the three pairs $(i,\,i')$ for $i \in \{0,\,1,\,2\}$. Moreover, $\Delta$ is isomorphic to its Alexander dual $\Delta^\vee$. The defining ideal of the $t$-Stanley-Reisner ring $V\ov{[\Delta]}$ is generated by the 10 square-free $t$-monomials corresponding to the 3 element subsets of the vertices not in $\Delta$.

We wish to compute the Betti numbers of the $t$-Stanley-Reisner ring $\tSRrcomp{\Delta}$. 
When $\Delta-T$ has a vertex, the simplicial cohomology modules $\simpHomol^q(\Delta-T;V)$ are isomorphic as $V$-modules to the reduced singular cohomology modules of the underlying topological space, which are invariant under homotopy equivalence. For $T=\zvertset$, the augmented simplicial complex $\Delta-T=\{\nl\}$ instead gives $\simpHomol^{-1}(\Delta-T;V)=V$ and zero cohomology in every other degree. Note also that $\simpHomol^q(\Delta- T;V)$ vanishes for $q \geq 3$, since $\dim \Delta = 2$. We give the data for the cohomology of 
$\Delta - T$ in the following table:

\begin{center}
\renewcommand{\arraystretch}{1.3}
    \begin{tabular}{|c|c|c|c|c|c|c|}
\hline
{\boldmath $|T|$} & {\boldmath $\Delta - T$} & \textbf{\boldmath \# Choices of $T$} & \rbx{-2.0pt}{{\boldmath $ \simpHomol^{-1}$}} &\rbx{-2.0pt}{{\boldmath $\simpHomol^{0}$}}  &\rbx{-2.0pt}{{\boldmath $\simpHomol^{1}$}}  &\rbx{-2.0pt}{{\boldmath $\simpHomol^{2}$}}  \\ \hline
6 & $\{\nl\}$         & 1  &  $V$  & $0$ & $0$ & $0$ \\ \hline
5 & $\bullet$           & 6  &  $0$  & $0$ & $0$ & $0$ \\ \hline
\rbx{2.0pt}{4} & \rbx{-3.5pt}{\includegraphics[scale=0.2]{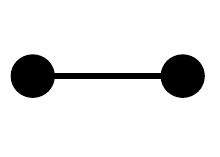}
\rbx{5.0pt}{$\simeq\bullet$}}   & 15 &  $0$  & $0$ & $0$ & $0$ \\ \hline
\rbx{-6.0pt}{3} & \Strut{25pt} \includegraphics[scale=0.2]{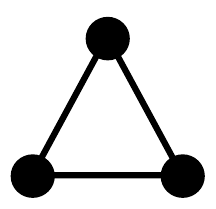}
  \rbx{6.0pt}{$\cong \; S^1$ (circle)}       &\rbx{6.5pt}{10} &\rbx{6.5pt}{$0$}  &\rbx{6.5pt}{$0$} & \rbx{6.5pt}{$V$} & \rbx{6.5pt}{$0$} \\
  &\kern -35pt {\includegraphics[scale=0.2]{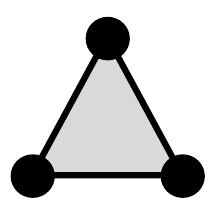}}
     \rbx{6.5pt}{$\simeq \,\,\,\bullet$}                              &\rbx{6.5pt}{10} & 
     \rbx{6.5pt}{$0$}  &\rbx{6.5pt}{$0$} & \rbx{6.5pt}{$0$} & \rbx{6.5pt}{$0$} \\ \hline
\rbx{8.5pt}{2} & \Strut{30pt}\includegraphics[scale=0.2]{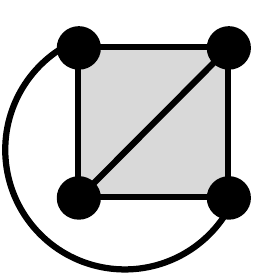}
 \rbx{9.0pt}{$\simeq \; S^1$ (circle)}                      & \rbx{8.5pt}{15} & \rbx{8.5pt}{$0$}  & 
     \rbx{8.5pt}{$0$} & \rbx{8.5pt}{$V$} & \rbx{8.5pt}{$0$} \\ \hline
\rbx{6.0pt}{1} & \Strut{27pt}\includegraphics[scale=0.2]{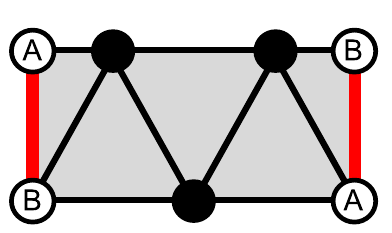}
    \rbx{6.0pt}{M\"obius strip $\simeq S^1$}  &\rbx{7.0pt}{6} & \rbx{7.0pt}{$0$} & \rbx{7.0pt}{$0$} & \rbx{7.0pt}{$V$} & \rbx{7.0pt}{$0$} \\ \hline
0 & $|\Delta| \cong \mathbb{P}^2_{\R}$             & 1 &  $0$  & $0$ & $\Ann_V 2$ & $V/2V$ \\ \hline
\end{tabular}
\end{center}
We now use Hochster's formula to compute the Koszul homology modules $\Homol_i(\varsnplus; \SRrcomp{\Delta})$ of the Stanley-Reisner ring $\SRrcomp{\Delta}$.
{\allowdisplaybreaks
\begin{align*}
    \Homol_0(\varsnplus; \SRrcomp{\Delta}) &= V\\
    \Homol_1(\varsnplus; \SRrcomp{\Delta}) &= V^{\oplus 10}\\
    \Homol_2(\varsnplus; \SRrcomp{\Delta}) &= V^{\oplus 15}\\
    \Homol_3(\varsnplus; \SRrcomp{\Delta}) &= \frac{V}{2V} \oplus V^{\oplus 6}\\
    \Homol_4(\varsnplus; \SRrcomp{\Delta}) &= \Ann_V 2\\
    \Homol_q(\varsnplus; \SRrcomp{\Delta}) &= 0 & \text{ for all } q \geq 5.
\end{align*}
Note that $V/2V$ vanishes unless $\charac K = 2$, and $\Ann_V 2$ vanishes unless $V$ itself has equal characteristic $2$, in which case $\Ann_V 2 = V$. We have
\begin{align*}
    \frac{V/2V}{t\cdot (V/2V)} &\cong \frac{V}{(2,t)V} = \begin{cases}
        K &\text{if } 2 \in tV,\\
        0 &\text{if } 2 \in V^\times.
    \end{cases}\\
    \\
    \Ann_{V/2V} t &= \frac{2V :_V t}{2V} =  \begin{cases}
        K &\text{if } \sqrt{2V} = tV, \text{ i.e. if $2 \in tV$ and $2 \neq 0$,}\\
        0 &\text{if $2 \in V^\times$ or if $\mathbb{F}_2 \subset V.$}   
    \end{cases}\\
    \\
    \frac{\Ann_V 2}{t \cdot \Ann_V 2} &= \begin{cases}
        K &\text{if } \mathbb{F}_2 \subset V,\\
        0 &\text{else}.
    \end{cases}\\
    \\
    \Ann_{\Ann_V 2} t &= 0.    
\end{align*}

Finally, we calculate the Betti numbers of the $t$-Stanley-Reisner ring $\tSRrcomp{\Delta}$. The zeroth Betti number can be computed directly; for the rest, we use Corollary \ref{cor-hochsform-tSR-koszulver}. 

\begin{align*}
    b_0(\tSRrcomp{\Delta}) &= \dim_K \Homol_0(t, \varsn; \tSRrcomp{\Delta}) \\
         &= \dim_K K \\
         &= 1.\\
    \\
    b_1(\tSRrcomp{\Delta}) &= \dim_K \frac{\Homol_1(\varsnplus; \SRrcomp{\Delta})}{t \cdot \Homol_1(\varsnplus; \SRrcomp{\Delta})} +
        \dim_K \Ann_{\Homol_0(\varsnplus; \SRrcomp{\Delta})}t\\
        &= \dim_K \frac{V^{\oplus 10}}{t \cdot V^{\oplus 10}} + \dim_K \Ann_Vt\\
        &= \dim_K K^{\oplus 10}\\
        &= 10.\\
    \\
    b_2(\tSRrcomp{\Delta}) &=\dim_K \frac{\Homol_2(\varsnplus; \SRrcomp{\Delta})}{t \cdot \Homol_2(\varsnplus; \SRrcomp{\Delta})}  +  \dim_K \Ann_{\Homol_1(\varsnplus; \SRrcomp{\Delta})}t\\
        &=\dim_K \frac{V^{\oplus 15}}{t \cdot V^{\oplus 15}} + \dim_K \Ann_{V^{\oplus 10}}t \\
        &= \dim_K K^{\oplus 15}\\
        &= 15.
    \\
    b_3(\tSRrcomp{\Delta}) &= \dim_K \frac{\Homol_3(\varsnplus; \SRrcomp{\Delta})}{t \cdot \Homol_3(\varsnplus; \SRrcomp{\Delta})} +\dim_K \Ann_{\Homol_2(\varsnplus; \SRrcomp{\Delta})} t\\
        &=  \dim_K \frac{\frac{V}{2V} \oplus V^{\oplus 6}}{t \cdot (\frac{V}{2V} \oplus V^{\oplus 6})} + \dim_K \Ann_{V^{\oplus 15}}t \\
        &= \begin{cases}
            6 &\text{ if $2 \in V^{\times}$},\\
            7 &\text{if }2 \in tV.
        \end{cases}
    \\
    b_4(\tSRrcomp{\Delta}) &= \dim_K \frac{\Homol_4(\varsnplus; \SRrcomp{\Delta})}{t \cdot \Homol_4(\varsnplus; \SRrcomp{\Delta})} + \dim_K \Ann_{\Homol_3(\varsnplus; \SRrcomp{\Delta})}t\\
        &=\dim_K \Ann_{\frac{V}{2V} \oplus V^{\oplus 6}}t +\dim_K \Ann_V 2 \otimes_V K\\
        &= \begin{cases}
            0 &\text{ if $2 \in V^{\times}$},\\
            1 &\text{if } 2 \in tV.
        \end{cases}\\
    b_5(\tSRrcomp{\Delta}) &= \dim_K \Ann_{\Ann_V 2} t = 0.
\end{align*}
}
\end{examp}

\section{Local cohomology}
\label{sec-local-cohomology} 

  We will now adapt Hochster's formula for local cohomology to the setting of $t$-Stanley-Reisner rings. In the previous section, we reduced the problem for $t$-Stanley-Reisner rings to an application of the original theorem. We now employ a different strategy: the main theorem of this section 
  (Thm.~\ref{Main4}) will be proved analogously to the proof of the original Hochster's formula. This proof strategy highlights the comparisons and contrasts between the classical setting and the new mixed characteristic setting. 
  It turns out that \textit{cellular sheaf cohomology} (see \S\ref{subsec-cellular-sheaves}) is an essential tool for our primary treatment, although we do give an alternative treatment which avoids the use of sheaves. We also discuss the special case where $\Delta$ is a homological manifold with boundary (see \S \ref{homby}), which generalizes our running example of the real projective plane.

\subsection{Background: Hochster's formula for local cohomology}
We first give a brief treatment of a classical 
formula that expresses the components of the local cohomology of a \SR ring in terms of reduced simplicial cohomology. This is followed by a result that
gives similar information about the local cohomology of an ideal generated by square-free monomials in  a \SR ring in terms of reduced relative simplicial cohomology. This result is not difficult to deduce from the original formula.

\subsubsection{Local cohomology for \SR rings}\label{orig} 

\begin{defn}
    Let $R$ be a commutative ring, and let $\seq{f} = f_1, \hdots, f_n \in R$ be minimal generators for an ideal $I \subset R$, up to radicals. The \emph{local cohomology modules} $\Homol_{I}^\bullet(R)$ are the cohomology of the modified \Cech complex
    $0 \to R \to \bigoplus R_{f_i} \to \hdots \to R_{f_1\hdots f_n} \to 0$
     where the transition maps are induced by signed localization maps $(-1)^j: R_{f_{i_0}\hdots \widehat{f_{i_j}} \hdots f_{i_q}} \to R_{f_{i_0}\hdots f_{i_q}}$ for every $i_0 < \hdots < i_q$ and $0 \leq j \leq q$.
\end{defn}

\begin{rmk} \label{rmk-cech-SR}
    For a Stanley-Reisner ring $R = A[\Delta]$ and the ideal $(\varsn)$ generated by the variables, the terms $R_{x_{i_0}\hdots x_{i_q}}$ of the \Cech complex vanish except when $\vertlist{i_0}{i_q}$ is a face of $\Delta$. Moreover, the signs of the maps $R_{\monomF{F}} \to R_{\monomF{G}}$ in the \Cech complex are given by $\bndsgn{\oriented{F}}{\oriented{G}}$ (see Definition \ref{def-simpchain}).
\end{rmk}

We now state Hochster's formula\footnote{The formula in Thm.~\ref{thm-og-hochsform-localcohomver} 
 was communicated by Hochster to Stanley in a letter in the early 1970s, and published, with permission, by Stanley in \cite[Thm.~4.1, p.~60]{Sta96}.}
for local cohomology and briefly describe the main ideas of the proof, which we will generalize in \S 
\ref{subsec-loc-cohom-tSR}:

\begin{thm}[{\cite[Thm 5.3.8]{brunsherz}}]\label{thm-og-hochsform-localcohomver}
  Let $K$ be a field, and let $\Delta$ be a simplicial complex on $\vertset$. Let $\m$ be the homogeneous maximal ideal of the Stanley-Reisner ring $K[\Delta]$. The multigraded Hilbert series of the local cohomology module $\Homol^i_{\m}(K[\Delta])$ is
  \[
     \mathrm{HS}_{\Homol^i_{\m}(K[\Delta])}(\Zn{s}) = \sum_{F \in \Delta} \dim_K \simpHomol_{i - |F| - 1}(\link_\Delta F; K) \prod_{\vertex{j} \in F}\frac{s_j^{-1}}{1- s_j^{-1}}.
  \]
  More generally, let $A$ be any commutative ring and let $\m$ be the ideal generated by the variables in the Stanley-Reisner ring $A[\Delta]$. For $\Zn{a} \in \Z^n$, define simplices 
  $H_{\Zn{a}} := \{ \vertex{i} : a_i > 0 \}$ and $G_{\Zn{a}} := \{\vertex{i} : a_i < 0\}.$ Then:
    
  \[
  [\Homol^i_{\m}(A[\Delta])]_{\Zn{a}} \cong \begin{cases}
         \simpHomol^{i-|G_{\Zn{a}}|-1}(\link_\Delta G_{\Zn{a}} ; A) &\text{if } H_{\Zn{a}} = \nl,\\
         0 &\text{else.}
     \end{cases}
  \]
\end{thm}

\begin{proof}[Idea of proof:]
    Let $\check C^\bullet$ be the \Cech complex for $R := A[\Delta]$ with respect to $\m = (\varsn)$. 
  For any face $F \in \Delta$, the component $[R_{\monomF{F}}]_{\Zn{a}}$ is isomorphic to $A$ if and only if $G_{\Zn{a}} \subseteq F$ and $F \cup H_{\Zn{a}} \in \Delta$. Otherwise, $[R_{\monomF{F}}]_{\Zn{a}}$ vanishes \cite[Lemma 5.3.6]{brunsherz}. Using this fact, one constructs an isomorphism of cochain complexes
  \[
    [\check C^\bullet]_{\Zn{a}} \cong \simpchain^\bullet(\link_{\st_\Delta H_{\Zn{a}}} G_{\Zn{a}}; A)[-|G_{\Zn{a}}|-1] \qquad \qquad \text{\cite[Lemma 5.3.7]{brunsherz}}
  \]
 The final key observation is that when $H_{\Zn{a}}\neq\nl$, the complex $\link_{\st_\Delta H_{\Zn{a}}} G_{\Zn{a}}$ is either $\Nl$ or a cone, with any vertex of $H_{\Zn{a}}$ as cone vertex. Its reduced cohomology therefore vanishes.  It follows that
  \[ 
     [\Homol^i_{\m}(R)]_{\Zn{a}} \cong \begin{cases}
         \simpHomol^{i-|G_{\Zn{a}}|-1}(\link_\Delta G_{\Zn{a}} ; A) &\text{if } H_{\Zn{a}} = \nl,\\
         0 &\text{else.}
     \end{cases}
  \]
If $A$ is a field, simplicial cohomology is the $A$-linear dual of simplicial homology in the same degree; in particular, their dimensions agree.
\end{proof}

\subsubsection{Local cohomology of a square-free monomial ideal in a \SR ring}\label{cohomJ} 

In this subsection, we consider an ideal $J$ of a \SR ring that is generated by square-free monomials in the ring and develop a formula for the local cohomology of $J$ with support in the ideal generated by the variables.   
We shall apply these results to a method for studying the local cohomology of \tSR rings later, in \S\ref{altcohom}. When we do that, $J$ will correspond to the annihilator of $t$ in a \tSR ring. The result we obtain here can be construed as 
a generalization of Theorem~\ref{thm-og-hochsform-localcohomver}.

\begin{notationdiscussion} Throughout this subsection, $\Delta$ will denote a finite abstract simplicial complex with vertex set $\vect v n$.  We let $A$ denote an arbitrary nonzero commutative ring, and $\ux: =\vect xn$ be  indeterminates over $A$ corresponding to the vertices $\vect v n$. We will consider an ideal $J \subset A[\Delta]$ generated by square-free monomials whose supports are in $\Delta$. The nonzero square-free monomials in $A[\Delta]$ are those with support in a face of $\Delta$, and since $J$ is an ideal we have that if $\monomF{F} \in J$ and $F \subseteq G$, then $\monomF{G} \in J$ also. Hence, the set of faces $F$ such that $\monomF{F}$ is \textit{not} in $J$ forms a simplicial subcomplex of $\Delta$. We denote this subcomplex as $\La_J$, but we frequently omit the subscript.  In fact, there is a bijection between ideals of $A[\Delta]$ generated by square-free monomials and subcomplexes $\La \inc \Delta$.  Given $\La\inc \Delta$, we simply let $J:=J_{\La}$ be spanned over $A$ in $A[\Delta]$ by all monomials whose support is not in $\La$. The subscripts on $J_{\La}$ and $\La_J$ will typically be omitted: unless otherwise specified, $J$ and $\La$ will correspond, each to the other.

\end{notationdiscussion}

\begin{fact}We have an obvious short exact sequence
$(\natural) \quad 0 \to J \to A[\Delta] \arwf{\pi} A[\Lambda] \to 0$,
where $\pi$ is the unique $A$-algebra homomorphism that kills $J$ and maps the image of every monomial in $\ux$ with support in $\Lambda$ that is in $A[\Delta]$ to the image of the same monomial in $A[\La]$. 
\end{fact}

\begin{notation} Given any algebra or module $M$ over
$A[\ux]$, and this includes all the \SR rings over
simplices with vertices in $\{\vect v n\}$ and their
ideals, we let $\Cv^{\bullet}(M)$ denote the modified
\Cech complex of $M$ with respect to $\ux$, whose
cohomology gives the local cohomology modules $H^{\bullet}_{(\ux)}(M)$. We immediately have:
\end{notation}

\begin{fact}\label{Cechcx}  There is an induced short exact sequence of modified \Cech complexes
\[(\dagger)\quad 0 \to \Cv^{\bu}(J) \to \Cv^{\bu}(A[\Delta]) \to \Cv^{\bu}(A[\Lambda]) \to 0. \]
\end{fact}

\begin{discussion}\label{relcohomJ}
Each of these modified \Cech complexes is $\Z^n$-graded as usual, where the grading comes 
from the grading on 
$A[\vect x n,\, x_1^{-1}, \, \ldots, \, x_n^{-1}]$ in which  for $\Zn{a} := (\vect a n) \in \Z^n$, the degree $\Zn{a}$ component 
is $Ax_1^{a_1}\cdots x_n^{a_n}$.  

The maps of \Cech complexes  preserve this grading, and so the exact sequence $(\dagger)$ is the direct sum of a family, indexed by  $\Zn{a} \in \Z^n$, of short exact sequences of much smaller complexes. We want to describe these $\Zn{a}$ components and, for every $\Zn{a} \in \Z^n$,  to understand the long exact sequence for cohomology that it provides.  
In this process, the component complexes turn
out to be finite complexes of finite rank free $A$-modules and, as we shall soon see, in every instance the  long exact sequence we obtain from the $\Zn{a}$ component is a standard one in algebraic topology. 

The components of $\Cv^{\bullet}(A[\Delta])$ and 
$\Cv^{\bullet}(A[\La])$ have already been analyzed in the preceding subsection in 
Theorem~\ref{thm-og-hochsform-localcohomver} and the following discussion titled \textit{``Idea of proof."} 
At the cochain level, the  maps for the complexes coming from the links of 
$G_{\Zn{a}}$ in $\st_\Delta H_\ua$ and $\st_{\La}H_{\ua}$ are 
the same as those used in defining relative simplicial cohomology of the pair.  
When $H_{\ua} = \nl$, these are the same as $\Delta$ and $\La$, \resp. 
The cited treatments of relative Stanley-Reisner theory use a coefficient field. For the local cohomology formula here, the cochain construction works over an arbitrary commutative ring; 
no identification of homology with cohomology is needed.  
See \cite[\S3.1, pp.\,199--202]{Hat02}, 
\cite[\S\S43--45]{Mun84a}, and \cite[\S\S3,4]{Spa95}.  

The relative simplicial cohomology of the pair can be viewed alternatively as the 
cohomology of the relative simplicial complex $\Delta/\La$, in the terminology of \cite[p.\,205]{Sta87} or \cite[Ch.\,3,\,\S7,\,p.\,116]{Sta96}, which is the set $\Delta \sm \La$.
There is also an exposition of relative Stanley-Reisner theory in 
\cite[\S1,\,\,p.\,103]{AdSa16}. 
\end{discussion}

Throughout the rest of this paper, if $\ua := (\vect an), \und{b} := (\vect bn)$, we 
write $\ua \leq \und{b}$ to mean that $a_i \leq b_i$ for $1 \leq i \leq n$.

From the discussion above:

\begin{thm}\label{cohJ} With notation as above, if $\ua : = (\vect an) \in \Z^n$ then 
$[H^j_{\ux}(J)]_{\ua}$ vanishes unless
$\ua \leq \Zo$. If $\ua \leq \Zo$, let $G_{\ua}$ denote the 
set of $v_i$ such that $a_i <0$. Then
\[ [H^j_{(\ux)}(J)]_{\ua} \cong \tH^{j-|G_{\ua}|-1}(\link_{\Delta}G_{\ua},\,
\link_{\La} G_{\ua}; \, A),\]
the reduced relative simplicial cohomology of the pair
$(\link_{\Delta}G_{\ua}, \, \link_{\La}G_{\ua})$ with coefficients in $A$.
\end{thm}

\begin{rmk}\label{empty}  Recall that the link of a subset of the vertices of an ordinary simplicial complex is $\Nl$ when the subset is not a face; see Definition~\ref{def-link-star}. Thus $\tH^{\bu}(\Sigma; A) = 0$ when $\Sigma := \Nl$ is the empty simplicial complex, while if $\Sigma =\{ \nl \}$ then $\tH^{-1}(\Sigma; \, A) \cong A$ while $\tH^i(\Sigma;\, A) =0$ 
if $i \not= -1$. See Remark~\ref{0}. 
 \end{rmk}

\begin{discussion}\label{rel} We include here some brief remarks about 
relative cohomology. Note that \cite{Hat02} often refers to $\Delta$-complexes,
which are defined in  \cite[\S2.1]{Hat02}. A simplicial complex is a $\Delta$-complex
and a $\Delta$-complex is a CW-complex \cite[p.\,107]{Hat02}.  $\Delta$-complexes are
called {\it semi-simplicial complexes} in \cite{EZ50}, where they were introduced.

If $Y \inc X$, where $Y$ is a non-empty closed subspace of $X$, is
the geometric realization of an inclusion of CW-complexes (and, hence, of
$\Delta$-complexes or of simplicial complexes), 
then $(X,\,Y)$ is a {\it good pair}, i.e. $Y$ is a 
deformation retract of an open neighborhood in $X$.  See \cite[Appendix,\,Prop.\,A.5]{Hat02}. 
For good pairs  $(X,\,Y)$  the reduced relative cohomology  $\tH^{\bullet}(X,\,Y)$ is 
the same as $\tH^{\bullet}(X/Y,\,Y/Y)$ over any coefficient ring. (This is also true for
reduced relative homology). Here, $X/Y$ is the quotient space, so that $Y/Y$ is a 
single point when $Y$ is nonempty. See \cite[Thm.\,2.13,\,\S3.1]{Hat02}.  When $Y = \{y\}$ is 
a single point of $X$, 
then $H^i(X,\{y\};A)\cong H^i(X;A)$ for all $i>0$, and evaluation at $y$ gives a split exact sequence
\[
 0 \to H^0(X,\{y\};A)\to H^0(X;A)\to A\to 0.
\]
If $X$ has finitely many connected components, the first two modules are free, and their ranks differ by one. 
Note also that if $Y$ is nonempty, 
$\tH^{\bu}(X,\,Y) \cong H^{\bu}(X,\,Y)$, since the empty set is a face of $Y$. 
\end{discussion}

\begin{remark}\label{relhom}  When the subcomplexes in the pairs have nonempty geometric realizations, reduced relative simplicial cohomology agrees with ordinary relative cohomology and is invariant under homotopy equivalences of pairs. See \cite[\S4]{Spa95}. For void complexes and complexes consisting only of the empty face, the augmentation data must also be retained; the geometric pair alone does not determine the relative cohomology. In all cases, homotopies of simplicial pairs that induce homotopies of their augmented relative cochain complexes preserve the cohomology.
\end{remark}

\subsection{Local cohomology for \texorpdfstring{$t$}{t}-Stanley-Reisner rings} 
\label{subsec-loc-cohom-tSR}

\begin{prop} 
    \label{prop-tSR-localization-components}
    Let $(V,t)$ be a discrete valuation ring with fraction field $L$ and residue field $K$, and let $\Delta$ be a simplicial complex on $\zvertset$. Let $F \in \Delta$. The localization $\tSRrcomp{\Delta}[1/\tmonomF{F}]$ is $\Z^n$-graded, and for $\Zn{a} \in \Z^n$ the degree $\Zn{a}$ homogeneous component is
    \[
    \left[ \tSRrcomp{\Delta} \left[\frac{1}{\tmonomF{F}} \right]\right]_{\Zn{a}} \cong \begin{cases}
        L &\text{if } \vertex{0} \in F, \; G_{\Zn{a}} \subseteq F, \text{ and } H_{\Zn{a}} \cup F \in \Delta,\\
        V &\text{if } \vertex{0} \notin F, \; G_{\Zn{a}} \subseteq F, \text{ and } \vertex{0} \cup H_{\Zn{a}} \cup F \in \Delta,\\
        K &\text{if } \vertex{0}\notin F, \; G_{\Zn{a}}\subseteq F, \; H_{\Zn{a}}\cup F \in \Delta \text{ and } \vertex{0} \cup H_{\Zn{a}}\cup F \notin \Delta,\\
        0 &\text{if } G_{\Zn{a}}\not\subseteq F \text{ or } H_{\Zn{a}}\cup F \notin \Delta,
    \end{cases}
    \]
    where $G_{\Zn{a}} := \{\vertex{i} \;|\; a_i < 0\}$ and $H_{\Zn{a}} := \{ \vertex{i} \;|\; a_i > 0\}$ are simplices in $\vertset$.
\end{prop}   

\begin{proof}
Let $R := \tSRrcomp{\Delta}$, $s:=\tmonomF{F}$, and $S=R[1/s]$.
If $G_{\Zn{a}}\not\subseteq F$, the degree-$\Zn{a}$ part of every
fraction with denominator a power of $s$ is zero, because a negative
exponent would occur in a variable that has not been inverted. Suppose
now that $G_{\Zn{a}}\subseteq F$. Then
$u=\monompos{a}$ is a well-defined Laurent monomial in $S$. Every
degree-$\Zn{a}$ fraction is a scalar multiple of $u$, with scalars in
$V$ if $v_0\notin F$ and in $L$ if $v_0\in F$.
For sufficiently large $k$, $s^ku$ is a $t$-monomial in $R$ whose
support is $H_{\Zn{a}}\cup F$. The square-free defining ideal therefore
gives
\[
 u=0\text{ in }S
 \quad\Longleftrightarrow\quad H_{\Zn{a}}\cup F\notin\Delta.
\]
If $v_0\in F$ and $u\ne0$, then $S_{\Zn{a}}=Lu\cong L$.
If instead $v_0\notin F$ and $u\ne0$, the same support test gives, for
every $m\geq1$,
\[
 t^m u=0\text{ in }S
 \quad\Longleftrightarrow\quad
 \{v_0\}\cup H_{\Zn{a}}\cup F\notin\Delta.
\]
Since every nonzero element of $V$ is a unit times a nonnegative power
of $t$, the annihilator of $u$ in $V$ is $0$ when this union is a face,
and is $tV$ otherwise. Consequently $S_{\Zn{a}}$ is $V$ or $K$,
respectively. Together with the two zero cases, this proves the stated
formula.
\end{proof}

\begin{rmk} \label{rmk-cech-decomposition}
    Fix $\Zn{a} \in \Z^n$, and let $\check{C}^\bullet$ be the \Cech complex of $\tSRrcomp{\Delta}$ for $t, \varsn$. Using Proposition \ref{prop-tSR-localization-components}, we can decompose the $\Zn{a}^{th}$ graded piece of the $i^{th}$ term of the \v Cech complex:
    \[
       \left[\check{C}^i \right]_{\Zn{a}} = L^{\oplus \calig{A}^i} \oplus V^{\oplus \calig{B}^i} \oplus K^{\oplus \calig{C}^i},
    \]
    where
    \begin{align*}
        \calig{A}^i &:= \{F \in \Delta : |F| = i, \; \vertex{0} \in F, \; G_{\Zn{a}} \subseteq F, \text{ and } H_{\Zn{a}} \cup F \in \Delta\}\\
        \calig{B}^i &:= \{F \in \Delta -\vertex{0}: |F| = i, \;  G_{\Zn{a}} \subseteq F, \text{ and } H_{\Zn{a}} \cup F \in \link_\Delta \vertex{0}\}\\
        \calig{C}^i &:= \{F \in \Delta - \vertex{0} : |F| = i, \; G_{\Zn{a}} \subseteq F,\; H_{\Zn{a}} \cup F \in \Delta, \text{ and } H_{\Zn{a}} \cup F \notin \link_\Delta \vertex{0}\}.\\
    \end{align*}
    
    In \cite[Lemma 5.3.7]{brunsherz}, it is shown that the set
    \[
     \calig{A}^i \cup \calig{B}^i \cup \calig{C}^i = \{ F \in \Delta : |F| = i, G_{\Zn{a}} \subseteq F, \text{ and } H_{\Zn{a}} \cup F \in \Delta\}.
    \]
     is in bijection with the $(i-|G_{\Zn{a}}|-1)$-faces of $\link_{\st_\Delta H_{\Zn{a}}}G_{\Zn{a}}$, and that in the equicharacteristic setting this bijection induces an isomorphism of complexes from the $\Zn{a}^{th}$ graded piece of the \v Cech complex to the cochain complex of $\link_{\st_\Delta H_{\Zn{a}}}G_{\Zn{a}}$ (with the appropriate degree shift), which turns out to be contractible and thus acyclic whenever $H_{\Zn{a}} \neq \nl$. 

     We shall see that, in our setting, the same bijection of faces now induces an isomorphism to the cochain complex of $\link_{\st_\Delta H_{\Zn{a}}} G_{\Zn{a}}$ with coefficients in a \textit{cellular sheaf}. 
\end{rmk}

\subsubsection{A brief introduction to cellular sheaf cohomology} 
\label{subsec-cellular-sheaves}

\begin{defn} \label{def-cellularsheaf} \cite[Def. 4.1.6]{curry}
Let $\cC$ be an abelian category, and let $\Delta$ be a simplicial complex. View $\Delta$ as a category whose objects are faces of $\Delta$ and whose morphisms are face inclusions $F \inj G$. A $\cC$-valued \emph{cellular sheaf} on $\Delta$ is a functor $\mathscr{F}:\Delta\to \cC$.
\end{defn}

\begin{rmk}
More explicitly, for each inclusion of faces $F \subseteq G$, we have a ``restriction map" $ \scr{F}(F) \to \scr{F}(G)$ in the category $\cC$ such that for $F \subseteq G \subseteq H$ we have a commuting diagram

 \[\begin{tikzcd}
  {\scr{F}(F)} & {\scr{F}(G)} & {\scr{F}(H)}
  \arrow[from=1-1, to=1-2]
   \arrow[curve={height=-18pt}, from=1-1, to=1-3]
   \arrow[from=1-2, to=1-3]
\end{tikzcd}\]

The use of the term ``restriction map" here may seem a bit suspect at first glance. For a sheaf $\mathscr{F}$ on a topological space $X$, the restriction maps are morphisms $\mathscr{F}(U) \to \mathscr{F}(V)$ for open subsets $V$ and $U$ such that $V$ is \textit{contained in} $U$. Put plainly, the restriction maps literally describe the restriction of $\mathscr{F}(U)$ to a smaller open set. For cellular sheaves, the ``restriction maps" appear to be pointing the wrong direction!
 
There is nevertheless a natural way to view cellular sheaves as sheaves in the classic topological sense. Equip $\Delta$ with the \emph{upper Alexandrov topology}, where the open sets are those sets closed under upwards inclusion; cellular sheaves are then equivalent to topological sheaves on $\Delta$.  The cohomology used below is the augmented cellular cohomology, with a term in degree $-1$ at the empty face, as in reduced simplicial cohomology. It must be distinguished from the ordinary sheaf cohomology on the full Alexandrov face poset: for a nonvoid complex, global sections on that poset are $\sF(\nl)$, an exact functor, so the higher derived-functor cohomology vanishes. For the void complex it also vanishes.  Cf.~Discussion~\ref{simp-vs-sing}.
\end{rmk}

\begin{defn} \label{def-cellularsheafcohomology} \cite[6.2.1]{curry} \cite[9.3]{EAT}
  Let $\Delta$ be a simplicial complex, $\cC$ an abelian category, and $\scr{F}: \Delta \to \cC$ a cellular sheaf. The \emph{(reduced) cochain complex of $\Delta$ with coefficients in $\scr{F}$} has terms
 \[
     \simpchain^i(\Delta; \scr{F}) := \bigoplus_{\dim F = i} \scr{F}(F)
 \]
 The differential $\cobnd_{\scr{F}}: \simpchain^i(\Delta; \scr{F}) \to \simpchain^{i+1}(\Delta; \scr{F})$ has coordinate maps $\coord{\cobnd_\scr{F}}{F}{G}: \scr{F}(F) \to \scr{F}(G)$ whenever $F \subseteq G$ for an $i$-face $F$ and an $(i+1)$-face $G$. These maps are the restriction maps provided by the sheaf with the sign $\bndsgn{\oriented{F}}{\oriented{G}} = \pm 1$, as outlined in Definition \ref{def-simpchain}. We denote the cohomology of this complex by $\tH^\bullet(\Delta; \scr{F})$. Note that $\tC^{\bullet}(\Delta; \sF)$ has a term in degree $-1$. When $\cC$ is a category of modules, we may also think of $\simpchain^i(\Delta; \scr{F})$ as formal linear combinations of the $i$-simplices of $\Delta$
such that the coefficient of a given $i$-simplex $F$ is an element of $\sF(F)$. Of course, if $\sF$ takes the value $A$ on every face and all restriction maps are the identity, this coincides with the usual definition 
of reduced simplicial cohomology with coefficients in $A$.
 \end{defn}

\subsubsection{Application of cellular sheaf cohomology to local cohomology of \texorpdfstring{$t$}{t}-Stanley-Reisner rings} 

The theory of cellular sheaf cohomology was developed, in a special case, in the unpublished thesis \cite{shepard}, pursuing ideas first presented in seminar by Fulton, Goresky, MacPherson, and McCrory in 1977-78. In the case where the coefficients are taken to be vector spaces over $\Q$, the thesis investigates the bounded constructible derived category of sheaves on a stratified topological space $X$ and, for certain choices of $X$, a bounded constructible {\it cellular} derived category, which is more concrete in a number of ways.  In several cases, including the case of the geometric realization of a finite simplicial complex, the categories are shown to be equivalent. The theory became of great interest to applied mathematicians beginning with \cite{curry}. There is an exposition in \cite{EAT}, whose terminology we follow here, where coefficients in the category of modules over a ring are allowed. The treatment of the cohomology theory we give, following \cite{EAT}, involves comparatively little machinery and is reasonably self-contained.  In particular, it does not require knowledge of derived categories.  The cohomology theory we consider arises quite naturally in studying the local cohomology of $t$-Stanley-Reisner rings.

\begin{prop} \label{prop-cech-vs-sheaf}
 Let $(V,t)$ be a discrete valuation ring with fraction field $L$ and residue field $K$, and let $\Delta$ be a simplicial complex on $\zvertset$. Let $\check{C}$ be the \Cech complex of $\tSRrcomp{\Delta}$ for $t, \varsn$. Fix a multidegree $\Zn{a} \in \Z^n$, and let $\Lambda := \link_{\st_\Delta H_{\Zn{a}}} G_{\Zn{a}}$. Consider the cellular sheaf of $V$-modules $\scr{F}$ on $\Lambda$ defined by
  \[\scr{F}(F) := \begin{cases}
        L &\text{if } \vertex{0} \in F,\\
        V &\text{if } F \in \link_{\Lambda} \vertex{0},\\
        K &\text{otherwise},
    \end{cases}  \]
where the restriction maps are the natural inclusion $V \inj L$, the natural surjection $V \surj K$, and the identity maps on $V$, $L$, and $K$. Then $[\check{C}]_{\Zn{a}} \cong \simpchain(\Lambda; \scr{F})[-1-|G_{\Zn{a}}|]$ as complexes of $V$-modules.
\end{prop}
\begin{proof}
Let $\calig{A}^i$, $\calig{B}^i$, and $\calig{C}^i$ be as in Remark \ref{rmk-cech-decomposition}, and define
\begin{align*}
 \calig{X}^j &:= \{F' \in \Lambda : \dim F' = j, \; \vertex{0} \in F'\}  & &(L\text{-valued faces)}\\
 \calig{Y}^j &:= \{F' \in \Lambda : \dim F' = j, \; F' \in \link_\Lambda \vertex{0}\}  & &(V\text{-valued faces)}\\
\calig{Z}^j &:= \{F' \in \Lambda : \dim F' = j, \; \vertex{0} \cup F' \notin \Lambda \}  & &(K\text{-valued faces)}
\end{align*}

Note that $\simpchain^j(\Lambda; \scr{F}) = L^{\oplus \calig{X}^j} \oplus V^{\oplus \calig{Y}^j} \oplus K^{\calig{Z}^j}$ as a $V$-module. From \cite[Lemma 5.3.7]{brunsherz}, we have a bijection
\[
       \phi: \calig{A}^i \cup \calig{B}^i \cup \calig{C}^i \to \calig{X}^{i-|G_{\Zn{a}}|-1} \cup \calig{Y}^{i-|G_{\Zn{a}}|-1} \cup \calig{Z}^{i-|G_{\Zn{a}}|-1} \qquad F \mapsto F - G_{\Zn{a}}
\]
It is straightforward to check that $\phi(\calig{A}^i) = \calig{X}^{i-|G_{\Zn{a}}|-1}$, $\phi(\calig{B}^i) = \calig{Y}^{i-|G_{\Zn{a}}|-1}$, and $\phi(\calig{C}^i) = \calig{Z}^{i-|G_{\Zn{a}}|-1}$. Thus, $\phi$ induces an isomorphism of $V$-modules $[\check{C}^i]_{\Zn{a}} \cong \simpchain^{i-|G_{\Zn{a}}|-1}(\Lambda; \scr{F})$.

For a fixed multidegree $\Zn{a}\in\Z^n$, reorder the vertices so that those of $G_{\Zn{a}}$ are last. Deleting these final vertices preserves the incidence sign for adjoining any other vertex. Thus the face bijection identifies the signed localization maps with the cellular coboundary maps before the sign convention for a shift is imposed. Put $r=|G_{\Zn{a}}|+1$. The differential on $\simpchain^\bullet(\Lambda;\scr{F})[-r]$ is $(-1)^r$ times the cellular differential. Consequently, multiplying the face identification in \Cech degree $i$ by $(-1)^{ri}$ gives an isomorphism of complexes with the stated shift.
\end{proof}

\begin{discussion}\label{XYZ}  In this discussion and Proposition~\ref{vanishing},
we develop the ideas in Proposition~\ref{prop-cech-vs-sheaf} further.
This will provide the tools we need for the proof of Theorem~\ref{Main4}.

Let $\La$ be a finite simplicial complex on an ordered vertex set in which the first element 
is $v_0$, which need not be a vertex of $\La$. Let $\La_0:= \link_\La v_0 \inc \La-v_0$.  As usual, let $(V,tV)$ be a DVR with fraction field $L$ and 
residue field $K$, and let $\sF$ have the values and maps in Proposition \ref{prop-cech-vs-sheaf}. 
Let $C^\bullet:= \tC^\bullet(\La;\sF)$.

The faces contributing to $C^\bullet$ fall into three classes: those containing $v_0$, those in $\La_0$, and those in $(\La-v_0)\sm \La_0$. Their coefficients are $L$, $V$, and $K$, respectively.
Let $X^\bullet$ and $Z^\bullet$ be the subcomplexes supported on the first and third classes, \resp.

In a simplicial complex, if $\sigma \inc \tau$ are faces, $\tau$ is called a {\it coface} of 
$\sigma$. Evidently,  a coface of a face containing $v_0$ again contains $v_0$. Also, 
if $T\in \La-v_0 \sm \La_0$, no coface of $T$ in $\La$ can contain $v_0$ or lie in $\La_0$. These observations show that $X^\bullet$ and $Z^\bullet$ are subcomplexes. The quotient by their direct sum is the complex $Y^\bullet$ supported on $\La_0$. Consequently,
\ben
\item \label{eq-SES-XYZ}  $0\to X^\bullet\oplus Z^\bullet \to C^\bullet\to Y^\bullet \to 0$  
\een
is a short exact sequence of complexes. In general, $Y^\bullet$ is a quotient complex,
but not a subcomplex of $C^\bullet$.

Deleting $v_0$ from a face that contains it gives that
\benr
\item \label{eq-XYZ-vals} $X^\bullet\cong\tC^\bullet(\La_0;L)[-1],\ 
 Y^\bullet\cong\tC^\bullet(\La_0;\,V),\ \text{and}\ 
 Z^\bullet\cong\tC^\bullet(\La-v_0,\La_0;\,K)$.  
\een
For the last identification, relative cochains are the cochains on $\La-v_0$ that vanish on 
$\La_0$, so their basis faces are exactly $\La-v_0\sm \La_0$. In the first identification, the shift includes the usual minus sign on the differential: deleting a vertex other than the initial $v_0$ changes its position by one.

The snake lemma construction of the cohomology of sequence (\ref{eq-SES-XYZ}) gives
\benr                      
\item \quad \label{eq-LES-XYZ} $\cdots\to \Homol^{j-1}(Y) \arwf{\delta^{j-1}}H^j(X)\oplus \Homol^j(Z) \to \Homol^j(C) \to \Homol^j(Y)
 \arwf{\delta^j}\Homol^{j+1}(X)\oplus \Homol^{j+1}(Z) \to \cdots$.  
\een
Here is an explicit description of the connecting map. Represent a class in $\Homol^j(Y)$ by a cocycle $y$ and lift it to $C^j$ by setting its $X$- and $Z$-coordinates equal to zero. Since $d_Yy=0$, the differential of this lift lies in $X^{j+1}\oplus Z^{j+1}$. Its cohomology class is $\delta^j([y])$. Changing the representative or the lift changes this element by a coboundary. This is the standard snake lemma construction; see also \cite[Lemma 12.13.12,\,Tag\,\,0117]{Stacks}.
For the vanishing argument below, (\ref{eq-XYZ-vals}) and (\ref{eq-LES-XYZ}) suffice. 
Let $\la^j:\tH^j(\La_0;\,V)\to\tH^j(\La_0;\,L)$, and
$ \pi^j:\tH^j(\La_0;\,V)\to\tH^j(\La_0;\,K)$
be the coefficient maps, and let 
$\partial^j:\tH^j(\La_0;\,K)\to\tH^{j+1}(\La-v_0,\La_0;\,K)$
be the connecting map for the pair $(\La-v_0,\Lambda_0)$. Note that $\la^j$,
which is induced by localization, fails to be injective precisely when $\tH^j(\La_0;\,V)$ has nonzero $V$-torsion. From (\ref{eq-XYZ-vals}) and (\ref{eq-LES-XYZ}) we have
\benr
\item \label{eq-alpha} $\delta^j = \alpha^j  := (\la^j,\partial^j\pi^j): \tH^j(\La_0;\,V)\to
 \tH^j(\La_0;\,L)\oplus\tH^{j+1}(\La-v_0,\, \La_0;\,K).$  
\een

The $X$ part of the lifted differential is coefficient inclusion: adjoining $v_0$ has 
incidence sign $+1$. The $Z$ part is obtained by reducing $y$ modulo $t$, extending it by zero to 
$\La-v_0$, and taking its coboundary. This is exactly the definition of $\partial^j\pi^j$.

Thus, a useful short form of (\ref{eq-LES-XYZ}) is 
\benr
\item \label{eq-alpha-SES}$0 \to \Coker\alpha^{j-1}  \to \tH^j(\La;\,\sF) \to \Ker\alpha^j\to 0.$  
\een
In general this sequence need not split. Also, note that $\pi^j$ is not necessarily surjective 
even though $V\to K$ is surjective. \end{discussion}

\begin{prop}[Vanishing when the positive support is nonempty]\label{vanishing}
Let $\Delta$ be a finite simplicial complex on $\{v_0,v_1,\ldots,v_n\}$. Let $G,H\inc\{v_1,\ldots,v_n\}$ be disjoint, with $H\ne\nl$, and put $\La=\link_{\st_\Delta H}G.$
Equip $\La$ with the sheaf of Proposition~\ref{prop-cech-vs-sheaf}. Then
 $\tH^j(\La;\,\sF)=0$ for every $j$.
\end{prop}
\begin{proof}
Fix $h\in H$. A face $F$ belongs to $\La$ precisely when
$F\cap G=\nl$ and $F\cup G\cup H \in \Delta$.
Whenever these conditions hold for $F$, they also hold for $F \cup \{h\}$. The same assertion 
holds after requiring $v_0\notin F$, and after requiring in addition $F\cup\{v_0\}\in\La$. 
Therefore each of $\La-v_0$ and $\La_0=\link_\La\{v_0\}$
is either void or a cone with cone vertex $h$. Their reduced cohomology vanishes over 
$V$, $L$, and $K$.

In particular, $X$ and $Y$ in (2) have zero cohomology. The usual relative cochain sequence
\[0\to\tC^\bullet(\La-v_0,\La_0;\,K)  \to \tC^\bullet(\La-v_0;\,K)
\to\tC^\bullet(\La_0;\,K)\to 0\]
shows that $Z$ also has zero cohomology. 
Now (1) gives $H^j(C)=0$ for every $j$.
\end{proof}

\begin{thm}\label{Main4}
Let $(V,tV)$ be a DVR with fraction field $L$ and residue field $K$, and let $\Delta$ be a finite simplicial complex on $\{v_0,v_1,\ldots,v_n\}$. Put
\[
 R=\RV,\qquad\m=(t,x_1,\ldots,x_n)R.
\]
For $\ua=(a_1,\ldots,a_n)\in\Z^n$, set
\[
 G_\ua=\{v_i:a_i<0\},\qquad \La_\ua=\link_\Delta G_\ua.
\]
On $\La_\ua$, let $\sF_\ua$ be the cellular sheaf with values
\[
 \sF_\ua(F)=
 \begin{cases}
 L,&v_0\in F,\\
 V,&F\in\link_{\La_\ua}\{v_0\},\\
 K,&F\notin\st_{\La_\ua}\{v_0\},
 \end{cases}
\]
and with restriction maps the identities within each type, $V\inj L$, and 
$V\surj K$. Then
\[
 [H^i_\m(R)]_\ua\cong
 \begin{cases}
 \tH^{i-|G_\ua|-1}(\La_\ua;\sF_\ua),&\ua\leq\Zo,\\
 0,&\text{some }a_i>0.
 \end{cases}
\]
If $G_\ua\notin\Delta$, the right-hand side is understood to be zero. In particular, 
writing $\sF_G$ for the corresponding sheaf on $\link_\Delta G$,
\[
 H^i_\m(R)=0
 \quad\Longleftrightarrow\quad
 \tH^{i-|G|-1}(\link_\Delta G;\sF_G)=0
 \quad\text{for every }G\in\Delta-v_0.
\]
\end{thm}
\begin{proof}
Let $\check C^\bullet$ be the \v{C}ech complex for $t,x_1,\ldots,x_n$, and put 
$H_\ua=\{v_i:a_i>0\}$. Proposition~\ref{prop-cech-vs-sheaf} identifies its degree $\ua$ component 
with the shifted cellular complex on
$\link_{\st_\Delta H_\ua}G_\ua$,
with shift $[-|G_\ua|-1]$. If $H_\ua\ne\nl$, Proposition~\ref{vanishing} 
makes its cohomology zero. If $H_\ua=\nl$, then $\st_\Delta\nl=\Delta$, so 
Proposition~\ref{prop-cech-vs-sheaf} gives 
the displayed formula on $\La_\ua=\link_\Delta G_\ua$.

A graded module is zero if and only if all its graded components are zero. 
Every face $G\in\Delta-v_0$ occurs as a negative support: take $a_i=-1$ for 
$v_i\in G$ and $a_i=0$ otherwise. This proves the final equivalence.
\end{proof}

\begin{rmk}
The sheaf $\scr{F}_{\Zn{a}}$ is {\it not} the same as the restriction of the sheaf $\scr{F}$ on $\Delta$ to $\Lambda_{\Zn{a}}$; it is very important here that we calculate the sheaf values on simplices by looking at the link and star of $\vertex{0}$ {\it inside of} $\Lambda_{\Zn{a}}$, not in $\Delta$. In particular, this means that if $\vertex{0}$ is not in the link of $G_{\Zn{a}}$, then the sheaf $\scr{F}_{\Zn{a}}$ on $\Lambda_{\Zn{a}}$ is the constant $K$-valued sheaf, because the link is the empty simplicial complex $\Nl$, 
    and does not contain even the empty face.
\end{rmk}

\begin{examp}\label{RP2}
    We return to the running example of a triangulation of the real projective plane. Let $\Delta$ be the simplicial complex depicted in the diagram for Example \ref{ex-betti-P2}. To compute the local cohomology modules of $\tSRrcomp{\Delta}$, we must examine the links of faces that do not contain $\vertex{0}$. Throughout the following discussion, we use $\scr{F}$ to refer to the sheaf on $\link_\Delta G$ defined as in Theorem \ref{Main4}.

\ben
   \item There are five 2-faces in $\Delta - \vertex{0}$. If $G$ is such a 2-face, then $\link_\Delta G = \{\nl\}$. Notice that $\scr{F}(\nl) = K$ in this case, since $\vertex{0} \notin \link_\Delta G$. It follows that $\simpHomol^{-1}(\link_\Delta G; \scr{F}) = K$, and at all other indices the sheaf cohomology vanishes.

   \item The link of any 1-face in $\Delta$ is a pair of disjoint vertices. There are five 1-faces in $\Delta - \vertex{0}$ for which $\link_\Delta G$ does \emph{not} include $\vertex{0}$. For such $G$, we have
   \[
      \simpHomol^{i}(\link_\Delta G; \scr{F}) = \simpHomol^{i}(\link_\Delta G; K) \cong \begin{cases}
          K &\text{if } i = 0,\\
          0 &\text{else.}
      \end{cases}
   \]

   \item For the remaining five 1-faces of $\Delta-\vertex{0}$, the link \textit{does} contain $\vertex{0}$. In this case, $\simpchain^\bullet(\link_\Delta G; \scr{F})$ is the cochain complex
   \[
      0 \to V \xrightarrow{\begin{pmatrix}
          \iota \\ \pi
      \end{pmatrix}} L \oplus K \to 0
   \]
It is straightforward to check that the cohomology of this complex is isomorphic to $L/V$ at index $0$ and vanishes everywhere else. 

   \item \noindent 
\begin{minipage}[t]{0.75\linewidth}
   For any vertex of $\Delta - \vertex{0}$, the link has the 
   form shown on the right. The colors indicate the values of 
   the sheaf: blue for $L$-valued faces, yellow for $V$-valued faces, and red for 
   $K$-valued faces. The face $\nl$ is $V$-valued. The cochain complex is: 
\end{minipage}%
\hfill
\begin{minipage}[t]{0.23\linewidth}   

\tikzset{every picture/.style={line width=0.75pt}} 

\begin{tikzpicture}[x=0.75pt,y=0.75pt,yscale=-1,xscale=1, baseline={([yshift=-3.0ex]current bounding box.north)}] 

\draw [color={rgb, 255:red, 208; green, 2; blue, 27 }  ,draw opacity=1 ][line width=3]    (125,95) -- (115,65) ;
\draw [color={rgb, 255:red, 208; green, 2; blue, 27 }  ,draw opacity=1 ][line width=3]    (155,95) -- (125,95) ;
\draw [color={rgb, 255:red, 208; green, 2; blue, 27 }  ,draw opacity=1 ][line width=3]    (155,95) -- (165,65) ; 
\draw [color={rgb, 255:red, 74; green, 144; blue, 226 }  ,draw opacity=1 ][line width=3]    (115,65) -- (140,45) ; 
\draw [color={rgb, 255:red, 74; green, 144; blue, 226 }  ,draw opacity=1 ][line width=3]    (165,65) -- (140,45) ;
\draw  [fill={rgb, 255:red, 74; green, 144; blue, 226 }  ,fill opacity=1 ] (135,45) .. controls (135,42.24) and (137.24,40) .. (140,40) .. controls (142.76,40) and (145,42.24) .. (145,45) .. controls (145,47.76) and (142.76,50) .. (140,50) .. controls (137.24,50) and (135,47.76) .. (135,45) -- cycle ;
\draw  [fill={rgb, 255:red, 248; green, 231; blue, 28 }  ,fill opacity=1 ] (110,65) .. controls (110,62.24) and (112.24,60) .. (115,60) .. controls (117.76,60) and (120,62.24) .. (120,65) .. controls (120,67.76) and (117.76,70) .. (115,70) .. controls (112.24,70) and (110,67.76) .. (110,65) -- cycle ; 
\draw  [fill={rgb, 255:red, 248; green, 231; blue, 28 }  ,fill opacity=1 ] (160,65) .. controls (160,62.24) and (162.24,60) .. (165,60) .. controls (167.76,60) and (170,62.24) .. (170,65) .. controls (170,67.76) and (167.76,70) .. (165,70) .. controls (162.24,70) and (160,67.76) .. (160,65) -- cycle ; 
\draw  [fill={rgb, 255:red, 208; green, 2; blue, 27 }  ,fill opacity=1 ] (120,95) .. controls (120,92.24) and (122.24,90) .. (125,90) .. controls (127.76,90) and (130,92.24) .. (130,95) .. controls (130,97.76) and (127.76,100) .. (125,100) .. controls (122.24,100) and (120,97.76) .. (120,95) -- cycle ; 
\draw  [fill={rgb, 255:red, 208; green, 2; blue, 27 }  ,fill opacity=1 ] (150,95) .. controls (150,92.24) and (152.24,90) .. (155,90) .. controls (157.76,90) and (160,92.24) .. (160,95) .. controls (160,97.76) and (157.76,100) .. (155,100) .. controls (152.24,100) and (150,97.76) .. (150,95) -- cycle ;
\end{tikzpicture} 
\end{minipage}%

   \[
   0 \to V \xrightarrow{\begin{pmatrix}
       \iota \\ 1 \\ 1 \\ \pi \\ \pi
   \end{pmatrix}} L \oplus V^2 \oplus K^2 \xrightarrow{\begin{pmatrix}
       -1 & \iota & 0 & 0 & 0\\
       -1 & 0 & \iota & 0 & 0 \\
       0 & 0 & 0 & -1 & 1\\
       0 & \pi & 0 & -1 & 0\\
       0 & 0 & \pi & 0 & -1
   \end{pmatrix}} L^2 \oplus K^3 \to 0
   \]
   It is straightforward to compute directly that $\simpHomol^q(\link_\Delta \vertex{i}; \scr{F}) \cong L/V$ for $q=1$ and vanishes otherwise.

\item Finally, we must compute the cellular sheaf cohomology of $\Delta$ itself. 
For this calculation, put $\Delta' =\Delta-v_0$ and $\La=\link_\Delta v_0$. The space $|\Delta'|$ is a M\"obius strip and $|\La|$ is its boundary circle. The reduced cohomology of 
$\La$ over $V$ or $L$ is concentrated in degree $1$, where it is $V$ or $L$, respectively. The relative cohomology is 
$$\tH^j(\Delta',\La;K)\cong \begin{cases}
 K,&j=1,2\text{ and }\operatorname{char}K=2,\\
 0,&\text{otherwise}.\end{cases}$$
 
 To verify this without a matrix calculation, retract the M\"obius strip $|\Delta'|$ onto its core circle. Its boundary circle $|\La|$ winds twice around the core. Thus, after choosing generators, the restriction $H^1(\Delta';K)\to H^1(\La;K)$ is multiplication by $2$. Both spaces are connected and have zero cohomology above degree $1$, so the long exact sequence of the pair reduces to
\[0\longrightarrow\tH^1(\Delta',\La;K)\longrightarrow K
\xrightarrow{\,2\,}K\longrightarrow\tH^2(\Delta',\La;K)
\longrightarrow0.
\]
All other relative cohomology groups vanish. This proves the displayed formula.

If $\op{char}K\ne2$, the only nonzero $\alpha^j$ is $\alpha^1:V \inj L.$
Equations (\ref{eq-SES-XYZ})--(\ref{eq-alpha-SES}) give: 

$$\tH^j(\Delta;\sF)\cong \begin{cases}L/V&\text{if $\charac K \neq 2$ and }j=2,\\0&\text{if $\charac K \neq 2$ and }j\ne2.\end{cases}$$

{\emergencystretch=2em
If $\op{char}K=2$, the map
$\partial^1:\tH^1(\La;\,K)\to\tH^2(\Delta',\, \La;\,K)$
is an isomorphism: the restriction $\Homol^1(\Delta';\,K)\to \Homol^1(\La;\,K)$ is 
multiplication by $2$ after choosing generators, because the boundary circle goes twice 
around the core circle of the M\"obius strip. In characteristic $2$ this restriction 
is zero, while $H^2(\Delta';\,K)=0$, so that the long exact sequence of the pair proves 
the assertion. For compatible generators we have
$\alpha^1=(\iota,\pi):V\to L\oplus K, \qquad v\longmapsto(v,\bar v).$
\par}
It is injective, and the relevant segment of (\ref{eq-LES-XYZ}) is
$$0\to K\to\tH^1(\Delta;\sF)
 \to V\arwf{(\iota,\pi)}L\oplus K
 \to\tH^2(\Delta;\sF)\to 0.$$
Consequently, if $\charac K = 2$, we have
\[
 \tH^1(\Delta;\sF)\cong K,\qquad
 \tH^2(\Delta;\sF)\cong(L\oplus K)/(\iota,\pi)(V)\cong L/tV.
\]
For the last isomorphism, send $(\ell,\bar a)$ to $\ell-a+tV$, with $a\in V$ any lift of $\bar a$. A different lift changes $\ell-a$ by an element of $tV$; the map is surjective, and its kernel is exactly $(\iota,\pi)(V)$. Multiplication by $t^{-1}$ induces a $V$-module isomorphism $L/tV\cong L/V$, so regardless of the residual characteristic we always have $\simpHomol^2(\Delta;\scr{F}) \cong L/V.$ 
\footnote{Note that in equal characteristic, so that $\widehat{V} \cong K[[t]]$, one can introduce a $\Z^{n+1}$ grading
on $H^{\bullet}_{\m}(R)$ that takes
account of degree in $t$.  $L/V$ and $L/tV$ differ in this grading:  their socle generators 
have degree $-1$ and $0$, in $t$.}  All other reduced cohomology groups vanish. 
\een

The following table shows the data for the cases where $\simpHomol^{i - |G|-1}(\link_\Delta G; \scr{F}_G)$ 
is nonzero for faces $G \in \Delta-\vertex{0}$.  

\begin{table}[H]
        \centering

\begin{tabular}{|m{0.15\textwidth}|m{0.12\textwidth}|m{0.12\textwidth}|m{0.12\textwidth}|m{0.12\textwidth}|}
\hline 
 $\displaystyle G\ \in \Delta -v_{0}$ & $\displaystyle i=0$ & $\displaystyle i=1$ & $\displaystyle i=2$ & $\displaystyle i=3$ \\
\hline 
 2-face &  &  &  & $\displaystyle K$ \\
\hline 
 1-face with $\displaystyle v_{0} \notin \link\ G$ &  &  &  & $\displaystyle K$ \\
\hline 
 1-face with $\displaystyle v_{0} \in \link\ G$ &  &  &  & $\displaystyle L/V$ \\
\hline 
 vertex &  &  &  & $\displaystyle L/V$ \\
\hline 
 $\displaystyle \nl $ &  &  & {\raggedright $\displaystyle K$ if $\charac K = 2$\par} & $\displaystyle L/V$\\
 \hline
\end{tabular} 

\end{table}

\end{examp}

\subsection{An alternate approach to local cohomology of \texorpdfstring{\boldmath$t$}{t}-Stanley-Reisner rings}\label{altcohom}

In this subsection we first describe the local cohomology of \tSR rings using exact sequences. We then use cellular cohomology to resolve extension information that these sequences alone do not determine. The next subsection develops a spectral sequence for cellular cohomology. In Theorem~\ref{2seqs} we describe two exact sequences that 
carry a great deal of information about the 
 behavior of the local cohomology  of a \tSR ring by relating it to cohomology that is more fully understood. 
The relative cohomology of the geometric realizations of a pair of simplicial complexes plays a role. 

We then use these ideas to investigate some examples, including our running example of the real projective plane. Throughout, we use the following:

\begin{notation}\label{lcterm} Let $(V,tV)$ be a DVR with residue class field $K$ and fraction field $L$, and let $\Delta$ be a finite abstract simplicial complex with vertex set $\{v_0, \vect v n\}$. Let $\La:= \link_{\Delta}v_0$, and let $C$ denote the cone over $\Lambda$ with vertex $v_0$, whose facets are the unions of $\{v_0\}$ and facets of $\La$.  Note that $C$ is the closed star of $v_0$ in $\Delta$, and, in particular, is a subcomplex of $\Delta$. Let $\Delta' : = \Delta-v_0$ denote the largest simplicial complex contained in $\Delta$ with vertices in $\{\vect v n\}$.

As usual, let $x_0$ and $\ux: = \vect xn$ be indeterminates corresponding to $v_0$ and $\{\vect vn\}$, and let $R$ denote the \tSR ring $\VD$.  Let $\m := (t,\,\ux)$ be the homogeneous maximal ideal of $R$. Let $J$ be the $V$-submodule of $R$ generated by all monomials in $\ux$ whose support is in $\Delta\sm C$. These monomials are killed by $t$, so $J$ naturally carries a $K$-vector space structure.
\end{notation}  

\begin{facts}\label{facts-Del'-La}We make the following observations:
    \benn
       \item \label{C-La} $C \cup (\Delta') = \Delta$ and $C \cap (\Delta') = \La$.

       \item \label{VCVL} $V\ov{[C]} \cong V[\La].$
       \item \label{JAnn} $J = \Ann_R t$.

       \item \label{Jcoh} By abuse of notation, $J$ may be considered an ideal of $K[\Delta']$, where this ring is identified with $(V/tV)[\Delta']$ so that $t$ kills $J$. Therefore $H^{\bullet}_{\m}(J) \cong H^{\bullet}_{(\ux)}(J)$, where the local cohomology module on the left is calculated over $R$, and the one on the right may be calculated over either $R$ or over $K[\ux]$.
    \een
\end{facts}

The proof of the following result is completely straightforward and is omitted.

\begin{thm}\label{J} With notation as in (\ref{lcterm}):
\bena
\item There is a short exact sequence $0 \to J \to R \arwf{\pi} V[\La] \to 0$,
where $\pi$ is a surjective map of $V$-algebras. 
\item The map $\pi: R \surj V[\La]$ has a splitting
$\eta:V[\La] \inj R$ \emph{as a map of $V$-modules}.
Both $V[\La]$ and the image of $\eta$ are free $V$-modules with the monomials in $\ux$ that have support in $\La$ as a free basis. 
\item There is a short exact sequence $0 \to J \to K[\Delta'] \arwf{\rho} K[\La] \to 0$,
where $J$ is regarded as an ideal of $K[\Delta']$ 
and $\rho$ is a surjective homomorphism of 
$K$-algebras. 
\een
\end{thm}

\begin{remark} The $V$-module splitting $\eta$ need not be a ring homomorphism. The product of two 
monomials with support in $\La$ may have support that is not in $\Lambda$ but is
in $\Delta$. This product will vanish in $V[\Lambda]$ 
but will be a nonzero element of $J$ in $R=V\ov{[\Delta]}$. \end{remark} 

\begin{rmk}\label{LCaddt}  For any Noetherian ring $R$, ideal $I\subseteq R$, element $t\in R$, and $R$-module $M$, there is a long exact sequence\footnote{One may prove this using the fact that if  $\uf:=\vect f k$ and $I$ is the radical of $(\uf)R$, the modified \Cech complex for $H^{\bullet}_{(\uf,t)}(M)$  is the mapping cone of $\cC^{\bullet} \to \cC^{\bullet}_t$, shifted by $[-1]$, where $\cC^{\bullet}$ is the modified \Cech complex for  $H^{\bullet}_{(\uf)}(M)$.}\ \  $\cdots \to H^{j-1}_I(M)_t \to H^j_{(I,t)}(M) \to H^j_I(M) \to H^j_I(M)_t \to \cdots$.
\end{rmk}

We can now state one of the main results of this subsection. The proof
is quite simple.

\begin{thm}\label{2seqs} Let notation be as in \ref{lcterm}. Then there are long exact cohomological sequences that preserve the $\Z^n$-grading as follows:
\benn
\item $\cdots\to H^{j-1}_{\m}(V[\La]) \to H^j_{(\ux)}(J) \to H^j_{\m}(\RV) \to H^j_{\m}(V[\La])  \to\cdots$, and 
\item $\cdots \to H^{j-1}_{(\ux)}(L[\La]) \to H^j_{\m}(V[\La]) \to H^j_{(\ux)}(V[\La]) 
\to H^j_{(\ux)}(L[\La])\to \cdots\,.$ 
\een
Note that in (2), we have natural identifications $H^{\bullet}_{(\ux)}(V[\La])_t \cong H^{\bullet}_{(\ux)}(V[\La]_t) \cong H^{\bullet}_{(\ux)}(L[\La])$.  

There are corresponding
exact sequences of graded components indexed by $\ua$ for all $\ua \in \Z^n$.
\end{thm}  
\begin{proof} Notice again that when $t$ kills a ring or module, it may
be replaced by 0 and so omitted from the generators of the 
support ideal of the local cohomology. Also, recall Fact \ref{facts-Del'-La}~\ref{Jcoh}. Hence (1) is the long exact sequence for local cohomology with support in $\m$ arising from  the short exact sequence in Theorem~\ref{J}, part (a). The sequence in (2) follows from the standard result in Remark~\ref{LCaddt}. \end{proof} 

In the sequel we shall see in examples how these sequences along with Theorem~\ref{cohJ} provide
a great deal of information about the local cohomology modules $H^{\bullet}_{\m}(V\ov{[\Delta]})$
of $t$-Stanley-Reisner rings. 

\subsubsection{The case of compact manifolds with boundary}\label{subsubsec-manifold} 
We continue the discussion of the results of \S\ref{altcohom}
in an important special case that affords a number of simplifications and includes the real projective plane as an example.  

We focus here on the case where  $\Delta'$ and $\Lambda$ as in Notation~\ref{lcterm}
satisfy a combinatorial condition weaker than being compact manifolds with boundary, 
introduced in \S\ref{homby}.  For background in piecewise-linear topology, 
see \cite{RoSa72, Mun84b}.

\subsubsection{Homology manifolds with boundary}\label{homby} 

To keep our treatment purely combinatorial and as general as possible, we work with
the class of examples such that $\Delta'$ and $\Lambda$  are {\it homology manifolds 
with boundary} in the sense of \cite[p.\,655]{KlNo16} over an appropriate base, although 
we generalize by allowing the base ring to be
arbitrary instead of restricting it to be a field or $\Z$. 

\begin{defn}
Let $\Sigma$ be a finite
abstract simplicial complex of dimension $n$ and let $A$ be a nonzero base ring. We say that 
$\Sigma$ is an $A$-{\it homology sphere} if for every face $F \in \Sigma$, including 
the empty face, $\link_{\Sigma} F$ has the same reduced homology as a sphere of 
dimension $n-|F|$. We say that  $\Sigma$ is an $A$-{\it homology manifold} if
for every \emph{nonempty} face $F \in \Sigma$, $\link_{\Sigma}F$  is an $A$-homology sphere of dimension $n-|F|$.

Similarly, we define $\Sigma$ of dimension $n$ to be a {\it homology} $A${\it -manifold 
with boundary} if the following two conditions\footnote{These conditions coincide with
the topological condition of having the geometric realization be a homology manifold 
with boundary. In the early topology literature, following Wilder \cite{Wil79}, who introduced 
homology manifolds, the term {\it generalized} manifold is used instead of {\it homology 
manifold}.  For several  decades, in dealing with homology manifolds with boundary, writers 
assumed both conditions (1) and (2) (see, for example,
\cite{Ray60}) in the definition. Both 
conditions are assumed as well in \cite{KlNo16}. However, even in the more general topological 
case, one has, with some restrictions, that $(1) \imp (2)$: see \cite{Mit90}, particularly the 
discussion at the top of p.~510.} hold:

\benn
\item For every \emp{nonempty} face $F \in \Sigma$, the reduced homology over $A$ of 
$\link_\Sigma F$ is the same as the reduced homology over $A$ either of a sphere or 
a ball of dimension $n-|F|$.  

\item The boundary $\partial\Sigma$ either equals $\{\nl\}$ or is an $A$-homology manifold of dimension $n-1$, where $\partial\Sigma$ is defined as the set of faces\footnote{It is
easy to see that $\partial \Sigma$ is always a subcomplex of $\Sigma$.}
$F \in \Sigma$  such that the $\link_{\Sigma}F$ has the same reduced homology over 
$A$ as a ball of dimension $n - |F|$, together with the empty face.   
\een
\end{defn}

\begin{rmk} The conditions on reduced $A$-homology modules in the definitions above are 
evidently equivalent to the same conditions on the correspondingly indexed reduced 
$A$-cohomology modules. 
\end{rmk}

\begin{rmk}\label{uc} By the universal coefficient theorems, if one has a finite simplicial 
$A$-homology manifold with  boundary, it is also a $B$-homology manifold with boundary for 
every nonzero $A$-algebra $B$. This is a case where the reduced chain and cochain complexes of all relevant links are free of finite rank over $A$, and so is their (co)homology, so that the calculation of (co)homology commutes 
with every base change. 

Moreover, if $\Sigma$ is a triangulation of a topological manifold (with boundary), then 
$\Sigma$ is an $A$-homology manifold (with boundary) over nonzero base ring $A$.
\end{rmk}

Note that $\La \inc \Delta'$ completely determines
$\Delta$:  with $C$ denoting the cone over $\La$ with the vertex $v_0$, we have that 
$\Delta = C \cup \Delta'$ while $C \cap \Delta' = \Lambda$.  See Fact~\ref{facts-Del'-La}~\ref{C-La}. If $\Delta'$ is a $d$-dimensional $A$-homology manifold with boundary $\Lambda$, it follows from the Mayer--Vietoris sequences of the links that $\Delta$ is an $A$-homology manifold if and only if $\Lambda$ is an $A$-homology sphere of dimension $d-1$.

Our running example, in which $\Delta'$ is a triangulation of a M\"obius strip and $\Lambda$ gives a triangulation of its boundary, $S^1$, is a case  where $\Delta'$ is a manifold with boundary $\Lambda$, and so both are $A$-homology manifolds for every coefficient ring $A$.

We next want to examine the information about the components of $H^j_{\m}(R)$ that one obtains 
in the special case where $\Lambda$ and $\Delta'$, are homology manifolds with boundary over 
$V$ (one also then has that hypothesis over $L$ and $K$, by Remark~\ref{uc}). 
We shall then specialize further: first to the case where the boundary is a sphere and then to our running example of a triangulation of the real projective plane.  

In the quotient notation below, when $\La=\{\nl\}$ we interpret $\Delta'/\La$ as $|\Delta'|$ with a disjoint basepoint. Here $\La$ is never the void complex $\Nl$. This convention makes the relative-to-quotient cohomology identifications valid also in dimension zero.

We begin by considering what happens with each of the exact sequences of Theorem~\ref{2seqs}
when we consider a multi-graded component. The first long exact sequence of Theorem~\ref{2seqs} is:
\[
  \cdots\to [H^{j-1}_{\m}(V[\La])]_{\ua} \to [H^j_{(\ux)}(J)]_{\ua} \to [H^j_{\m}(R)]_{\ua} \to [H^j_{\m}(V[\La])]_{\ua}  \to\cdots
\]

We can get very precise information about the $H^j_{(\ux)}(J)$ terms from the homology manifold with boundary assumptions, via Thm. \ref{cohJ}. In fact, we have:

\begin{lemma}\label{HJ} 
Assume that $\Delta'$ is a $K$-homology manifold of dimension $d$ with vertices $\{\vect vn\}$
with boundary $\La$, which is a nonempty 
$K$-homology manifold of dimension $d-1$. 
As usual, let $J$ be the ideal of $K[\Delta']$ spanned by the monomials with support in 
$\Delta' - \La$.

 Assume that $\ua \in \Z^n$ is nonpositive and that $\Ga$
is the set of vertices indexed by integers corresponding to strictly negative coordinates 
of $\ua$. If $\ua \not= \Zo$,  then  $[H^j_{(\ux)}(J)]_{\ua}$ vanishes unless $j = d+1$ and 
$\Ga$ is a face of  $\Delta'$, in which case it is isomorphic with $K$.  

If $\ua = \Zo$, we have $[H^j_{(\ux)}(J)]_{\Zo} \cong \tH^{j-1}(\Delta', \, \Lambda;\, K) \cong \tH^{j-1}(\Delta'/\La; \, K)$.
\end{lemma}

\begin{proof}
By Theorem~\ref{cohJ}, we have $[H^j_{(\ux)}(J)]_{\ua} \cong \tH^{j-|G_{\ua}|-1}(\link_{\Delta'}G_{\ua},\,
\link_{\La} G_{\ua}; \, K)$,
(since we are considering $J$ as an ideal in a \SR-ring over $K$).
For the remainder of this argument we omit $K$ from our simplicial cohomology notation.  
There are four cases:
\benn
    \item If $\Ga$ is not a face of $\Delta'$, then $[H^j_{(\ux)}(J)]_{\ua}=0$ for all $j$.  

    \item If $\Ga$ is a face of $\Delta'$ but not of $\Lambda$, we simply have $[H^j_{(\ux)}(J)]_{\ua} \cong \tH^{j - |\Ga| - 1}(\link_{\Delta'}\Ga)$. Since $\Ga$ is not in the boundary,  the link is a $K$-homology sphere of dimension $d-|\Ga|$.  The reduced cohomology is isomorphic to $K$ when $j = d+1$ and vanishes at all other indices.

    \item If $\Ga$ is a nonempty face of $\Lambda$, then  its link in $\Delta'$ is a homology ball of dimension $d-|\Ga|$, while its link in $\Lambda$ is a homology sphere of dimension $d-|\Ga|-1$.  Since the reduced cohomology of the ball is 0 in all degrees, in the long exact sequence for  cohomology associated with the pair $(\link_{\Delta'}\Ga,\, \link_{\La} \Ga)$  (see \cite[\S3.1]{Hat02}) every third term vanishes. Hence, for every integer $k$ we have that $\tH^{k+1}(\link_{\Delta'}\Ga,\, \link_{\La} \Ga) \cong \tH^k(\link_{\La}\Ga)$. Since $\link_{\La}\Ga$ is a $K$-homology sphere of dimension $d-1-|G_a|$, the latter term is 0 except when $k = d-1-|\Ga|$.  Thus, the only value of $j$ such that $\tH^{j - |\Ga| - 1}(\link_{\Delta'}\Ga,\, \link_{\La}\Ga )$ is not zero occurs when $j -|\Ga|-1 = d-1-|\Ga|+1$, i.e., when $j = d+1$, in which case it is isomorphic to $K$.

    \item If $\Ga = \nl$, then $\ua = \Zo$, and the relevant links of $\Ga$ are $\Delta'$ and $\Lambda$. The stated result is immediate from Theorem~\ref{cohJ}.
    \een
 \end{proof}

The following result follows easily from known results in the literature. The \CM property for $K[\La]$ when  $\La$ is a sphere was used in Richard Stanley's breakthrough proof of the
Upper Bound Conjecture. We give a proof, since it is very short.

\begin{lemma}\label{manifold} Let $\Lambda$ be a simplicial complex with vertices $\{\vect vn\}$ 
that is an $A$-homology manifold of dimension $d-1$. For $\ua\in\Z^n$, let $\Ga$ be its negative support. If any coordinate of $\ua$ is positive, then $[H^j_{(\ux)}(A[\Lambda])]_{\ua}=0$ for every $j$.
If $\ua\leq\Zo$ and $\Ga\neq\nl$, then $[H^j_{(\ux)}(A[\Lambda])]_{\ua}=0$ unless $\Ga\in\La$ and $j=d$, in which case it is isomorphic to $A$.

For $\ua=\Zo$, we have$[H^j_{(\ux)}(A[\Lambda])]_{\Zo}\cong\tH^{j-1}(\La;A)$.

\end{lemma}
\begin{proof} By Theorem~\ref{thm-og-hochsform-localcohomver}, we need only consider the
case where $\ua$ is nonpositive and $\Ga \in \La$.  Moreover, since $\link_{\La}\Ga$ is an $A$-homology
sphere of dimension $d-1-|\Ga|$ when $\Ga$ is nonempty, it follows from Theorem~\ref{thm-og-hochsform-localcohomver} that
$[H^j_{(\ux)}(A[\Lambda])]_{\ua} \cong \tH^{j-1-|G_{\ua}|}(\link_{\La} \Ga;\,A) = 0$ unless $j = d$, in which case it is isomorphic to $A$.

When $\ua\leq\Zo$ and $G_{\ua}=\nl$, we have 
$[H^j_{(\ux)}(A[\Lambda])]_{\Zo}\cong\tH^{j-1}(\La;A)$.
\end{proof}

Consequently:  

\begin{cor}\label{manifbdry} With notation as in Lemma~\ref{manifold}, if $\La$ is 
a $V$-homology manifold of dimension $d-1$, then for all $j \in \Z$ we have a long exact sequence 
\[\cdots \to \tH^{j-2}(\La;\, L) \to [H^j_{\m}(V[\La])]_{\Zo} \to \tH^{j-1}(\La;\, V) \to \cdots\,.\]
If $\ua \not= \Zo$ is nonpositive and $G_{\ua} \in \La$, then for all integers $j$ except $j= d+1$,  $[H^j_{\m}(V[\La])]_{\ua} = 0$, while 
$[H_{\m}^{d+1}(V[\Lambda])]_{\ua} \cong L/V$.
\end{cor}
\begin{proof} The first statement is immediate from the exact sequence (2) in 
Theorem~\ref{2seqs} and the results of Lemma~\ref{manifold}. \smallskip

Now assume that $\ua \not= \Zo$ is nonpositive and that $G_{\ua} \in \La$. 
Then Lemma~\ref{manifold} and the exact sequence (2) of Theorem~\ref{2seqs} show that
$[H_{\m}^j(V[\Lambda])]_{\ua} =0$ when $j$ is neither $d$ nor $d+1$, while 
$[H_{\m}^d(V[\Lambda])]_{\ua}$
and  $[H_{\m}^{d+1}(V[\Lambda])]_{\ua}$ are isomorphic to the kernel and 
cokernel, \resp, of the inclusion map $V \inj L$.   
\end{proof}

\begin{examp}\label{LaSphere} Assume the hypothesis is the same as in Corollary~\ref{manifbdry}.  Assume,
in addition, that $\La$ is a $V$-homology sphere of dimension $d-1$.    
It is straightforward to show that in this case:
\[[H^j_{\mathfrak m}(V[\Lambda])]_{\ua}
\cong
\begin{cases}
L/V,
 &j=d+1,\quad \ua\leq \Zo,\quad G_{\ua}\in\Lambda,\\[1mm]
0,&\text{otherwise}.
\end{cases}\]

Here $\ua=\Zo$ is included, since $G_{\Zo}=\nl\in\Lambda$.
Evidently,  $V[\Lambda]$ is Cohen-Macaulay of dimension $d+1$. 
\end{examp}
\medskip

From Lemma~\ref{HJ}, Lemma~\ref{manifold}, Corollary~\ref{manifbdry}, and 
the exact sequence (1) in Theorem~\ref{2seqs} we have at once:

\begin{thm}\label{toptolc} Let $\Delta'$ of dimension $d$ be a $V$-homology manifold with boundary $\La\neq\Nl$,
which is a $V$-homology manifold of dimension $d-1$. Let $C$ be the simplicial cone over $\Lambda$ with a new vertex $v_0$, and set $\Delta:=C\cup\Delta'$. Let $R$ be the $t$-Stanley-Reisner ring $V\ov{[\Delta]}$. Then
$[H^j_{\m}(R)]_{\ua}=0$  unless all entries of $\ua$ are nonpositive and $G_{\ua} \in \Delta'$,
in which case:

\bena
\item If $G_{\ua} \in \Delta'$ is nonempty, then 
$[H^j_{\m}(R)]_{\ua} = 0$ if $j \not= d+1$. If $j=d+1$ then for $G_{\ua}\in\Delta'\sm\Lambda$ we have $[H^{d+1}_{\m}(R)]_{\ua}\cong K$,
while for $G_{\ua}\in\La$ there is an exact sequence $0 \to K \to [H_{\m}^{d+1}(R)]_{\ua} \to L/V \to 0$. 

\item If $\ua = \Zo$ then we have an exact sequence
\[\cdots \to \tH^{j-1}(\Delta'/\La; \, K) \to [H^j_{\m}(R)]_{\Zo} \to 
H^j_{\m}(V[\La])_{\Zo} \to \tH^j(\Delta'/\La; \, K)\to\cdots \] 
Thus, for any $j \in \Z$ such that $\tH^{j-1}(\Delta'/\La; \, K) = \tH^j(\Delta'/\La; \, K)
= 0$, we have $[H^j_{\m}(R)]_{\Zo} \cong H^j_{\m}(V[\La])_{\Zo}$.

In particular, if $\tH^{j-2}(\La;K)=\tH^{j-1}(\Delta;K)=0$, then $[H^j_{\m}(R)]_{\Zo}=0$. Indeed, the coefficient sequence $0\to V\arwf{t}V\to K\to0$ shows that multiplication by $t$ is surjective on $\tH^{j-2}(\La;V)$ and injective on $\tH^{j-1}(\La;V)$. The former module is zero by finite generation and Nakayama's lemma; the latter is torsion-free because $V$ is a DVR. Hence $[H^j_{\m}(V[\La])]_{\Zo}=0$ by Theorem~\ref{2seqs}(2). Also $\tH^{j-1}(\Delta'/\La;K)\cong\tH^{j-1}(\Delta;K)=0$, because the cone $C$ is contractible. 
\een
\end{thm} 

\begin{remark} Theorem~\ref{toptolc} goes a long way towards understanding the components of the local
cohomology of $t$-Stanley-Reisner rings in terms of topological invariants.  But it does not provide 
complete information because of the difficulty in understanding the homomorphisms in the
long exact sequences, which illustrates the advantage of using cellular cohomology. In 
particular, using the cellular sheaf cohomology approach, one can say more about whether a  
component is better described as $L/tV$ or $L/V$, which are isomorphic as $V$-modules.  
There is further discussion of this point in Example~\ref{LaSph}.

On the other hand, Theorem~\ref{toptolc} may be easier to use in certain cases in which one 
knows that a great many of the relevant simplicial cohomology modules vanish for topological 
reasons, especially when one knows that a great many of the topological cohomology modules vanish.\end{remark}

From Theorem~\ref{toptolc} we have at once:

\begin{cor} With the same hypotheses and notation as in Theorem~\ref{toptolc}, if $k \in \N$
and $\tH^{j-1}(\La; \, K) =$ $\tH^j(\Delta; \, K) = 0$ for $j < k-1$, the depth of $R$ on 
$\m$ is at least $k$. \end{cor}

\begin{examp}\label{LaSph} Let the hypotheses and notation be the same as 
in Theorem~\ref{toptolc}. \emp{Assume
in addition, that $\Delta'$ is connected and that $\La$ is a $V$-homology $(d-1)$-sphere.} 
We describe in some detail the
information about this case available from Theorem~\ref{toptolc}.  When that fails to give a complete answer, we work
out what further can be said using cellular sheaf cohomology, which illustrates the need for this method.  In some cases, we also need to use
more techniques from algebraic topology, like the notion of an orientable $K$-homology manifold
over $K$ (which is defined whenever $K$ is any commutative ring) and  
Poincar\'e-Lefschetz duality
for manifolds with boundary. We provide references for the additional topological material.
Some key results obtained are displayed boxed.

By Lemma~\ref{HJ} for $\ua\neq \Zo$, 
$[H_{(\ux)}^j(J)]_{\ua}\cong K$ if $j = d+1$, $\ua \leq \Zo$
and $G_{\ua} \in \Delta'$ and vanishes otherwise. If $\ua = \Zo$ then $[H_{(\ux)}^j(J)]_{\Zo} \cong \tH^{j-1}(\Delta',\Lambda;K) 
\cong\tH^{j-1}(\Delta'/\Lambda;K).$ 
Since $\Lambda$ is a $V$-homology sphere, Example~\ref{LaSphere} gives

\ben
\item $[H_{\m}^j(V[\La])]_{\ua}\cong
\begin{cases}
L/V,&j=d+1,\ \ua\leq0,\ G_{\ua}\in\Lambda,\\
0,&\text{otherwise}.
\end{cases}$ 
\een
This includes $\ua=\Zo$. Theorem~\ref{2seqs}(1) yields the sequence:
\benr
\item $\cdots \to H_{\m}^{j-1}(V[\La]) \to H_{(\ux)}^j(J) 
\to H_{\m}^j(R)\to H_{\m}^j(V[\La]) \to\cdots$ 
\een 
We now focus on the case where $\ua \neq \Zo$. As usual, we may assume also that $\ua \leq \Zo$ 
and that $G_{\ua} \in \Delta'$:  without both these conditions,  the component vanishes.
Therefore, we now assume these conditions and that $\ua\leq \Zo$. 
Then $[H_{\m}^j(R)]_{\ua} = 0$ unless $j=d+1$. When $j=d+1$, there are two distinct cases, 
depending on whether $G_{\ua}\in\Delta'\sm\Lambda$ or $\nl \neq G_{\ua} \in \La$.
 If $G_{\ua}\in\Delta'\sm\La$, then (1) gives $[H_{\m}^{d+1}(V[\La])]_{\ua}=0.$
Thus, (2) reduces to $ 0\to K \to [H_{\m}^{d+1}(R)]_{\ua}\to 0$ and 
$[H_{\m}^{d+1}(R)]_{\ua} \cong K\ $ if  
$\ua\leq \Zo,\ \nl \neq G_{\ua}\in\Delta'\sm\Lambda.$ 
This is the first case of Theorem~\ref{toptolc}(a): the $L/V$ quotient occurs only for faces in $\Lambda$. On the other hand, if $\nl\neq G_{\ua}\in\Lambda$ then (2) gives
\benr
\item $0\to K \to [H_{\m}^{d+1}(R)]_{\ua} \to L/V \to 0$. 
\een 
Theorem~\ref{toptolc} records this extension but does not completely identify it. 
The cellular sheaf description does identify it. Put $r=d-|G_{\ua}|$ and let
$B_G:=\op{link}_{\Delta'}G_{\ua}$ and $S_G:=\op{link}_{\La}G_{\ua}.$
Then $B_G$ is a $K$-homology $r$-ball and $S_G$ is its $K$-homology $(r-1)$-sphere boundary. 
Note that the relevant cellular sheaf complex has three pieces, namely
$\tC^\bullet(B_G,S_G;K)$,  $\tC^\bullet(S_G;L)[-1]$, and $\tC^\bullet(S_G;V).$
Their only relevant cohomology modules are, respectively, 
$K$ in degree $r$, $L$ in degree $r$, and $V$ in degree $r-1$.
The connecting map from the third term to the first two is, up to signs and choices of orientation,
$V\arwf{(\iota,\pi)}L\oplus K$, where $\iota:V\inj L$ and $\pi:V\surj K=V/tV$. 
Thus, $[H_{\m}^{d+1}(R)]_{\ua} \cong
\op{Coker}\bigl(V\arwf{(\iota,\pi)}L\oplus K\bigr)$. 
There is an explicit isomorphism 
$\displaystyle \frac{L\oplus K}{\{(v,\bar v):v\in V\}}\arwf{\ \cong\ }L/tV$
given by $(\ell,\bar a)\longmapsto \ell-a+tV$, where $a\in V$ is any lift of $\bar a$. Therefore
$[H_{\m}^{d+1}(R)]_{\ua}\cong L/tV$  if $\ua\leq \Zo$, and $\nl\neq G_{\ua} \in \La.$ 

After choosing the sign of the identification of the leftmost $K$, (3) is the natural sequence
$0\to V/tV \to L/tV\to L/V\to 0$, which is not split: $L/tV$ is $t$-divisible, 
whereas $K\oplus L/V$ is not. Multiplication by $t^{-1}$ gives an  abstract $V$-module isomorphism $L/tV\cong L/V$, but $L/tV$ is the more informative description because it records the embedded copy of $K$. Thus, for $\ua \neq \Zo$, 

\[ \boxed{[H_{\m}^{j}(R)]_{\ua}\cong \begin{cases}
 K,
 &j=d+1,\ \ua\leq \Zo,\
   G_{\ua}\in\Delta'\sm\Lambda,\\[1mm]
L/tV,
 &j=d+1,\ \ua\leq \Zo,\
   \nl\neq G_{\ua}\in\La,\\[1mm]
0,&\text{otherwise.}
\end{cases}}\]

We next consider the multidegree $\Zo$ components below the top homological degree.
Let $Q$ denote the quotient space $\Delta'/\La$. Theorem~\ref{toptolc} yields
$\cdots\to \tH^{j-1}(Q;K) \to [H_{\m}^j(R)]_{\Zo} \to [H_{\m}^j(V[\La])]_{\Zo} 
\to \tH^j(Q;K)\to\cdots$. 
By (1), the middle term involving $V[\La]$ vanishes unless $j=d+1$. Hence, for all $j\leq d$:
\benr
\item $[H_{\m}^j(R)]_{\Zo} \cong \tH^{j-1}(Q;K) = \tH^{j-1}(\Delta',\La;K).$ 
\een
These are the only possible local cohomology modules below the top dimension. Note that
they are concentrated entirely in multidegree $\Zo$,
they are finite-dimensional $K$-vector spaces, and
they are annihilated by $\m$. The cone $C$ is contractible, and, by excision, we have  
that $\tH^j(Q;K) \cong \tH^j(\Delta,C;K) \cong\tH^j(\Delta;K).$ Hence, we have that
$[H_{\m}^j(R)]_{\Zo} \cong \tH^{j-1}(\Delta;K)$ for $0\leq j\leq d.$ 

This has a useful interpretation: the obstruction to being Cohen-Macaulay is concentrated in 
multidegree $\Zo$,   and is measured by the reduced $K$-cohomology of the closed homology
manifold $\Delta$.

We next focus on the top degree component. Note that  
$\tH^d(\Delta',\Lambda;K) \cong \tH^d(Q;K).$ 

Theorem~\ref{toptolc}  gives an extension
\benr
\item $0\to \tH^d(Q;\,K) \to [H_{\m}^{d+1}(R)]_{\Zo} \to L/V \to 0$. 
\een
The cellular sheaf complex again gives more information. Let
$\partial:\tH^{d-1}(\La;\,K) \to \tH^d(\Delta',\La;\,K) $
be the connecting homomorphism of the pair, and define
$\beta:V=\tH^{d-1}(\Lambda;V)\arwf{\ \pi\ }\tH^{d-1}(\Lambda;K)\arwf{\ \partial\ }\tH^d(Q;\,K)$. 
Then, up to an overall sign that does not matter, we have that
$[H_{\m}^{d+1}(R)]_{\Zo} \cong 
\op{Coker} \left(V\arwf{(\iota,\beta)} L\oplus \tH^d(Q;\,K) \right).$ 
This completely describes the extension in (5).  
Since $\beta$ factors through $K$, let $\bar\beta:K\to \tH^d(Q;\,K)$
be the induced $K$-linear map. We may identify $\bar\beta$ with $\partial$. 
Its rank is at most one.  There are two cases: 

\[\boxed{[H_{\m}^{d+1}(R)]_{\Zo} \cong \begin{cases}   L/V\oplus \tH^d(Q;\,K) & \text{\ if\ } \bar\beta = 0 \\ L/tV\oplus \tH^d(Q;\,K)/\bar\beta(K) &  \text{\ if\ } \ \bar\beta \neq 0.
                                  \end{cases}}\]

If we think in terms of abstract  $V$-modules this can be summarized as
\benr
\item $[H_{\m}^{d+1}(R)]_{\Zo} \cong L/V\oplus K^{\,\dim_K \tH^d(Q;\,K)-\op{rank}\bar\beta}.$  
\een

We can say more.  If $\Delta'$ is 
$K$-orientable\footnote{See Spanier \cite[Ch.\,5B{\bf 3}\,p.\,278]{Spa95} or \cite[\S3.3]{Hat02} 
which treats the case of manifolds: the definitions for $A$-homology manifolds are identical. 
This notion of being $\Z$-orientable is the same as the usual notion for manifolds.
Every $(\Z/2\Z)$-manifold is $(\Z/2\Z)$-orientable, because $1 = -1$ mod 2 and so sign-reversing
loops like the twist in a M\"obius strip are algebraically undetectable.}
then $\tH^d(Q;\,K)\cong K$
and the connecting map
$\tH^{d-1}(\Lambda;K)\to H^d(\Delta',\Lambda;K)$
is an isomorphism. For $d\geq1$, this is dual to the homology boundary map sending the relative fundamental class to the boundary fundamental class; reduced homology is used when $d=1$. The case $d=0$ follows directly from the augmented cochain complex. Thus $\bar\beta$ is an isomorphism and
$[H_{\m}^{d+1}(R)]_{\Zo}\cong L/tV.$ 
If $\Delta'$ is not $K$-orientable, then
$\tH^d(Q;\,K)=0,$ and
$[H_{\m}^{d+1}(R)]_{\Zo}\cong L/V.$ 

Thus in either connected case, $[H_{\m}^{d+1}(R)]_{\Zo}\cong L/V$
as an abstract $V$-module, but in the orientable case the more natural model is the 
non-split extension $L/tV$. In residual characteristic $2$, a connected homology 
manifold is $K$-orientable, so the degree-zero component has the natural description $L/tV$. \smallskip

We now put things together to give the complete component formula.
The answer as a collection of abstract $V$-modules is: \smallskip

\[\boxed{[H_{\m}^j(R)]_{\ua} \cong
\begin{cases}
\tH^{j-1}(\Delta',\Lambda;K),
   &\ua=\Zo,\ 0\leq j\leq d,\\[2mm]
K,
   &j=d+1,\ \ua\neq0,\ \ua\leq \Zo,\
     G_{\ua}\in\Delta'\sm\Lambda,\\[2mm]
L/V,
   &j=d+1,\ \ua\leq \Zo,\
     G_{\ua}\in\Lambda,\\[2mm]
0,&\text{otherwise}.  
\end{cases}}\]
\smallskip

Of course, in the third line, $\ua=\Zo$ is included. For nonzero boundary degrees, and in the orientable degree-zero case, the more precise natural module is $L/tV$, not merely an unspecified 
copy of $L/V$. Equivalently, as a finely graded $V$-module, 
\begin{center} 
\[ \boxed{\Strut{20pt} H_{\m}^{d+1}(R)  \cong\,\, [H_{\m}^{d+1}(R)]_{\Zo}\,\,\,\, \mplus\,\,{\bigoplus_{\substack{\ua\leq \Zo,\ \ua\neq \Zo, \,\, G_{\ua}\in\Delta'\sm\Lambda}}} K \,\,\,
\mplus \,\,\, \bigoplus_{\substack{\ua\leq \Zo,\,\, \ua\neq \Zo,\ G_{\ua}\in\Lambda}} L/tV}\]
\end{center}
This is a description as a graded $V$-module. The multiplication maps by the $x_i$ between distinct multigraded pieces contain additional incidence information not recorded by Theorem~\ref{toptolc}.

We next comment on depth and the Cohen-Macaulay property.
Since $\dim R=d+1$, formula (4) shows that $R$ is generalized Cohen-Macaulay: 
every $H_{\m}^j(R)$ for $j<d+1$ is a finite-dimensional $K$-vector space concentrated 
in degree zero. Moreover:\medskip

\centerline{$R$ is Cohen-Macaulay $\iff \tH^j(\Delta',\Lambda;K)=0$ 
for every $j<d \iff \tH^j(\Delta;K)=0$ for every $j<d$.} 

If $\Delta'$ is $K$-orientable, 
Poincar\'e-Lefschetz duality \cite[Ch.\,6,\,Thm.\,20,\,p.\,298]{Spa95} 
identifies $H^{j-1}(\Delta',\Lambda;K) \cong H_{d-j+1}(\Delta';K)$.
Therefore, in that case $R$ is Cohen-Macaulay precisely when $\Delta'$ is a $K$-homology ball.
More explicitly, with the convention that the minimum of the empty set is $+\infty$, we have:
\medskip

\centerline{
$\depth_{\m}R=\min\left\{d+1,\,1+\min\{j:\tH^j(\Delta',\Lambda,\,K)\neq 0\} \right\}$
}

This completes Example~\ref{LaSph}.
\end{examp}

\begin{examp}\label{RP2again} 
We revisit the case of the real projective plane example, calculated in Example~\ref{RP2},
specializing Example~\ref{LaSph}.  Now, $d=2$, $\Delta'$ is a M\"obius strip, and 
$Q:=\Delta'/\La$ is homotopic to
the real projective plane. In residual characteristic $2$, we have
$[H_{\m}^{2}(R)]_{\Zo}\cong K$ and $[H_{\m}^{3}(R)]_{\Zo}\cong L/tV$,
whereas in residual characteristic different from $2$, the lower component vanishes and 
the top degree-zero component is $L/V$. This recovers the behavior found in Example~\ref{RP2}.
\end{examp}

\subsection{A spectral sequence approach to computing cellular sheaf cohomology}\label{spectral}

Although spectral sequences are not needed to analyze cellular sheaf cohomology in the
case of $t$-Stanley-Reisner rings, the methods here suggest a general approach
for calculating such cohomology. The technique we are about to discuss can be applied 
to any cellular sheaf with values in an abelian category on a finite abstract simplicial 
complex $\La \not= \Nl$. Additional  
structure on the values of $\sF$ plays no role in this theory: in a sense, the values might as
well be in the category of abelian groups. If, for example, the values are $V$-modules, 
where $V$ is an arbitrary ring, that structure is automatically preserved throughout the 
calculation.

The idea is to use the poset structure of $\La$ to filter the sheaf $\sF$. But if
the functor $\sF$ on $\La$ factors through a simpler poset, that may provide a
substantial advantage, making calculations easier.  This is the case in the application
to local cohomology of $t$-Stanley-Reisner rings. 

A poset with a minimum element is called \emph{ranked} if every saturated chain from the minimum to a fixed element has the same finite length. This length is the \emph{rank} of the element. The term \emph{height} is also used. If these ranks are bounded, their maximum is the \emph{rank} of the poset, denoted $\rank(\Phi)$. The face poset of a nonempty finite simplicial complex is ranked, with unique minimum $\nl$, and its rank is one more than its dimension.

Throughout the sequel, $\Phi$ denotes a ranked poset of finite rank with a minimum element.

Now consider a cellular sheaf $\sF$ on $\La$ taking values in the abelian category $\cC$.
Assume that $\sF$ factors as $\La\arwf{\alpha}\Phi\arwf{\sG}\cC$, where $\La$ is its face poset under inclusion, $\alpha$ is order-preserving, and $\sG$ is a functor. Put $\rho:=\rank(\Phi)$.

We can consider any cellular sheaf using this point of view by taking $\Phi$ to be the poset
$\La$ and the map $\alpha$ to be the identity map. However, in many cases, there is a 
choice of $\La \to \Phi$ that is more helpful. 

We abbreviate $\tC^\bullet:=\tC^\bullet(\La;\sF)$. Define a decreasing filtration $\fC$ degree by degree by
\[\ang{\tC^q}_i:=\bigoplus_{\substack{F\in\La,\ \dim F=q\\ \rank(\alpha(F))\geq i}}\sG(\alpha(F)).
\]

Since $\alpha$ is order-preserving, the coboundary preserves these subobjects, so they form subcomplexes. This is a finite filtration: $\ang{\tC^\bullet}_0=\tC^\bullet$ and $\ang{\tC^\bullet}_{\rho+1}=0$. Its spectral sequence is denoted $E_r^{i,j}$ for $r\geq0$. We use the standard cohomological convention of \cite[\S12.24, Tag 012K]{Stacks}:
\[E_0^{i,j}=\ang{\tC^{i+j}}_i/\ang{\tC^{i+j}}_{i+1}.
\]
Thus $i$ is the filtration degree and $i+j$ is the cochain degree; the reduced degree $-1$ is included.

It is easy to calculate $E_0$, which is the associated graded complex $\gr_{\fC}(\tC^{\bullet})$ for this filtration. At level $i$,  the faces $F \in \La$ that survive are precisely those such that $\alpha(F)$ has rank $i$. The induced coboundary retains only those cofaces whose images under $\alpha$ also have rank $i$. Comparable elements of equal rank are equal, so this coboundary does not mix distinct fibers of $\alpha$.  For $\phi\in\Phi$, let $\tC_\phi^q$ be the direct sum of the copies of $\sG(\phi)$ indexed by the $q$-faces $F$ with $\alpha(F)=\phi$, with coboundary retaining only cofaces in that same fiber. These are complexes in the associated graded object, not in general subcomplexes of $\tC^\bullet$. We have
\[E_0^{i,j}\cong\bigoplus_{\rank(\phi)=i}\tC_\phi^{i+j}.
\]

Given $\phi \in \Phi$, we can define two simplicial subcomplexes of $\La$, namely $\La_{\leq \phi}$,
which consists of all faces of $\La$ whose image in $\Phi$ is $\leq \phi$, and
$\La_{< \phi} \inc \La_{\leq \phi}$, consisting of all faces of $\La$ whose image
is $< \phi$.

There is a natural isomorphism of cochain complexes
\[ \tC_\phi^\bullet\cong\tC^\bullet(\La_{\leq\phi},\La_{<\phi};\sG(\phi)).
\]
Indeed, the relative cochains are supported exactly on faces with image $\phi$, and all maps within this fiber are identities on $\sG(\phi)$, with the usual incidence signs. Consequently,
\[ E_1^{i,j}\cong\bigoplus_{\rank(\phi)=i}\tH^{i+j}(\La_{\leq\phi},\La_{<\phi};\sG(\phi)).
\]
Only finitely many summands are nonzero, since $\La$ is finite. The theory of
spectral sequences defines a differential $d_r$ on $E_r$ that respects the bigrading, although
it shifts degrees. Specifically, $d_r:E^{i,j}_r \to E_r^{i+r,j-r+1}$. 
Taking the cohomology of $E_r$ with respect to $d_r$ gives $E_{r+1}$.   For fixed $n$,  
the terms on the $n^{\op{th}}$ diagonal defined by $i+j = n$, map to the terms on the $(n+1)\,$th diagonal, defined by $i+j = n+1$. Since only the columns $0\leq i\leq\rho$ can be nonzero, $d_r=0$ for $r>\rho$, and the double array is then stable up to canonical isomorphism. The double array one gets for all large $r$ is denoted $E_{\infty}$.  The terms $E_\infty^{i,j}$ for $i+j = n$ give a great deal of information about the cohomology of the original complex,  $\Homol^n\big(\tC^{\bullet}(\La;\,\sF)\big)$.  Its filtration is given by
$\op{Im}\bigl(\Homol^n(\ang{\tC^\bullet}_i)\to\Homol^n(\tC^\bullet)\bigr)$. The $i$th associated graded object is $E_\infty^{i,n-i}$. These graded objects need not determine the extensions needed to recover $\Homol^n(\tC^\bullet)$.

The cellular sheaves that arise in the study of the local cohomology of 
$t$-Stanley-Reisner rings take on only three values, $V$, $L$, and $K$.  
We may take $\Phi$ to be the poset $\{V,\,L,\,K\}$
with the morphisms $V \surj K$, $V \inj L$, and the identity maps. 
$V$ is the unique minimum 
element, and has rank 0,  while $L$, $K$ are maximal elements with rank 1. 
Here $\sG$ is essentially the identity map. The complexes 
$\tC^{\bullet}_L$, $\tC^{\bullet}_V$, and $\tC^{\bullet}_K$ correspond 
to the complexes  $X^{\bullet}$, $Y^{\bullet}$, and $Z^{\bullet}$ respectively, 
discussed in the proof of 
Proposition~\ref{prop-cech-vs-sheaf} and in 
Discussion~\ref{XYZ}. In this case
$\tC^{\bullet}(\La,\, \sF)$ fits into an exact sequence 
$0 \to \tC^{\bullet}_L \oplus \tC^{\bullet}_K \to \tC^{\bullet}(\La,\, \sF) \to \tC^{\bullet}_V \to 0$. The differential $d_1$ is the connecting homomorphism
$\Homol^q(\tC_V^\bullet)\to\Homol^{q+1}(\tC_L^\bullet\oplus\tC_K^\bullet)$
in the associated long exact cohomology sequence, and $E_2=E_\infty$. Thus the spectral sequence gives the same kernel and cokernel pieces as that long exact sequence. Recovering the middle cohomology modules still requires the extension information, as in Example~\ref{LaSph}.

\section{Serre conditions}
\label{sec-serre}
A commutative ring $R$ is said to have the Serre condition \Serre{k} if, for every prime $P$, the depth of $R_P$ is at least the minimum of $k$ and the height of $P$. 
Since depth can be detected by the vanishing of local cohomology modules, Hochster's theorem for local cohomology suggests that Serre conditions for a Stanley-Reisner ring could be checked topologically, by examining the singular homology of the links. Indeed, it is a theorem due to Naoki Terai that a Stanley-Reisner ring $K[\Delta]$ over a field $K$ has \Serre{k} if and only if the reduced simplicial homology of $\link_\Delta F$ vanishes in certain degrees for every face $F$ (see 
Thm.~\ref{thm-terai}). In this section, we will extend Terai's theorem first to Stanley-Reisner rings defined over arbitrary Noetherian rings, and then to $t$-Stanley-Reisner rings.

\subsection{Serre conditions for Stanley-Reisner rings} \label{subsec-SR-serre} 

Throughout \S\ref{subsec-SR-serre}, let $A$ be a Noetherian commutative ring and let $\Delta$ be a nonempty simplicial complex on vertices $\vertset$. The excluded case $\Delta=\Nl$ gives the zero ring $A[\Delta]$, which satisfies every Serre condition independently of $A$.

\begin{defn}
    \label{def-serre}
    Let $R$ be a commutative Noetherian ring, and let $k \in \N$. We say that $R$ satisfies the \emph{$k^{th}$ Serre condition}, or $R$ has \Serre{k}, if
    $\depth_{P} R_P \geq \min \{k, \height P\} \quad \forall P \in \Spec(R)$.

    More generally, a Noetherian $R$-module $M$ has \Serre{k} if $\depth_{P} M_P \geq 
    \min \{k, \,\dim(M_P) \}$ for all primes $P \in \op{Supp}(M)$.
\end{defn}

The following result \cite[Prop.\,6.4.1,\,Cor.\,6.4.2]{EGA} will be useful.

\begin{prop}\label{serrek} Let $A \to R$ be a faithfully flat homomorphism of Noetherian rings. If $R$ satisfies \Serre{k}, then $A$ satisfies \Serre{k}. Conversely, $R$ satisfies
\Serre{k} provided that both of the following conditions hold:
\benn
\item $A$ satisfies \Serre{k}.  
\item For every prime $\p \in \Spec(A)$, the fiber $\kappa_{\p} \otimes_A R$ satisfies \Serre{k}.
\een

Consequently a polynomial ring or Laurent polynomial ring over a Noetherian ring $R$ satisfies \Serre{k} if and only if  $R$ does.

\end{prop}

\begin{rmk}\label{rmk-serre-locus}
Suppose $R$ is the quotient of a Cohen-Macaulay ring $S$. It follows 
from \cite[Remarques 6.11.9(ii)]{EGA} and \cite[Prop 6.11.8]{EGA} that 
the \Serre{k} locus of $R$ as an $S$-module
   \[
      U_\text{\Serre{k}}(R) = \{P \in \Spec S \; : \; R_P \text{ has \Serre{k}}\}
   \]
   is Zariski open in $\Spec S$. Since $\Spec R \cong V(I)$ is closed in $\Spec S$, the restriction
   \[
       U_{\text{\Serre{k}}}(R) \cap \Spec R = \{P \in \Spec R \; : \; R_P \text{ has \Serre{k}}\}
   \]
   is open in $\Spec R$, and thus its complement, the \textit{non}-Serre locus of $R$, is closed.
\end{rmk} 

The following criterion is given in \cite[p.\,452]{terai}, based on a criterion of Terai and Yanagawa published in \cite[Cor 3.7]{yanagawa}:

\begin{thm}[\cite{terai}] \label{thm-terai}
    Let $k \geq 2$, let $K$ be a field, and let $\Delta$ be a simplicial complex. Then $K[\Delta]$ has \Serre{k} if and only if $\simpHomol_j(\link_\Delta F; K) = 0$ for all $j < \min\{k-1, \dim \Delta - \dim F - 1\}$ and $F \in \Delta$.
\end{thm}

Throughout the rest of this section unless otherwise specified, $A$ denotes a nonzero commutative 
Noetherian ring.

\begin{rmk}\label{rmk-serre-1}
    The \Serre{1} condition is equivalent to having no embedded primes; it follows that every reduced ring has \Serre{1}, including every Stanley-Reisner ring $K[\Delta]$ over a field. Over a commutative Noetherian ring $A$, the Stanley-Reisner ring $A[\Delta]$ has \Serre{1} if and only if $A$ itself has \Serre{1}. The $t$-Stanley-Reisner ring $\tSRrcomp{\Delta}$ is reduced and so  always has \Serre{1} (Proposition \ref{prop-tSRbasics} (e)). For the remainder of this section, we will concern ourselves only with \Serre{k} for $k \geq 2.$
\end{rmk}

The following result from \cite[Lemma 2.1,\,Thm.\,3.3(a)]{BaBr25} is very useful
in studying Serre conditions for $\Z^n$-graded rings. Note that the first three
sections of \cite{BaBr25} are primarily expository, based on earlier work in
\cite{GoWa78} and \cite{MaRo74}.

\begin{thm}\label{*} Let $R$ be a Noetherian $\Z^n$-graded ring, $M$ a \fg $\Z^n$-graded $R$-module, and $\p \in \op{Supp}M$ a prime ideal. Let $\p^*$ be the ideal generated by all $\Z^n$-homogeneous elements
in $\p$, which is prime. Then $\dim(M_\p) -\op{depth}(M_{\p}) = \dim(M_{\p^{*}}) - \op{depth}(M_{\p^{*}})$.
\end{thm}
            
\begin{cor}\label{p*} If $R$ is a Noetherian $\Z^n$-graded ring and $k \in \N$, 
then $R$ has \Serre{k} if and only if 
$\depth R_P\geq\min\{k,\height P\}$ for every $\Z^n$-graded prime $P$.
\end{cor}
\begin{proof}
Only sufficiency requires proof. Suppose that the depth inequality fails at a prime $\p$. Put $d=\dim R_\p$, $d^*=\dim R_{\p^*}$, $e=\depth R_\p$, and $e^*=\depth R_{\p^*}$. Since $\p^*\subseteq\p$, we have $d^*\leq d$. Theorem~\ref{*} gives $d-e=d^*-e^*$, so $e^*\leq e<k$ and $e^*<d^*$, the latter because $e<d$. Thus the depth inequality also fails at the graded prime $\p^*$, a contradiction.
\end{proof}

\begin{cor} Let $\cP$ be a property of Noetherian rings that is
local:  that is, a ring $T$ has $\cP$ if and only if $T_Q$ has $\cP$ for every $Q\in\Spec T$. 
Assume also that for every Noetherian $\Z^n$-graded ring $T$, it suffices to test $\cP$ 
at its $\Z^n$-graded primes.

Let $R$ be a Noetherian $\Z^n$-graded ring, let $P$ be a $\Z^n$-graded prime, and let $W_P$ be the multiplicative system of $\Z^n$-homogeneous elements of $R\sm P$. Then $R_P$ has $\cP$ if and only if $W_P^{-1}R$ has $\cP$.
\end{cor}

\begin{proof}
Set $S=W_P^{-1}R$. A local property is preserved under localization, so $S$ having $\cP$ implies that $R_P$ has $\cP$.
Conversely, suppose that $R_P$ has $\cP$. The ring $S$ is still $\Z^n$-graded. Every $\Z^n$-graded prime of 
$S$ is $W_P^{-1}Q$ for a $\Z^n$-graded prime $Q$ of $R$ disjoint from $W_P$. Such a $Q$ is contained in $P$: otherwise a homogeneous generator of $Q$ outside $P$ would belong to $Q\cap W_P$. For every such $Q$, the ring $S_{W_P^{-1}Q}\cong R_Q$
is a localization of $R_P$, and therefore has $\cP$. The $\Z^n$-graded-prime testing hypothesis 
now implies that $S$ has $\cP$.\end{proof}

\begin{remark}\label{Laur} Let $\p\in\Spec A$, let $F\in\Delta$, and put $R=A[\Delta]$ and $P_{\p,F}=\p R+(x_i:v_i\notin F)R$. Recall that $\Fx=\prod_{v_i\in F}x_i$. There is an isomorphism
\[
 R[1/\Fx]\cong A[\link_\Delta F][x_i^{\pm1}:v_i\in F].
\]
Indeed, after the variables indexed by $F$ 
are inverted, the non-face relations are precisely those of $\link_\Delta F$; the variables 
indexed by $F$ remain independent Laurent variables. 
See \cite[Lemma 2.1, Remark 2.2]{TeTr12}, \cite[Proof of Lemma 1.2]{TeYo09}, 
and \cite[\S6.3, p.\,224]{Vil15}, especially Exercises 6.3.56 and 6.3.57. 
This argument works over an arbitrary coefficient ring.

Write $R_\p=A_\p\otimes_A R$, so this notation localizes coefficients only. A nonzero multihomogeneous element of $R$ has the form $a\ux^{\ua}$; it is outside $P_{\p,F}$ exactly when $a\notin\p$ and its monomial support is contained in $F$. Consequently, if $W_{P_{\p,F}}$ denotes the homogeneous elements outside $P_{\p,F}$, then
\[ W_{P_{\p,F}}^{-1}R=R_\p[1/\Fx]
 \cong A_\p[\link_\Delta F][x_i^{\pm1}:v_i\in F].
\]
\end{remark}

\begin{cor} With notation as in Remark~\ref{Laur}, the following hold:
\benn
\item $R_{P_{\p,F}}$ satisfies 
\Serre{k} $\iff$ $R_\p[1/\Fx]$ satisfies \Serre{k} $\iff$ $A_\p[\link_\Delta F]$ satisfies \Serre{k}.
\item $\depth R_{P_{\p,F}}=\depth A_\p+\depth\kappa_\p[\link_\Delta F]$, where the last depth is taken at the ideal generated by the variables of the link.
\een
\end{cor}
\begin{proof}
Serre conditions are local and are detected at graded primes by Corollary~\ref{p*}. The preceding corollary and Remark~\ref{Laur} therefore yield the first equivalence in (1), while Proposition~\ref{serrek} removes the Laurent variables.

For (2), the map $A_\p \to R_{P_{\p,F}}$ is flat and local. Let $T=\kappa_\p(x_i:v_i\in F)$ 
and let $\n$ denote the ideal generated by the variables of $\link_\Delta F$. Its closed fiber is
${R_{P_{\p,F}}/\p R_{P_{\p,F}}  \cong T[\link_\Delta F}]_{\n}$.
The depth of this local ring is $\depth\kappa_\p[\link_\Delta F]$: extension of the coefficient field preserves vanishing of the graded local cohomology modules, for example by Hochster's formula. Proposition~\ref{flatdp} proves (2). The flat local dimension formula also gives
$\dim R_{P_{\p,F}}=\height\p+\dim\link_\Delta F+1$, as in Proposition~\ref{prop-ADeltabasics}(c).
\end{proof}

\begin{cor}\label{testlink} 
Assume that $k\geq 2$. Then $A[\Delta]$ has \Serre{k} if and only if $A$ has \Serre{k} and $\tH_j(\link_\Delta F; \kappa_{\p}) = 0$ for all $F \in \Delta$, $\p \in \Spec A$ with $\height \p < k$, and $j < \min\{k-1 - \height \p,\; \dim \Delta - \dim F - 1\}.$ \end{cor}
\begin{proof}
Let $R := A[\Delta]$. The map $A\to R$ is faithfully flat, since the monomials supported on faces form an 
$A$-basis that includes $1$. Thus $R$ having \Serre{k} implies that $A$ has \Serre{k} by Proposition~\ref{serrek}.

Assume now that $A$ has \Serre{k}. For $\p\in\Spec A$, write $h=\height\p$ and $d_F=\dim\link_\Delta F+1$. At the graded prime $P_{\p,F}$, the preceding corollary gives
\[\depth R_{P_{\p,F}}
 =\depth A_\p+\depth\kappa_\p[\link_\Delta F],
 \qquad \dim R_{P_{\p,F}}=h+d_F.
\]
If $h\geq k$, the required inequality follows from $\depth A_\p\geq k$. If $h<k$, then $\depth A_\p=h$, and it becomes
\[\depth\kappa_\p[\link_\Delta F]
 \geq\min\{k-h,d_F\}.
\]
As $F$ varies, these inequalities are equivalent to \Serre{k-h} for $\kappa_\p[\Delta]$: localizing a Stanley--Reisner ring at the prime complementary to $F$ has the displayed link depth and dimension, and Corollary~\ref{p*} detects Serre conditions at these primes. Consequently, $R$ has \Serre{k} if and only if $A$ has \Serre{k} and $\kappa_\p[\Delta]$ has \Serre{k-h} for every $\p$ with $h<k$.

For $k-h\geq2$, Theorem~\ref{thm-terai} translates this last condition exactly into the stated homology vanishing. To handle $k-h=1$, observe first that a minimal prime of the nonzero ring $A$ has $h=0$. In either direction of the proposed equivalence, the corresponding field criterion forces $\Delta$ to be pure: if $F$ were a facet of dimension less than $\dim\Delta$, then $\link_\Delta F=\{\nl\}$ and its reduced homology in degree $-1$ would violate that criterion. For a pure $\Delta$, the bounds with $k-h=1$ require only the automatic vanishing in degrees below $-1$, and the vanishing in degree $-1$ for nonempty links with a vertex. Thus they hold automatically, just as \Serre{1} holds for every Stanley--Reisner ring over a field. This proves both implications.
\end{proof}

\begin{remark}
The \Serre{k} condition for a Stanley--Reisner ring $\kappa[\Sigma]$ over a field depends only on the characteristic of $\kappa$, because its link cohomology is obtained from cohomology over the prime field by extension of scalars. To apply Corollary~\ref{testlink}, one must first check that $A$ has \Serre{k}. For each characteristic $c$ occurring among its residue fields, let
\[h_c=\min\{\height\p:\p\in\Spec A,\ \op{char}\kappa_\p=c\},
 \qquad
 \kappa_c=\begin{cases}\Q,&c=0,\\ \mathbb F_c,&c>0.\end{cases}
\]
Only characteristics with $h_c<k$ need be considered. For each of these, it suffices to test
\[\simpHomol^j(\link_\Delta F;\kappa_c)=0
 \quad\text{for }j<\min\{k-1-h_c,\dim\Delta-\dim F-1\}
\]
for every $F\in\Delta$, since this is the strongest range for that characteristic.
\end{remark}

\subsection{Serre conditions for \texorpdfstring{$t$}{t}-Stanley-Reisner rings} 

We give criteria for \tSR rings to satisfy \Serre{k}, analogous to Terai's
criterion for \SR rings over a field. 

\begin{notation}\label{not5} Throughout this subsection, as usual, let
$(V,\,tV,\, K)$ be a DVR with fraction field $L$, let
$\Delta \neq \Nl$ be a finite
abstract simplicial complex with vertices in $\{v_0, \, \vect v n\}$, and 
let $\ux = \vect xn$.  Let $R := V\ov{[\Delta]}$ be a \tSR ring and let $\m$ 
denote the maximal ideal $(t, \ux)R$.  We use $\Gamma$ to denote a variable 
finite abstract simplicial complex, and $\cK$ to denote a variable field. \par

In these definitions it will be convenient to let 
$y_0 :=t$ and $y_i:= x_i$, $1 \leq i \leq n$.
Recall the notation $\tmonomF{F} = \phi(\prod_{v_i \in F} x_i)$ and $\ov{P_F} = (\phi(x_i) : v_i \notin F)R$, where $\phi$ is evaluation $x_0 \mapsto t$.
By Proposition~\ref{prop-tSRbasics}(b),  the $\ov{P_F}$ are precisely the \mlt prime ideals of $R$.
Moreover, we let $\Gamma_F :=\link_{\Delta}F$. By Proposition~\ref{prop-tSRbasics}, we have 
$\height{\ov{P_F}} = \dim(\Gamma_F)+1$.

Given any local, or \SR, or \tSR ring $S$, we use $\depth(S)$ to denote the depth of
$S$ on its maximal or homogeneous maximal ideal.
\end{notation}

We need the following depth calculation:

\begin{prop} Let notation be as in (\ref{not5}). Then: $\depth R_{\ov{P_F}}=
\begin{cases}
\depth K[\Gamma_F],&v_0\notin F,\\[2mm]
\depth L[\Gamma_F],&v_0\in F.
\end{cases}
$
\end{prop}
\begin{proof} First invert the nonzero $t$-monomial $\tmonomF{F}$. If $G\notin\Delta$, then, 
after $\tmonomF{F}$ is inverted, $\tmonomF{G}$ is a unit multiple of $\tmonomF{G \sm  F}$. Conversely, if 
$H$ is disjoint from $F$ and $H\notin\Gamma_F$
then $F\cup H\notin\Delta$, and $\tmonomF{H}=\tmonomF{F\cup H}/\tmonomF{F}$
in the localization. It follows that the defining ideal after localization 
is generated precisely by the non-faces of  $\Gamma_F$.
There are three cases.

Case 1: $v_0\notin F$ and $F\cup\{v_0\}\in\Delta$.
Here $v_0$ is a vertex of $\Gamma_F$,  and
$(*) \quad R_{\tmonomF{F}} \cong V\ov{[\Gamma_F]} [x_i^{\pm1}:v_i\in F]$. 
Let $\n _F$ denote the homogeneous maximal ideal of
$V\ov{[\Gamma_F]}$. By $(*)$,   
$\ov{P_F}R_{\tmonomF{F}}= \n _FV\ov{[\Gamma_F]}[x_i^{\pm1}:v_i\in F].$
The map $V\ov{[\Gamma_F]}_{\n _F}\to R_{\ov{P_F}}$
is flat  local, and its closed fiber is
$K(x_i:v_i\in F)$. Consequently, the flat 
local depth formula from Proposition~\ref{flatdp} gives  
$\depth R_{\ov{P_F}} =\depth V\ov{[\Gamma_F]}_{\n _F}$.
By Proposition~\ref{prop-tSRbasics}(d),
$\depth V\ov{[\Gamma_F]}_{\n _F} = \depth K[\Gamma_F].$

Case 2: $v_0\notin F$ and $F\cup\{v_0\}\notin\Delta$.
In this case, $t\tmonomF{F}=0$
in $R$. Since $\tmonomF{F}$ becomes a unit, $t=0$ in $R_{\tmonomF{F}}$. Hence
$\quad R_{\tmonomF{F}} \cong K[\Gamma_F][x_i^{\pm1}:v_i\in F].$  
Localizing this isomorphism at the prime corresponding to $\ov{P_F}$, 
and again using  the flat local depth formula with field fiber, shows that
$\depth R_{\ov{P_F}}=\depth K[\Gamma_F].$
Therefore, both cases with $v_0\notin F$ give the asserted depth.

Case 3: $v_0\in F$.
Then $R_{\tmonomF{F}} \cong L[\Gamma_F][x_i^{\pm1}:v_i\in F \sm \{v_0\}]$
because $\tmonomF{F}$ is divisible by $t$, so that $t$ becomes invertible. 
The corresponding flat local map from the localization of $L[\Gamma_F]$ has closed fiber
that is a field, and therefore $\depth R_{\ov{P_F}}=\depth L[\Gamma_F]$.  \end{proof}

We next use this result to characterize the \Serre{k} condition for \tSR rings in terms of the 
behavior of \SR rings, and
then give a topological version.

\begin{thm}\label{SkChar} For every $k\geq0$, $R$ satisfies \Serre{k}
if and only if the following two conditions hold:
\benn
\item  For every $F\in\Delta$ with $v_0\notin F$, 
$\depth K[\link_{\Delta}F] \geq \min\{k,\dim\link_{\Delta}F+1\}$.  

\item For every $F\in\Delta$ with $v_0 \in F$, $\depth L[\link_{\Delta}F]\geq
\min\{k,\dim\link_{\Delta}F+1\}.$  
\een
\end{thm}

\begin{proof} By Corollary~\ref{p*}, $R$ satisfies \Serre{k} if and only if the required depth inequality holds at every graded prime. The graded primes are exactly the $\ov{P_F}$.
The height equality $\height\ov{P_F}=\dim\link_\Delta F+1$ and the preceding proposition
show that the \Serre{k} inequality at $\ov{P_F}$ is precisely (1) or (2), according as $v_0\notin F$ or $v_0\in F$. \end{proof}

Consequently:

\begin{cor} The ring $R$ satisfies \Serre{k} if and only if the following two conditions hold:
\benn
\item For every prime $Q$ of $K[\Delta]$ containing $x_0$,
$
 \depth K[\Delta]_Q\geq\min\{k,\height Q\}.
$
\item $L[\link_\Delta\{v_0\}]$ satisfies \Serre{k}.
\een
\end{cor}
\begin{proof}
The graded primes of $K[\Delta]$ containing $x_0$ correspond exactly to faces $F$ omitting $v_0$. Their depths and heights are $\depth K[\link_\Delta F]$ and $\dim\link_\Delta F+1$. If a prime $Q$ containing $x_0$ fails the displayed inequality, its graded part $Q^*$ still contains $x_0$ and fails the inequality by the proof of Corollary~\ref{p*}. Thus (1) is equivalent to condition (1) of Theorem~\ref{SkChar}. Every face containing $v_0$ has the form $F=\{v_0\}\cup H$ with $H\in\link_\Delta\{v_0\}$, and
\[\link_{\link_\Delta\{v_0\}}H=\link_\Delta F.
\]
The graded-prime criterion over $L$ therefore identifies (2) with condition (2) of Theorem~\ref{SkChar}. If $v_0$ is not a vertex of $\Delta$, there are no such faces and the ring in (2) is the zero ring, so this condition is vacuous.
\end{proof}

\begin{lemma}\label{cKG} Let $\cK$ be a field, let $\Gamma \neq \Nl$ be a simplicial complex, and let $\n$ be the homogeneous maximal ideal of $\cK[\Gamma]$. For  $s \in \N$, the following are equivalent:
\benn
\item $\depth_{\n }\cK[\Gamma]\geq s$ and  

\item  for every  $H \in \Gamma$ and every 
$j < s - |H| - 1$, $\simpHomol^j(\link_{\Gamma} H; \, \cK)=0$. 
\een
\end{lemma}
\begin{proof}
The inequality in (1) is equivalent to $\Homol^i_\n(\cK[\Gamma])=0$ for every $i<s$. Hochster's formula says that the graded components with a positive coordinate vanish, and that, for $\ua\leq\Zo$ with negative support $H\in\Gamma$,
\[
 [\Homol^i_\n(\cK[\Gamma])]_\ua
 \cong\simpHomol^{i-|H|-1}(\link_\Gamma H;\cK).
\]
Components whose negative support is not a face also vanish. Thus (2) implies (1). Conversely, for $j\geq-1$ satisfying $j<s-|H|-1$, set $i=j+|H|+1$ and choose a multidegree with negative support $H$ and no positive coordinates. Then $0\leq i<s$, so (1) forces the 
required cohomology group to vanish. Degrees $j<-1$ vanish automatically, proving (2).
\end{proof}

We are now ready to prove the main result of this subsection, which characterizes the \Serre{k}
condition for \tSR rings, and which may reasonably be called the {\it $t$-Terai criterion}.

A simplicial complex has \emph{pure dimension} $d$ if every facet has dimension $d$. The purity needed below follows from \cite[Cor.\,(5.10.9), p.\,117]{EGA}: a catenary Noetherian local ring satisfying \Serre{2} is equidimensional. Indeed, a DVR is Cohen--Macaulay and hence universally catenary, and its finitely generated algebras and their localizations are catenary. For a facet $F$ of $\Delta$, the corresponding component of $V[\Delta]$ at its homogeneous maximal ideal has dimension $|F|+1$. The corresponding component of $V\ov{[\Delta]}$ at $\m$ has dimension $|F|$: its quotient is a polynomial ring over $V$ in $|F|-1$ variables if $v_0\in F$, and over $K$ in $|F|$ variables if $v_0\notin F$. All these minimal primes lie in the respective homogeneous maximal ideal. Hence, if either ring satisfies \Serre{k} for $k\geq2$, equidimensionality of that localization forces all facets of $\Delta$ to have the same dimension.

\begin{thm}\label{thm-tSR-terai}
Assume that $\Delta$ is a finite abstract simplicial complex of dimension $d$ and that $k\geq2$. Then $R=V\ov{[\Delta]}$ satisfies \Serre{k} if and only if $\Delta$ is pure and the following conditions hold for all faces $F$ of $\Delta$,  including the empty face:

\benn
\item For every $F\in\Delta$ with $v_0\notin F$, 
$\tH^j(\link_\Delta F;K)=0$ for all $j<\min\{k-1,d-\dim F-1\}$.  

\item For every $F \in \Delta$ with $v_0\in F$,
   \beni
   \item $\tH^j(\link_\Delta F;L)=0$ for all $j < \min\{k-1,d-\dim F-1\}$  and   
    \item  $\tH^j(\link_\Delta F;K)=0$ for $j<\min\{k-2,d-\dim F-1\}$.  
\een
\een
Since $\Delta$ is pure, $d-\dim F-1 = \dim\link_\Delta F$, and so the 
second entries in the minimums of both (i) and (ii) may be replaced by $\dim\link_\Delta F$.
\end{thm}

\begin{proof}
If $R$ satisfies \Serre{k}, then $\Delta$ is pure by the preceding argument. We may therefore work 
with a pure $d$-dimensional complex in proving both directions.
For $G\in\Delta$, set $s_G=\min\{k,\dim\link_\Delta G+1\}$. By Lemma~\ref{cKG}, the condition
$\depth\cK[\link_\Delta G]\geq s_G$ is equivalent to 
the condition that $\tH^j(\link_{\link_\Delta G}H;\cK)=0$
for every $H\in\link_\Delta G$ and $j < s_G-|H| - 1.$ Write $F=G\cup H$ and let $r := |F\sm G| = |H|$. Then $\link_{\link_\Delta G}H=\link_\Delta F$. Purity gives
\[\begin{aligned}
 s_G-r-1
 &=\min\{k,d-\dim G\}-r-1\\
 &=\min\{k-r-1,d-\dim F-1\}\\
 &=\min\{k-r-1,\dim\link_\Delta F\}.
 \end{aligned}
\]
For each fixed $F$, the largest required range is obtained from the 
smallest allowed $r$.

For the conditions involving $K$ in Theorem~\ref{SkChar}, the allowed faces $G \inc F$ omit $v_0$. 
If $v_0\notin F$, we may take $G=F$, giving $r=0$ and precisely condition (1). 

If $v_0\in F$, we must have $r\geq1$, and $G=F\sm\{v_0\}$ gives  $r=1$. The resulting range is $j < \min\{k - 2,\dim\link_\Delta F\}$,
which is condition (2)(ii). For the conditions involving $L$ in Theorem~\ref{SkChar} 
the allowed $G$ contain $v_0$, so that $F$ also contains $v_0$. Taking $G = F$ gives the smallest value $r = 0$, and the resulting range is 
condition (2)(i).

This proves necessity. Conversely, for every allowed pair $G\inc F$, the displayed bound is at most the 
corresponding bound in (1), (2)(i), or (2)(ii). 
Hence, these vanishing conditions give all the link cohomology vanishings 
required by Lemma~\ref{cKG} and, therefore, both depth conditions of Theorem~\ref{SkChar}, which proves that $R$ satisfies \Serre{k}.\end{proof}

\begin{examp} Let $\Delta$ be the triangulation of $\Prj^2_\R$ given in Example \ref{ex-betti-P2}. The left-hand table shows
\begin{minipage}[t]{0.49\textwidth} 
\begin{center}
\vspace{1.5pt}
\renewcommand{\arraystretch}{1.6}
\begin{tabular}{|c|c|c|c|c|c|}
\hline
{\boldmath $\dim\; F$} \Strut{15pt}    & {\boldmath$\link_\Delta F$} & {\boldmath $\Tilde{\Homol}^{-1}$} & {\boldmath $\Tilde{\Homol}^{0}$} & {\boldmath $\Tilde{\Homol}^{1}$} & {\boldmath $\Tilde{\Homol}^{2}$} \\ \hline
-1                                  &    $\Delta$                                                               & \cellcolor[HTML]{FFCCC9}$0$                        & \cellcolor[HTML]{FFCCC9}$0$                       & \cellcolor[HTML]{FFCCC9}$\Ann_\kappa 2$        & $\kappa / 2\kappa$                                   \\ \hline
\rbx{4pt}{0}      \Strut{25pt}                               & \rbx{23.3pt}{{\tikzset{every picture/.style={line width=0.75pt}}      

\begin{tikzpicture}[x=0.75pt,y=0.75pt,yscale=-1,xscale=1,baseline={([yshift=-0.2ex]current bounding box.north)}]

\draw  [fill={rgb, 255:red, 0; green, 0; blue, 0 }  ,fill opacity=1 ] (19.5,4.81) .. controls (19.5,3.26) and (20.84,2) .. (22.48,2) .. controls (24.13,2) and (25.46,3.26) .. (25.46,4.81) .. controls (25.46,6.36) and (24.13,7.62) .. (22.48,7.62) .. controls (20.84,7.62) and (19.5,6.36) .. (19.5,4.81) -- cycle ;

\draw   (40.19,17.02) -- (33.37,36.69) -- (11.44,36.64) -- (4.71,16.94) -- (22.48,4.81) -- cycle ;
\draw  [fill={rgb, 255:red, 0; green, 0; blue, 0 }  ,fill opacity=1 ] (37.22,17.02) .. controls (37.22,15.46) and (38.55,14.21) .. (40.19,14.21) .. controls (41.84,14.21) and (43.17,15.46) .. (43.17,17.02) .. controls (43.17,18.57) and (41.84,19.83) .. (40.19,19.83) .. controls (38.55,19.83) and (37.22,18.57) .. (37.22,17.02) -- cycle ;
\draw  [fill={rgb, 255:red, 0; green, 0; blue, 0 }  ,fill opacity=1 ] (30.39,36.69) .. controls (30.39,35.14) and (31.72,33.88) .. (33.37,33.88) .. controls (35.01,33.88) and (36.35,35.14) .. (36.35,36.69) .. controls (36.35,38.24) and (35.01,39.5) .. (33.37,39.5) .. controls (31.72,39.5) and (30.39,38.24) .. (30.39,36.69) -- cycle ;
\draw  [fill={rgb, 255:red, 0; green, 0; blue, 0 }  ,fill opacity=1 ] (8.46,36.64) .. controls (8.46,35.09) and (9.79,33.83) .. (11.44,33.83) .. controls (13.08,33.83) and (14.42,35.09) .. (14.42,36.64) .. controls (14.42,38.19) and (13.08,39.45) .. (11.44,39.45) .. controls (9.79,39.45) and (8.46,38.19) .. (8.46,36.64) -- cycle ;
\draw  [fill={rgb, 255:red, 0; green, 0; blue, 0 }  ,fill opacity=1 ] (1.73,16.94) .. controls (1.73,15.39) and (3.06,14.13) .. (4.71,14.13) .. controls (6.35,14.13) and (7.69,15.39) .. (7.69,16.94) .. controls (7.69,18.49) and (6.35,19.75) .. (4.71,19.75) .. controls (3.06,19.75) and (1.73,18.49) .. (1.73,16.94) -- cycle ;

\end{tikzpicture}}}                                                              & \cellcolor[HTML]{FFCCC9} \rbx{4pt}{$0$}                        & \cellcolor[HTML]{FFCCC9} \rbx{4pt}{$0$}               
&\rbx{4pt}{$\kappa$}                                          & \rbx{4pt}{$0$}                                               \\ \hline
1                                       &   $\bullet \quad \bullet$                                                                & \cellcolor[HTML]{FFCCC9}$0$                        & $\kappa$                                          & $0$                                               & $0$                                               \\ \hline
2                                       &   $\{\nl\}$                                                          & $\kappa$                                           & $0$                                               & $0$                                               & $0$                                               \\ \hline
\end{tabular}
\end{center}
\end{minipage}
\hfill
\begin{minipage}[t]{0.46\textwidth}
 the data of $\Tilde{\Homol}^\bullet(\link_\Delta F; \kappa)$ for some arbitrary field $\kappa$. The cells $\Tilde{\Homol}^i$ are shaded for values of $i < \dim \Delta - \dim F - 1$. 
We check whether the Serre condition \Serre{k} holds for $\tSRrcomp{\Delta}$.
The only potentially problematic entry is $\simpHomol^1(\link_\Delta \nl; \kappa)$, which is nonzero if and only if the characteristic of $\kappa$ is 2. Since $\vertex{0} \notin \nl$, we require $\simpHomol^1(\link_\Delta \nl; K) = 0$ if $1 < \min \{ k - 1, 2\}$. It follows that $\tSRrcomp{\Delta}$ always has \Serre{2}, and $\tSRrcomp{\Delta}$ has \Serre{3} 
if and only if the characteristic of $K$ is different from 2.
\end{minipage}

\end{examp}

\begin{examp}
    We generalize the previous example to the case where $\Delta$ is a homology manifold over $V$ (see \S\ref{homby}). By Remark \ref{uc}, we have for any $V$-algebra $B$ and any nonempty face $F \in \Delta$ that $\simpHomol^j(\link_\Delta F; B) = 0$ for any 
    $j \neq \dim \Delta - |F|.$ In particular, the vanishing conditions in 
    Theorem~\ref{thm-tSR-terai} are automatically satisfied for all $F \neq \nl$. It follows that when $\Delta$ is a homology manifold, $\tSRrcomp{\Delta}$ has \Serre{k} if and only if $\simpHomol^j(\Delta; K) = 0$ for all $j < \min\{k- 1, \dim \Delta\}$, a condition which depends only on the geometric realization $|\Delta|$.
\end{examp}

\quad\bigskip

$\begin{array}{ll}
\qquad\textrm{Department of Mathematics}&\qquad\qquad\qquad \textrm{Department of Mathematics and Statistics}\\
\qquad\textrm{University of Michigan}   &\qquad\qquad\qquad \textrm{University of New Mexico}\\
\qquad\textrm{Ann Arbor, MI 48109--1043}      &\qquad\qquad\qquad \textrm{Albuquerque, New Mexico 87131}\\
\qquad\textrm{USA}                            &\qquad\qquad\qquad \textrm{USA}\\
\qquad\quad & \quad\\
\qquad\textrm{E-mail: hochster@umich.edu}     &\qquad\qquad\qquad \textrm{E-mail: ostrahan@unm.edu}\\ 
     
\end{array}$

\end{document}